\documentclass[11pt]{amsart}

\usepackage[margin=1.25in]{geometry}
\usepackage[T1]{fontenc}
\usepackage{lmodern}
\usepackage{microtype}

\usepackage{amsmath,amssymb,amsthm,mathtools}
\numberwithin{equation}{section}

\usepackage{graphicx}
\usepackage{tikz}
\usetikzlibrary{arrows.meta,positioning,calc}

\usepackage{array,longtable}
\usepackage{enumitem}

\usepackage[hidelinks]{hyperref}
\hypersetup{
pdfauthor={Yin Cai, Guozheng Cheng, Xiang Fang, Menghan Li,
Hongdou Qu, Chengbo Xiao},
pdftitle={Exact Fourier dimensions of dyadic Mandelbrot cascades
under minimal integrability}
}

\newtheorem{theorem}{Theorem}[section]
\newtheorem{proposition}[theorem]{Proposition}
\newtheorem{lemma}[theorem]{Lemma}
\newtheorem{corollary}[theorem]{Corollary}

\theoremstyle{definition}
\newtheorem{definition}[theorem]{Definition}
\newtheorem{example}[theorem]{Example}

\theoremstyle{remark}
\newtheorem{remark}[theorem]{Remark}

\newcommand{\E}{\mathbb{E}}
\newcommand{\Pbb}{\mathbb{P}}
\newcommand{\R}{\mathbb{R}}
\newcommand{\Z}{\mathbb{Z}}

\newcommand{\calE}{\mathcal{E}}

\newcommand{\wh}{\widehat}
\newcommand{\weak}{\xrightarrow{\mathrm{w}}}

\newcommand{\dimF}{\dim_{\mathrm F}}
\newcommand{\dimE}{\dim_{\mathrm E}}
\newcommand{\dimH}{\dim_{\mathrm H}}
\newcommand{\dimtwo}{\dim_2}

\newcommand{\Aloc}{A_{\mathrm{loc}}}
\newcommand{\alphamin}{\alpha_{\min}}

\newcommand{\one}{\mathbf 1}
\newcommand{\spt}{\operatorname{spt}}
\newcommand{\Var}{\operatorname{Var}}

\title[Fourier dimensions of Mandelbrot cascades]
{Exact Fourier dimensions of dyadic Mandelbrot cascades under minimal integrability}

\author[Y. Cai]{Yin Cai}
\address{School of Mathematics\\
Hangzhou Normal University\\
Hangzhou 311121\\
P. R. China}
\email{20253036@hznu.edu.cn}

\author[G. Cheng]{Guozheng Cheng}
\address{School of Mathematical Sciences\\
Dalian University of Technology\\
Dalian 116024\\
P. R. China}
\email{gzhcheng@dlut.edu.cn}

\author[X. Fang]{Xiang Fang}
\address{Department of Applied Mathematics\\
National Yang Ming Chiao Tung University\\
Hsinchu 30010\\
Taiwan}
\email{xfang@nycu.edu.tw}

\author[M. Li]{Menghan Li}
\address{School of Mathematical Sciences\\
Dalian University of Technology\\
Dalian 116024\\
P. R. China}
\email{lmh2025@mail.dlut.edu.cn}

\author[H. Qu]{Hongdou Qu}
\address{School of Mathematical Sciences\\
Dalian University of Technology\\
Dalian 116024\\
P. R. China}
\email{quhongdou0926@mail.dlut.edu.cn}

\author[C. Xiao]{Chengbo Xiao}
\address{School of Mathematical Sciences\\
Dalian University of Technology\\
Dalian 116024\\
P. R. China}
\email{xiaochengbo@mail.dlut.edu.cn}



\date{}

\begin{document}

\begin{abstract}
We determine the Fourier dimension of dyadic Mandelbrot cascades under the
minimal Kahane--Peyri\`ere integrability condition.  The interval theorem is
proved in a vector-valued dyadic cascade model: sibling weights may have
arbitrary dependence.  For every balanced energy-admissible vector law, almost
surely on non-extinction,
\begin{equation*}
    \dim_{\mathrm F}(\mu)=\dim_{\mathrm E}(\mu)=\dim_2(\mu)=D_{\mathrm E}(X).
\end{equation*}
In the canonical scalar case, under
\begin{equation*}
    W\ge 0,\qquad \mathbb E W=1,\qquad
    \mathbb E[W\log_2^+W]<\infty,\qquad
    \mathbb E[W\log_2 W]<1 ,
\end{equation*}
the formula becomes
\begin{equation*}
    \dim_{\mathrm F}(\mu)=\dim_{\mathrm E}(\mu)=\dim_2(\mu)
    =
    \sup_{1<q<2}
    \max\left\{
        0,\,
        2-\frac{2}{q}\bigl(1+\log_2\mathbb E[W^q]\bigr)
    \right\},
\end{equation*}
with the convention that the corresponding term is zero when
\(\mathbb E[W^q]=\infty\).  In particular, this scalar specialization gives the
canonical Mandelbrot--Kahane Fourier-dimension formula under the minimal
integrability condition.

We also prove the endpoint theorem for the dyadic Mandelbrot cascade on the unit
circle.  Under the same scalar assumption, almost surely on non-extinction,
\begin{equation*}
    \dim_{\mathrm F}(\mu_\circ)
    =
    \sup_{q>1}
    \max\left\{
        0,\,
        \frac{q-1-\log_2\mathbb E[W^q]}{q}
    \right\}.
\end{equation*}
The interval and circle formulas share a light-tail/heavy-tail dichotomy
but have different mechanisms: energy dimension for the interval, and
minimum lower local dimension for the circle.  The circle lower bound follows
from a finite-moment annular Fourier theorem.
\end{abstract}

\maketitle

\tableofcontents

\section{Introduction and main results}
\label{sec:introduction-main-results}

The theory of Mandelbrot cascades originates in the study of
\emph{intermittent turbulence}.  Motivated by refinements of Kolmogorov's
phenomenology due to Kolmogorov, Landau--Lifshitz, Obukhov, and Yaglom
\cite{Kolmogorov1962,LandauLifshitz1959,Obukhov1962,Yaglom1966}, Mandelbrot
introduced multiplicative random models for the turbulent energy dissipation
field \cite{Mandelbrot1972,Mandelbrot1974a,Mandelbrot1974b,Mandelbrot1976}.
In its canonical form, the construction produces a random cascade measure by
iteratively redistributing mass according to independent copies of a
nonnegative random weight.  The weight encodes the intermittent transfer of
mass from one scale to the next, and the resulting martingale structure yields a
tractable model for random measures with highly nonuniform local scaling.

Kahane and Peyri\`ere \cite{KahanePeyriere1976}, and subsequently Kahane's
martingale framework \cite{Kahane1987PositiveMartingales}, placed
Mandelbrot's construction on a rigorous probabilistic foundation.  The
resulting measures form a basic class of random multifractals and a central
model in multiplicative chaos \cite{Kahane1985Chaos,RhodesVargas2014}.  They
have also served as a testing ground for branching random walks
\cite{Biggins1977,Biggins1977Martingale}, random measures, and multifractal
analysis.  We refer to
\cite{Barral1999,Barral2000,ColletKoukiou1992,Heurteaux2016,
Kahane1985RandomSeries,Kahane1993,Liu2000,RhodesVargas2014,
WaymireWilliams1994} for classical developments and related perspectives.

The harmonic-analytic side of the theory has proved more elusive.  Exact
Fourier-dimension formulas are rare for multifractal measures.  When available,
they connect random geometry with restriction phenomena, Salem-type behavior,
and fine regularity properties of measures; see also
\cite{Bluhm1999,Salem1951}.  In this context, the canonical Mandelbrot cascade
is a fundamental test case for the principle that random multifractal structure
can force Fourier decay.  The problem of determining its Fourier dimension was
raised by Mandelbrot in the turbulence setting \cite{Mandelbrot1976}, and was
later placed by Kahane within a broader program on Fourier decay for natural
random measures \cite{Kahane1993}.

For a long time, the exact Fourier behavior of Mandelbrot cascades remained
largely open.  Before the recent breakthroughs
\cite{ChenHanQiuWang2025Harmonic,ChenLiSuomala2025,RyouSuomala2026}, the only complete result
of this type in a genuinely multiplicative setting was the fractal percolation
theorem of Shmerkin and Suomala, proved through the theory of spatially
independent martingales \cite{ShmerkinSuomala2018}.  Parallel progress was made
for related models from Gaussian multiplicative chaos: Falconer and Jin obtained
lower bounds for the Fourier dimension of planar Gaussian multiplicative chaos
\cite{FalconerJin2019}, while Garban and Vargas proved Rajchman and polynomial
decay phenomena for Gaussian multiplicative chaos on the circle
\cite{GarbanVargas2023}.  Taken together, these developments placed the
Mandelbrot--Kahane problem among the basic open problems in the Fourier analysis
of random multiplicative measures.

Recent work has changed this picture substantially.  Two contemporaneous lines
of work are especially relevant.  In one direction, Chen, Li, and Suomala
determined the Fourier dimension of the Mandelbrot cascade on the line under a
subexponential tail assumption, and also obtained a Fourier-dimension lower
bound for the corresponding cascade on the unit circle
\cite{ChenLiSuomala2025}.  The latter gives an early result of this type for a
curved support with nonvanishing curvature.  In another direction, Chen, Han, Qiu, and Wang developed a different approach and computed the Fourier dimension
under the all-positive-moment assumption
\[
    \mathbb{E}(W^p)<\infty
    \qquad\text{for every }0<p<\infty,
\]
together with a nonlattice hypothesis
\cite{ChenHanQiuWang2025Harmonic,ChenHanQiuWang2025MandelbrotKahane}.
Subsequently, Lin, Qiu, and Tan presented a unified Fourier-dimension theory
for classical multiplicative chaos \cite{LinQiuTan2025}; in particular, their
framework is announced to reduce the all-moment hypothesis to the second-moment
condition
\[
    \mathbb{E}(W^2)<\infty .
\]
More recently, Ryou and Suomala obtained Fourier-dimension results, under
all-moment assumptions, for cascades supported on curves with nonvanishing
curvature \cite{RyouSuomala2026}.

A common feature of these results is that they rely on tail or moment
assumptions strong enough to support concentration estimates for the relevant
random increments.  In particular, the existing methods do not reach the
infinite-variance regime
\[
    \mathbb{E}(W^2)=\infty,
\]
which is still allowed by the minimal Kahane--Peyri\`ere condition.

The present paper removes this obstruction and provides a complete picture on both
the line and the circle.  The main contributions may be summarized as follows.

\begin{enumerate}[label=\textup{(\roman*)},leftmargin=2.4em]

\item We identify the light-tail regime through the equivalence
\begin{equation}
\label{E:light-tail}\mathbb{E}(W^{1+\varepsilon})<\infty\ \text{ for some } \varepsilon>0\quad\Longleftrightarrow\quad\exists q>1:\ \kappa(q)<1,
\end{equation}
where 
\(
\kappa(q):=2^{1-q}\mathbb{E}(W^q).
\)
The Fourier lower bound is then obtained from a random smoothing argument
governed by the contraction ratio \(\kappa(q)\).  This yields the exact
Fourier-dimension formula throughout the light-tail regime.

\item In the complementary heavy-tail regime,
\begin{equation}
\label{E:heavy-tail}
\mathbb{E}(W^{1+\varepsilon})=\infty
\quad\text{for every } \varepsilon>0,
\end{equation}
the concentration methods used in the earlier works are no longer available.
We instead pass through the energy dimension.  We prove a sharp dichotomy for
the energy dimension and show that the heavy-tail regime is precisely the case
in which this dimension vanishes.  The deterministic inequality
\(\dim_{\mathrm F}\leq \dim_{\mathrm E}\) then forces zero Fourier dimension and
completes the interval case.

\item On the circle, the same light-tail/heavy-tail dichotomy persists, but
through a different mechanism.  The energy dimension is no longer the correct
benchmark.  The curvature of the support creates new analytic obstacles, and
the random smoothing argument does not carry over.  The main replacement is a
finite-\(r\) annular theorem, Theorem~\ref{thm:finite-r-annular}, which provides
Fourier control on curved annuli and supplies the principal analytic input for
the circle lower bound.

\item The interval theorem is proved in a vector-valued form: sibling weights
may have arbitrary dependence.  More precisely, we prove the same
Fourier-dimension formula for a vector-valued cascade in which the two sibling
weights are not assumed independent; see
Theorem~\ref{thm:main-vector-exact-fourier-dimension}.  Thus the canonical
independent Mandelbrot cascade is a special case of a more robust vector-weight
theory.

\end{enumerate}

\subsection{The interval vector theorem}
\label{subsec:intro-interval-vector-theorem}

We begin with the dyadic vector-weight cascade on \([0,1]\).  Let
\(X=(X_0,X_1)\) be a random vector in \([0,\infty)^2\).  At each vertex \(u\)
of the rooted binary tree attach an independent copy
\[X(u)=(X_0(u),X_1(u))\] of \(X\).  The coordinates \(X_0(u)\) and \(X_1(u)\)
may be dependent; independence is assumed only between different vertices.

If \(u=(u_1,\ldots,u_n)\in\{0,1\}^n\), write
\(u|k=(u_1,\ldots,u_k)\), with \(u|0=\varnothing\).  For \(1\le k\le n\),
\[X_{u_k}(u|k-1)\] denotes the \(u_k\)-th coordinate of the random vector
attached to the parent \(u|k-1\).  The path weight is
\begin{equation}\label{eq:vector-path-weight-definition-1}
L_u
=
\prod_{k=1}^n X_{u_k}(u|k-1),
\qquad
L_{\varnothing}=1.
\end{equation}
Let \(I_u\subset[0,1]\) be the dyadic interval corresponding to \(u\).  The level-\(n\) cascade measure is
\begin{equation}\label{eq:intro-vector-level-measure}
d\mu_n(x)=\sum_{\lvert u\rvert=n}2^nL_u\one_{I_u}(x)\,dx.
\end{equation}
Its total mass is \(M_n=\mu_n([0,1])=\sum_{\lvert u \rvert=n}L_u\).  The mean-one condition
\[\E(X_0+X_1)=1\] makes \((M_n)_{n\ge0}\) a nonnegative martingale.  The
stronger balanced condition \(\E X_0=\E X_1=1/2\) is the geometric centering
condition: it implies \(\E\mu_n=\mathcal L\) for every \(n\), and is used in the
Fourier lower bound.

The following is the vectorial version of the minimal Kahane--Peyri\`ere integrability condition.

\begin{definition}\label{def:energy-admissible-vector-law}
The dyadic vector-weight law is called \emph{energy-admissible} if
\begin{equation}\label{eq:energy-admissible-balance}
\E X_0=\E X_1=\frac12,
\end{equation}
\begin{equation}\label{eq:energy-admissible-positive-entropy}
\E\bigl[X_0\log_2^+X_0+X_1\log_2^+X_1\bigr]<\infty,
\end{equation}
and
\begin{equation}\label{eq:energy-admissible-negative-drift}
\E\bigl[X_0\log_2 X_0+X_1\log_2 X_1\bigr]<0,
\end{equation}
where \(0\log_2 0=0\).
\end{definition}

By Theorem~\ref{thm:vector-limiting-measure-construction-paper}, under the
mean-one condition \(\mu_n\) converges weakly almost surely to a finite random
Borel measure \(\mu\), with
\(\mu([0,1])=M:=\lim_{n\to\infty}M_n\).  By
Corollary~\ref{cor:energy-admissibility-nontriviality-paper},
energy-admissibility leads to \(\Pbb(M>0)>0\).  All interval-cascade dimension
statements are made on \(\{M>0\}\).

The parameter in the interval theorem is defined from the vector \(q\)-mass profile.

\begin{definition}\label{def:vector-energy-parameter}
For \(q>0\), define the vector moment profile by
\begin{equation}\label{eq:vector-rho-definition}
\rho(q)=\E[X_0^q+X_1^q]\in[0,\infty].
\end{equation}
The associated energy parameter is 
\begin{equation}\label{eq:vector-DE-definition}
D_E(X)
=
\sup\left\{
s\in(0,1):
\inf_{0<q<2}2^{qs/2}\rho(q)<1
\right\},
\end{equation}
with the convention that the supremum of the empty set is \(0\).
\end{definition}

Equivalently, as proved in Section~\ref{sec:vector-cascades-energy-profile},
\begin{equation}\label{eq:intro-DE-variational}
D_E(X)
=
\sup_{\substack{1<q<2\\ \rho(q)<1}}
\left(-\frac{2}{q}\log_2\rho(q)\right),
\end{equation}
again with the empty supremum interpreted as \(0\).  This is the vector analogue of the usual scalar dyadic exponent.

The first main result implies the exact dimension formula on the interval.

\begin{theorem}
\label{thm:main-vector-exact-fourier-dimension}
Assume that the dyadic vector-weight law is energy-admissible. Let \(\mu\) be the limiting dyadic vector-weight cascade measure on \([0,1]\), and set \(M=\mu([0,1])\). Then, almost surely on \(\{M>0\}\),
\begin{equation}\label{eq:main-vector-exact-formula}
\dimF(\mu)=\dimE(\mu)=\dimtwo(\mu)=D_E(X).
\end{equation}
In particular, if \(D_E(X)>0\), then for every \(0<\sigma<D_E(X)\)
there is a finite random constant \(C_\sigma\) such that
\begin{equation}\label{eq:main-vector-decay}
\lvert\widehat{\mu}(\xi)\rvert
\le
C_\sigma \lvert\xi\rvert^{-\sigma/2}
\end{equation}
for all sufficiently large 
\(\lvert\xi\rvert\).
\end{theorem}

The proof of Theorem~\ref{thm:main-vector-exact-fourier-dimension} has two main parts. The energy part proves that, almost surely on \(\{M>0\}\),
\[
\dimE(\mu)=\dimtwo(\mu)=D_E(X).
\]
It uses dyadic square sums, a subcritical fractional-moment estimate, a no-plateau lemma for the moment profile, a finite-block weighted-growth obstruction in the supercritical regime, and an alive-tree amplification argument.  The Fourier part proves that, almost surely on \(\{M>0\}\),
\[
\dimF(\mu)\ge D_E(X).
\]
It uses the balanced centering \(\E\mu_n=\mathcal L\), a vector-valued \(\ell^r\) contraction, a mesoscopic Fourier decomposition, and a ladder iteration on dense frequency grids.  The deterministic inequality
\[
\dimF(\nu)\le\dimE(\nu)
\]
for finite Borel measures then yields the reverse Fourier inequality.

The scalar dyadic Mandelbrot cascade is recovered by taking independent copies
\(W_0,W_1\) of a scalar weight \(W\), with \(\E W=1\), and setting
\(X_i=W_i/2\), \(i=0,1\).  Then
\(\rho(q)=2^{1-q}\E[W^q]\), and scalar energy-admissibility is exactly the
minimal Kahane--Peyri\`ere condition.  Hence the vector theorem gives, almost
surely on the non-extinction event \(\{M>0\}\),
\[
\dimF(\mu)=\dimE(\mu)=\dimtwo(\mu)
=
\sup_{1<q<2}
\max\left\{
0,\,
2-\frac2q\bigl(1+\log_2\E[W^q]\bigr)
\right\},
\]
where the contribution of a given \(q\) is taken to be \(0\) whenever
\(\E[W^q]=\infty\).  Thus, in the canonical scalar setting, the theorem settles
the Mandelbrot--Kahane Fourier-dimension problem under the minimal
Kahane--Peyri\`ere integrability condition.  This scalar formula is obtained as
a specialization of the vector-valued dyadic theorem, which allows arbitrary
dependence between sibling weights.

The \(b\)-adic and higher-dimensional canonical versions are treated separately
in the companion paper  \cite{Fang2026}.  For cascades on \([0,1]^d\subset\mathbb{R}^d\), the formula
takes the form
\[
    \dim_{\mathrm F}(\mu)=\min\{2,\dim_{\mathrm E}(\mu)\}
\]
almost surely on non-extinction. 
The reduction from the dyadic interval case is largely formal, except for the
additional ambient upper bound \(\dim_{\mathrm F}(\mu)\leq 2\), which is proved
in \cite{Fang2026} through a minimal-integrability version
of the Chen--Li--Suomala upper-bound argument.
We
work here with the dyadic interval model in order to keep the main probabilistic
and Fourier-recursive mechanisms in their simplest form.

\subsection{The circle endpoint theorem}
\label{subsec:intro-circle-endpoint-theorem}

The second model is a scalar dyadic cascade on
\[\mathbb S^1=\{x\in\R^2:\lvert x \rvert=1\}.\]  Let
\(f:[0,1)\to\mathbb S^1\) be given by
\(f(t)=(\cos 2\pi t,\sin 2\pi t)\), and use \(f\) only to impose the dyadic
filtration on the circle.  Let \(W\ge0\), \(\E W=1\), and attach independent copies \(W_v\) of \(W\) to
non-empty binary words \(v\).  If
\(Q_v=\prod_{j=1}^{\lvert v \rvert}W_{v|j}\), then the level-\(n\) circle cascade is
\[
d\mu_{\circ,n}(x)=\sum_{\lvert v \rvert=n}Q_v\one_{\mathcal I_v}(x)\,d\sigma(x),
\]
where \(\sigma\) is normalized arclength and
\(\mathcal I_v=f(J_v)\) is the dyadic arc associated to \(v\).  The limiting
circle cascade is denoted by \(\mu_\circ\).

\begin{definition}[Minimal Kahane--Peyri\`ere regime]
\label{def:minimal-KP-circle}
We say that \(W\) is in the minimal Kahane--Peyri\`ere regime if
\begin{equation}\label{eq:minimal-KP-circle}
W\ge0,
\qquad
\E W=1,
\qquad
\E[W\log_2^+ W]<\infty,
\qquad
\E[W\log_2 W]<1.
\end{equation}
\end{definition}

Under this condition the circle cascade is nondegenerate in the
Kahane--Peyri\`ere sense:
\(\Pbb(\mu_\circ(\mathbb S^1)>0)>0\).  The event
\(\{\mu_\circ(\mathbb S^1)>0\}\) is the non-extinction event.

The endpoint exponent is a local-mass exponent, not an interval energy exponent.

\begin{definition}[Endpoint local exponent]
\label{def:circle-endpoint-exponent}
Define
\begin{equation}\label{eq:circle-Aloc-definition}
\Aloc(W)
=
\sup_{q>1}
\max\left\{
0,\,
\frac{q-1-\log_2\E[W^q]}{q}
\right\},
\end{equation}
where the term corresponding to \(q\) is interpreted as \(0\) whenever \(\E[W^q]=\infty\).
\end{definition}

\begin{remark}
\label{rem:scalar-interval-circle-exponent-comparison}
In the scalar interval specialization, the corresponding interval exponent is
\[
D^+(W)
=
\sup_{1<q<2}
\max\left\{
0,\,
2-\frac2q\bigl(1+\log_2\E[W^q]\bigr)
\right\}.
\]
Define
\[
A_{<2}(W)
=
\sup_{1<q<2}
\max\left\{
0,\,
\frac{q-1-\log_2\E[W^q]}{q}
\right\},
\]
with the convention in both displays that the corresponding term is \(0\) when
\(\E[W^q]=\infty\).  Then
\[
D^+(W)=2A_{<2}(W)
\qquad \text{and} \qquad 
A_{<2}(W)\le \Aloc(W),
\]
since \(\Aloc(W)\) optimizes over all \(q>1\).  Thus the same scalar moment
expression appears in the interval and circle formulas, but with different
normalizations and different geometric obstructions.  The factor \(2\) in the
interval formula reflects the Fourier-dimension convention relative to the
one-dimensional energy exponent, whereas the circle formula is governed by the
minimum lower local dimension on a curved support.
\end{remark}

The second main result establishes the endpoint Fourier-dimension formula on the circle, the prototypical curved-support case with nonzero curvature.

\begin{theorem}
\label{thm:main-circle-endpoint-formula}
Assume that \(W\) is in the minimal Kahane--Peyri\`ere regime. Let \(\mu_{\circ}\) be the dyadic Mandelbrot cascade measure on \(\mathbb{S}^{1}\) generated by \(W\). Then, almost surely on \(\{\mu_{\circ}(\mathbb{S}^{1})>0\}\),
\begin{equation}\label{eq:main-circle-endpoint-formula}
\dimF(\mu_{\circ})=\Aloc(W).
\end{equation}
In particular, if \(\Aloc(W)>0\), then for every
\(0<\sigma<\Aloc(W)\) there is a finite random constant \(C_{\sigma}\)
such that
\begin{equation}\label{eq:main-circle-decay}
\lvert\widehat{\mu_{\circ}}(\xi)\rvert
\leq
C_{\sigma}\lvert\xi\rvert^{-\sigma/2}
\end{equation}
for all sufficiently large \(\lvert\xi\rvert\), with
\(\xi\in\mathbb{R}^{2}\).
\end{theorem}

We emphasize that the circle theorem is proved only for the canonical
Mandelbrot cascade.  At present our method does not extend to the vector-valued
cascade setting of Theorem~\ref{thm:main-vector-exact-fourier-dimension}.
The extension from the circle to fixed \(C^2\) planar Jordan curves with
nonvanishing curvature is carried out in the companion paper \cite{CF-curve}.
We keep the present paper focused on the circle, where the curvature mechanism
appears in its simplest form.

The upper bound in Theorem~\ref{thm:main-circle-endpoint-formula} is an endpoint
obstruction.  For a finite Borel measure \(\nu\) on \(\mathbb S^1\), define the minimum lower local dimension
\[
\alphamin(\nu)
=
\inf_{x\in\operatorname{spt}\nu}
\liminf_{r\downarrow0}
\frac{\log_2\nu(B(x,r))}{\log_2 r}.
\]
The local-dimension theorem establishes 
\[\alphamin(\mu_\circ)=\Aloc(W)\]
on
non-extinction, while the deterministic curved-support estimate yields
\(\dimF(\nu)\le\alphamin(\nu)\) for every nonzero finite measure supported on
\(\mathbb S^1\).  Hence \(\dimF(\mu_\circ)\le\Aloc(W)\).

The lower bound is the main Fourier-analytic part of the circle theorem.  It
follows from the finite-\(r\) annular theorem stated below.  If
\(s<\Aloc(W)\), then some \(r>1\) witnesses \(s\), in the sense that
\(2^{1-r}\E[W^r]\le 2^{-r(s+\delta)}\) for some \(\delta>0\).  The annular
theorem then yields almost sure Fourier decay at exponent \(s\).  Letting
\(s\uparrow\Aloc(W)\) through a countable sequence yields
\(
\dimF(\mu_\circ)\ge \Aloc(W)\).

\subsection{The finite-moment annular theorem}
\label{subsec:intro-finite-r-annular-theorem}

The finite-moment annular theorem, proved in
Section~\ref{sec:finite-r-annular-theorem}, is the main input for the circle
lower bound.  It is formulated for an auxiliary scalar weight \(U\), since in
the endpoint argument \(U\) will be chosen as a finite-moment witness for a
strict subendpoint exponent \(s<\Aloc(W)\).  This separates the annular Fourier
estimate from the endpoint optimization and makes the role of the finite moment
transparent.

\begin{definition}\label{def:finite-r-witnessed-hypothesis}
Let \(0<s<1\).  A nonnegative random variable \(U\) with \(\E U=1\) satisfies
the finite-\(r\) witnessed hypothesis at exponent \(s\) if there exist \(r>1\)
and \(\delta>0\) such that
\begin{equation}\label{eq:finite-r-witnessed-gap}
\E[U^r]<\infty,
\qquad
2^{1-r}\E[U^r]\le 2^{-r(s+\delta)}.
\end{equation}
\end{definition}

The strict gap \(\delta>0\) is used only inside the annular estimate; it provides
the surplus needed for the local-mass bounds, predictable capping, and
compensator estimates.  The endpoint theorem is recovered afterwards by applying
the result to a countable sequence of exponents \(s\uparrow\Aloc(W)\).

Let \(\widetilde{\nu}\) be the dyadic scalar cascade on \([0,1)\) generated by
\(U\), and let \(\nu_\circ=f_\#\widetilde{\nu}\) be its pushforward to the
circle.

\begin{theorem}[Finite-\(r\) annular theorem]
\label{thm:main-finite-r-annular}
Let \(0<s<1\), \(r>1\), and \(\delta>0\).  Let \(U\ge0\) satisfy \(\E U=1\) and \eqref{eq:finite-r-witnessed-gap}.  Then there exist constants \(C,c,\eta>0\), depending only on \(s,r,\delta\) and the law of \(U\), such that, for every \(n\ge1\),
\begin{equation}\label{eq:finite-r-annular-estimate}
\Pbb\left(
\sup_{2^n\le \lvert\xi\rvert\le 2^{n+1}}
\lvert\wh{\nu_\circ}(\xi)\rvert>C2^{-sn/2}
\right)
\le
C\exp(-c2^{\eta n})+C2^{-cn}.
\end{equation}
Consequently,
\begin{equation}\label{eq:finite-r-annular-as-decay}
\lvert\wh{\nu_\circ}(\xi)\rvert=O(\lvert\xi\rvert^{-s/2})
\qquad (\lvert\xi\rvert\to\infty)
\end{equation}
almost surely.
\end{theorem}

The two terms in \eqref{eq:finite-r-annular-estimate} come from different parts of the argument. For the almost sure consequence, only the summability of the right-hand side is used.

We emphasize two features of the theorem.  First, \(U\) is not assumed bounded.  The proof uses predictable capping before applying martingale concentration, and the cost of the cap is paid by an \(r\)-tail compensator.  Second, the conclusion is annular and summable in \(n\).  This is stronger than a pointwise-in-frequency estimate and is exactly what is needed for almost sure Fourier decay.

The proof decomposes the phase
\(
t\mapsto \xi\cdot f(t)
\)
into a stationary tube and dyadic derivative bands.  The stationary tube and very small derivative bands are controlled by local mass estimates.  On the remaining bands the Fourier integral is written as a sum of exact dyadic martingale arrays.  The pre-bin part gains from oscillation before the natural phase-bin scale, while the post-bin part is controlled by the local-mass decay after that scale.  Predictable capping makes Freedman's inequality applicable without boundedness assumptions.  The resulting martingale estimates have stretched-exponential tails, and the \(r\)-tail compensator is summably small uniformly over the annulus.

For the extension from the circle to general curves with nonvanishing curvature,
the main additional issue is deterministic.  The explicit circular phase must be
replaced by a fixed-curve phase-geometry package: stationary tubes, derivative
bands, and phase-bin coefficient estimates uniform over frequency annuli.  Once
these inputs are available, the probabilistic part of the annular argument is
largely unchanged. 
This extension is carried out in  \cite{CF-curve}.

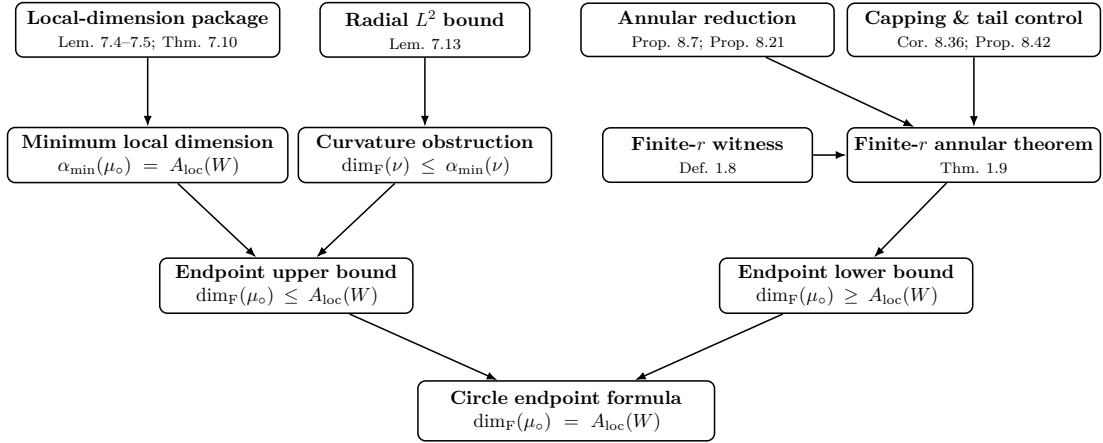
\begin{figure}[htbp]
\centering
\resizebox{0.95\textwidth}{!}{%
\begin{tikzpicture}[
  font=\small,
  main/.style={
    draw,
    rounded corners,
    thick,
    align=center,
    text width=5.2cm,
    minimum height=1.1cm
  },
  branch/.style={
    draw,
    rounded corners,
    thick,
    align=center,
    text width=4.4cm,
    minimum height=1.0cm
  },
  branch1/.style={
    draw,
    rounded corners,
    thick,
    align=center,
    text width=4.6cm,
    minimum height=1.0cm
  },
  branch2/.style={
    draw,
    rounded corners,
    thick,
    align=center,
    text width=3.6cm,
    minimum height=1.0cm
  },
  branch3/.style={
    draw,
    rounded corners,
    thick,
    align=center,
    text width=4.9cm,
    minimum height=1.0cm
  },
  tool/.style={
    draw,
    rounded corners,
    thick,
    align=center,
    text width=4.4cm,
    minimum height=0.95cm
  },
  arrow/.style={
    -{Latex[length=2mm]},
    thick
  }
]

\node[branch1] (lefttool) at (-7.7,0) {
\textbf{Local-dimension package}\\
{\scriptsize Lem.~\ref{lem:circle-uniform-cylinder-upper-bound-paper}--\ref{lem:circle-cylinder-to-local-dimension-paper};
Thm.~\ref{thm:exact-dimension-support-mass-witnesses}}};

\node[branch2] (radial) at (-2.6,0) {
\textbf{Radial \(L^2\) bound}\\
{\scriptsize Lem.~\ref{lem:radial-L2-local-mass-paper}}};

\node[tool] (decomp) at (2.6,0) {
\textbf{Annular reduction}\\
{\scriptsize Prop.~\ref{prop:local-mass-good-events-paper};
Prop.~\ref{prop:stationary-tube-phase-bin-reduction-paper}}};

\node[tool] (martingale) at (7.5,0) {
\textbf{Capping \& tail control}\\
{\scriptsize Cor.~\ref{cor:one-band-centered-capped-estimate-paper};
Prop.~\ref{prop:r-tail-compensator-paper}}};

\node[branch3] (localdim) at (-7.7,-2.3) {
\textbf{Minimum local dimension}\\
\(\alpha_{\min}(\mu_\circ)=A_{\mathrm{loc}}(W)\)};

\node[branch] (obstruction) at (-2.6,-2.3) {
\textbf{Curvature obstruction}\\
\(\dim_{\mathrm F}(\nu)\le \alpha_{\min}(\nu)\)};

\node[branch2] (witness) at (2.6,-2.3) {
\textbf{Finite-\(r\) witness}\\
{\scriptsize Def.~\ref{def:finite-r-witnessed-hypothesis}}};

\node[branch] (annular) at (7.5,-2.3) {
\textbf{Finite-\(r\) annular theorem}\\
{\scriptsize Thm.~\ref{thm:main-finite-r-annular}}};

\node[branch] (upper) at (-5.15,-4.7) {
\textbf{Endpoint upper bound}\\
\(\dim_{\mathrm F}(\mu_\circ)\le A_{\mathrm{loc}}(W)\)};

\node[branch] (lower) at (5.15,-4.7) {
\textbf{Endpoint lower bound}\\
\(\dim_{\mathrm F}(\mu_\circ)\ge A_{\mathrm{loc}}(W)\)};

\node[main] (main) at (0,-7.0) {
\textbf{Circle endpoint formula}\\
\(\dim_{\mathrm F}(\mu_\circ)=A_{\mathrm{loc}}(W)\)};

\draw[arrow] (lefttool) -- (localdim);
\draw[arrow] (radial) -- (obstruction);

\draw[arrow] (localdim) -- (upper);
\draw[arrow] (obstruction) -- (upper);

\draw[arrow] (decomp) -- (annular);
\draw[arrow] (martingale) -- (annular);
\draw[arrow] (witness) -- (annular);

\draw[arrow] (annular) -- (lower);

\draw[arrow] (upper) -- (main);
\draw[arrow] (lower) -- (main);

\end{tikzpicture}%
}
\caption{Proof roadmap for the circle endpoint theorem.}
\label{fig:roadmap-circle-simplified}
\end{figure}

\noindent \textbf{Organization.}
Section~\ref{sec:preliminaries-notation} fixes common notations for binary trees, dyadic intervals, circle arcs, Fourier dimension, energy dimension, and descendant cascades.

Section~\ref{sec:vector-cascades-energy-profile} constructs the balanced vector cascade on \([0,1]\), records the Kahane--Peyri\`ere nondegeneracy input, and develops the moment profile \(\rho(q)\) and the parameter \(D_E(X)\).

Section~\ref{sec:vector-square-sums-energy-theorem} proves the vector square-sum
and energy theorem.  The subcritical side is obtained from fractional-moment
decay of the normalized dyadic square sums. The supercritical side uses a no-plateau lemma, a finite-block weighted-growth argument, and alive-tree
amplification.  It proves that, almost surely on \(\{M>0\}\), one has 
\(
\dimE(\mu)=\dimtwo(\mu)=D_E(X).
\)

Section~\ref{sec:vector-dense-grid-fourier-lower-bound} proves the vector Fourier lower bound.  The centered profile is compared on dense grids; a vector-valued \(\ell^r\) contraction yields the basic recursive estimate; a ladder iteration yields \(q\)-level annular decay; and optimization over \(q\) yields
\(
\dimF(\mu)\ge D_E(X)
\).

Section~\ref{sec:exact-interval-theorem} combines the energy theorem, the Fourier lower bound, and the deterministic inequality 
\(
\dimF\le\dimE
\)
to prove Theorem~\ref{thm:main-vector-exact-fourier-dimension}.

Section~\ref{sec:circle-endpoint-obstruction} constructs the circle cascade, proves the minimum lower local dimension theorem 
\(
\alphamin(\mu_\circ)=\Aloc(W)
\), 
and proves the curved-support upper bound 
\(
\dimF(\nu)\le\alphamin(\nu)
\) 
for finite measures supported on \(\mathbb{S}^{1}\).  This yields the upper bound 
\(
\dimF(\mu_\circ)\le\Aloc(W)
\) 
on non-extinction.

Section~\ref{sec:finite-r-annular-theorem} proves Theorem~\ref{thm:main-finite-r-annular}.  It contains the local-mass good events, the stationary-tube and phase-bin reduction, the predictable capping argument, the \(r\)-tail compensator, and the final annular grid assembly.

Section~\ref{sec:endpoint-lower-bound-circle-theorem} uses the finite-\(r\) annular theorem to prove 
\(
\dimF(\mu_\circ)\ge\Aloc(W)
\), and then combines this with the endpoint obstruction to prove Theorem~\ref{thm:main-circle-endpoint-formula}.

\section{Preliminaries and notation}
\label{sec:preliminaries-notation}

This section establishes notation used throughout the paper. We use a single
rooted binary tree both for the interval cascade and for the dyadic
parametrization of the circle. The interval model is defined on \([0,1]\).
The circle model is defined on
\[
\mathbb{S}^{1}=\{x\in\mathbb{R}^{2}: \lvert x\rvert=1\},
\]
with its dyadic structure induced by the standard parametrization
\[
f(t)=(\cos 2\pi t,\sin 2\pi t),
\qquad 0\leq t<1.
\]
\subsection{Binary tree, dyadic intervals, and circle arcs}
\label{subsec:binary-tree-dyadic-intervals-circle-arcs}

For \(n\ge 0\), let \(\{0,1\}^n\) denote the set of binary words of length \(n\). The unique word of length \(0\) is denoted by \(\varnothing\). We write
\[
\mathcal T=\bigcup_{n=0}^{\infty}\{0,1\}^n
\]
for the rooted binary tree.

For \(u\in\mathcal T\), its length is denoted by \(\lvert u \rvert\). If
\[
u=(u_1,\ldots,u_n)\in\{0,1\}^n,
\]
then, for \(0\le k\le n\), we set
\[
u|k=(u_1,\ldots,u_k),
\]
with 
\(
u|0=\varnothing
\).
If 
\(
u\in\mathcal T
\)
and \(i\in\{0,1\}\), then \(ui\) denotes the word obtained by appending \(i\) to \(u\). If 
\(
u\in\{0,1\}^m
\)
and 
\(
v\in\{0,1\}^n
\),
their concatenation is denoted by 
\(
uv\in\{0,1\}^{m+n}
\).

We write \(u\preceq v\) if \(u\) is a prefix of \(v\). Equivalently, \(u\preceq v\) if and only if there exists 
\(
w\in\mathcal T
\)
such that \(v=uw\). If \(u\preceq v\) and \(u\neq v\), we write \(u\prec v\).

For 
\(
u=(u_1,\ldots,u_n)\in\{0,1\}^n
\),
define
\[
a_u=\sum_{j=1}^n u_j2^{-j},
\qquad
a_{\varnothing}=0.
\]
The dyadic interval associated with \(u\) is denoted by \(I_u\). We use the standard half-open convention on \([0,1]\): namely,
\[
I_{\varnothing}=[0,1),
\]
and, for \(\lvert u \rvert=n\ge 1\),
\[
I_u
=
\begin{cases}
\displaystyle [a_u,a_u+2^{-n}), & u\neq(1,\ldots,1),\\[4pt]
\displaystyle [a_u,1], & u=(1,\ldots,1).
\end{cases}
\]

Let
\[
\mathcal D_n=\{I_u:\lvert u \rvert=n\},
\qquad
\mathcal D=\bigcup_{n=0}^{\infty}\mathcal D_n.
\]
Then \(\mathcal D_n\) is a partition of \([0,1]\) into \(2^n\) intervals of length \(2^{-n}\), and
\[
I_u=I_{u0}\sqcup I_{u1}.
\]
The endpoint convention is immaterial for Lebesgue-density computations, but it ensures that every \(x\in[0,1]\) belongs to a unique level-\(n\) interval.

For the circle model, we use the same binary words, but reserve a separate notation for the corresponding dyadic intervals in the parameter space. Define
\(
J_{\varnothing}=[0,1)
\),
and, for \(u\in\{0,1\}^n\) with \(n\ge 1\),
\[
J_u=[a_u,a_u+2^{-n})\subset[0,1).
\]
Thus \(J_u\) is the parameter interval associated with \(u\). Although \(J_u\) is the same subset of \([0,1)\) as \(I_u\), the notation \(J_u\) will be used for the circle parametrization in order to distinguish it from the interval cascade.

The corresponding dyadic arc on the circle is
\[
\mathcal I_u=f(J_u)\subset\mathbb S^1,
\]
where
\(
f(t)=(\cos 2\pi t,\sin 2\pi t).
\)
Let \(\sigma\) denote normalized arclength measure on \(\mathbb S^1\). Then
\(
\sigma(\mathcal I_u)=2^{-\lvert u \rvert}.
\)
For \(n\ge 0\), set
\[
\mathcal D_{\circ,n}=\{\mathcal I_u:\lvert u \rvert=n\}.
\]
The family 
\(
\mathcal D_{\circ,n}
\)
is a half-open dyadic partition of \(\mathbb S^1\).

For \(x\in\mathbb S^1\), let \(\mathcal I_n(x)\) denote the unique level-\(n\) dyadic arc containing \(x\), with respect to the above half-open convention. Similarly, for \(x\in[0,1]\), let \(I_n(x)\) denote the unique level-\(n\) dyadic interval containing \(x\).

\subsection{Fourier dimension and energy dimension}
\label{subsec:fourier-energy-dimension-definitions}

We recall the notions of Fourier dimension and energy dimension for finite Borel measures, using standard terminology; see also \cite{FanLauRao2002}. If \(\nu\) is a finite Borel measure on \(\R^d\), its Fourier transform is defined by
\[
\wh\nu(\xi)=\int_{\R^d}e^{-2\pi i x\cdot\xi}\,d\nu(x),
\qquad
\xi\in\R^d.
\]
The Fourier dimension of a nonzero finite Borel measure \(\nu\) on \(\R^d\) is
\[
\dimF(\nu)
=
\sup\left\{
0\le s\le d:
\lvert\wh\nu(\xi)\rvert=O(\lvert\xi\rvert^{-s/2})
\text{ as }\lvert\xi\rvert\to\infty
\right\}.
\]
Equivalently, \(s\) is admissible in the above supremum if there exist constants \(C<\infty\) and \(R<\infty\) such that
\[
\lvert\wh\nu(\xi)\rvert\le C\lvert\xi\rvert^{-s/2}
\qquad\text{whenever } \lvert\xi\rvert\ge R.
\]
For measures on the interval we take \(d=1\). For measures on the circle we use the ambient Fourier transform in \(\R^2\), so that the frequency variable is \(\xi\in\R^2\). Since the circle itself is one-dimensional, the endpoint theorem will identify a value in \([0,1]\), although the ambient definition permits \(s\le 2\).

For a nonzero finite Borel measure \(\nu\) on \([0,1]\) and \(0<s<1\), its \(s\)-energy is
\[
I_s(\nu)
=
\iint_{[0,1]^2}\lvert x-y \rvert^{-s}\,d\nu(x)d\nu(y).
\]
The energy dimension of \(\nu\) is
\[
\dimE(\nu)
=
\sup\{0<s<1:I_s(\nu)<\infty\}.
\]
We also define the dyadic correlation dimension by
\[
\dimtwo(\nu)
=
\liminf_{n\to\infty}
\frac{\log_2 \Sigma_n(\nu)}{-n},
\]
where
\[
\Sigma_n(\nu)=\sum_{\lvert u \rvert=n}\nu(I_u)^2.
\]
In the interval part of the paper we prove that, for the cascade measure \(\mu\),
\[
\dimE(\mu)=\dimtwo(\mu)=D_E(X)
\]
on the non-extinction event.

We shall use the following deterministic comparison; for a proof, see Mattila~\cite[Theorem~3.10]{Mattila2015}.

\begin{proposition}\label{prop:fourier-energy-upper-bound-paper}
Let \(\nu\) be a nonzero finite Borel measure on \([0,1]\). Then 
\[\dimF(\nu)\le \dimE(\nu).\]
\end{proposition}

For measures supported on \(\mathbb S^1\), we use a different deterministic upper bound, namely the curved-support estimate
\[
\dimF(\nu)\le\alphamin(\nu),
\]
proved in Section~\ref{sec:circle-endpoint-obstruction}. This estimate is the upper-bound mechanism in the circle endpoint theorem.

\subsection{Probability conventions and descendant copies}
\label{subsec:probability-conventions-descendants}

All random variables are defined on a common probability space
\(
(\Omega,\mathcal F,\Pbb).
\)
Expectation is denoted by \(\E\), and conditional expectation with respect to a sub-\(\sigma\)-field \(
\mathcal G\subset\mathcal F
\) 
is denoted by
\(
\E[\cdot\mid\mathcal G].
\)

For the dyadic vector cascade, to each vertex \(u\in\mathcal T\) we attach an independent copy
\[
X(u)=(X_0(u),X_1(u))
\]
of the vector \(X=(X_0,X_1)\). The vectors assigned to distinct vertices are independent. The two coordinates \(X_0(u)\) and \(X_1(u)\) at a fixed vertex are not assumed to be independent.

The natural filtration for the vector cascade is
\[
\mathcal F^X_n
=
\sigma\{X(u):\lvert u \rvert\le n-1\},
\qquad n\ge 0,
\]
with \(\mathcal F^X_0\) trivial. Thus \(\mathcal F^X_n\) contains precisely the weights needed to form all level-\(n\) path weights.

For \(u\in\mathcal T\), the descendant environment rooted at \(u\) is
\[
\mathcal E^{(u)}
=
\{X(uv):v\in\mathcal T\}.
\]
For 
\(
v=(v_1,\ldots,v_m)\in\{0,1\}^m
\), 
define the descendant path weight by
\[
L_v^{(u)}
=
\prod_{k=1}^m X_{v_k}(u(v|k-1)),
\qquad
L_{\varnothing}^{(u)}=1.
\]
Then
\(
L_{uv}=L_uL_v^{(u)}
\).
The descendant cascade rooted at \(u\) has the same law as the original cascade and is independent of \(\mathcal F^X_{\lvert u \rvert}\). Moreover, descendant cascades rooted at distinct vertices of the same generation are conditionally independent given \(\mathcal F^X_n\).

For the scalar circle cascade, independent copies \(W_v\) of the scalar weight \(W\) are attached to non-empty binary words \(v\). The natural filtration is
\[
\mathcal F^W_n
=
\sigma\{W_v:1\le \lvert v \rvert\le n\},
\qquad n\ge 0,
\]
with \(\mathcal F^W_0\) trivial. If \(v\in\mathcal T\), the descendant weights below \(v\) are
\[
\{W_{vu}:u\in\mathcal T,\ u\neq\varnothing\}.
\]
They form an independent copy of the original scalar environment and are independent of 
\(
\mathcal F^W_{\lvert v \rvert}
\).

We use superscripts to denote descendant copies. Thus \(M^{(u)}\), \(\mu^{(u)}\), \(Y^{(v)}\), and \(\mu_\circ^{(v)}\) denote terminal masses and limiting measures constructed from the corresponding descendant environments. The precise constructions are given in the relevant sections. Throughout, a superscript \((u)\) or \((v)\) indicates that the object is constructed inside the subtree rooted at the corresponding word.

\section{Dyadic vector cascades and the energy profile}
\label{sec:vector-cascades-energy-profile}

This section settles the preliminary issues concerning the dyadic vector cascade used in
Theorem~\ref{thm:main-vector-exact-fourier-dimension}. Vector-valued
multiplicative cascades have been studied in more general settings; see, for
example, \cite{Barral2001}. In the dyadic setting considered here, let
\[
X=(X_0,X_1)\in[0,\infty)^2.
\]
The two coordinates of \(X\) are allowed to be dependent, while independent
copies of the vector \(X\) are assigned to distinct vertices of the binary tree.

\subsection{The balanced dyadic vector cascade}
\label{subsec:balanced-dyadic-vector-cascade}

Let
\[
\{X(u)=(X_0(u),X_1(u)):u\in\mathcal T\}
\]
be an independent family of copies of \(X\). For 
\(
u=(u_1,\ldots,u_n)\in\{0,1\}^n
\),
define
\begin{equation}\label{eq:vector-path-weight-definition}
L_u=\prod_{k=1}^n X_{u_k}(u|k-1),
\qquad
L_{\varnothing}=1.
\end{equation}
Equivalently,
\[
L_{ui}=L_uX_i(u),
\qquad
u\in\mathcal T,\ i\in\{0,1\}.
\]

For \(n\ge0\), define the level-\(n\) random measure on \([0,1]\) by
\begin{equation}\label{eq:vector-level-measure-definition}
d\mu_n(x)=\sum_{\lvert u \rvert=n}2^nL_u\one_{I_u}(x)\,dx.
\end{equation}
Thus
\(
\mu_n(I_u)=L_u
\ (\lvert u \rvert=n),
\)
and the total mass at level \(n\) is
\begin{equation}\label{eq:vector-total-mass-definition}
M_n=\mu_n([0,1])=\sum_{\lvert u \rvert=n}L_u.
\end{equation}

\begin{definition}
\label{def:mean-one-balanced-vector-law-paper}
The vector law is called mean-one if
\(
\E(X_0+X_1)=1.
\)
It is called balanced if 
\(
\E X_0=\E X_1=\frac12
\).
\end{definition}

We shall repeatedly use the following centering observation. It records
that, under the balanced assumption, the expected level-\(n\) density is precisely
the constant density \(1\) on \([0,1]\); thus the approximating measures are
centered at Lebesgue measure.
\begin{lemma}
\label{lem:expected-path-weights-lebesgue-centering}
For 
\(
u=(u_1,\ldots,u_n)\in\{0,1\}^n
\),
\[
\E L_u=\prod_{k=1}^n \E X_{u_k}.
\]
In particular, if the law is balanced, then
\[
\E L_u=2^{-n}
\qquad(\lvert u \rvert=n),
\qquad
\text{and} \qquad
\E\mu_n=\mathcal L
\qquad(n\ge0).
\]
\end{lemma}

\begin{lemma}
\label{lem:vector-total-mass-martingale-paper}
If the vector law is mean-one, then \((M_n)_{n\ge0}\) is a nonnegative martingale with respect to
\[
\mathcal F^X_n=\sigma\{X(u):\lvert u \rvert\le n-1\},
\]
where \(\mathcal F^X_0\) is trivial. Consequently, there exists a finite nonnegative random variable \(M\) such that
\(M_n\to M\) almost surely.
Moreover,
\(
\E M\le1.
\)
\end{lemma}

\begin{proof}
Using \(L_{ui}=L_uX_i(u)\), independence of the level-\(n\) descendants from \(\mathcal F_n^X\), and the mean-one condition \(\mathbb E(X_0+X_1)=1\), we get \(\mathbb E[M_{n+1}\mid\mathcal F_n^X]=M_n\).
\end{proof}

We next record the weak convergence of the level measures.  The proof is
included not only for completeness, but also to fix the descendant notation used
throughout the sequel.

For \(u\in\mathcal T\), set
\[
L_v^{(u)}
=
\prod_{k=1}^{\lvert v \rvert}
X_{v_k}\bigl(u(v|k-1)\bigr),
\qquad
L_{\varnothing}^{(u)}=1,
\qquad
M_m^{(u)}=\sum_{\lvert v \rvert=m}L_v^{(u)}.
\]
Then
\(
L_{uv}=L_uL_v^{(u)}
\).
Under the assumption
\(
\E(X_0+X_1)=1
\),
each process
\(
(M_m^{(u)})_{m\ge0}
\)
is a copy of the total-mass martingale built in the subtree rooted at \(u\).
These descendant martingales are independent of the environment above \(u\),
and descendant martingales rooted at distinct vertices of the same generation
are conditionally independent.  Since \(\mathcal T\) is countable, we may work
on an event of probability one on which
\(
M_m^{(u)}\to M^{(u)}
\)
for every \(u\in\mathcal T\).

\begin{theorem}[Existence of the vector cascade measure]
\label{thm:vector-limiting-measure-construction-paper}
Assume the mean-one condition 
\(
\E(X_0+X_1)=1
\).
Then there exists a finite
random Borel measure \(\mu\) on \([0,1]\) such that
\[
\mu_n\xrightarrow{\mathrm{w}}\mu
\qquad
\text{almost surely}.
\]
Moreover,
\(
\mu([0,1])=M
\)
almost surely. In particular, the events
\(\{\mu\neq0\}\) and \(\{M>0\}\) agree modulo null events.
\end{theorem}

\begin{proof}
Work on the full-probability event on which \(M_n\to M<\infty\) and
\(M_m^{(u)}\to M^{(u)}\) for every \(u\in\mathcal T\).  Then
\[
\mu_n(I_u)=L_uM_{n-|u|}^{(u)}\longrightarrow L_uM^{(u)}
\]
for every dyadic cylinder \(I_u\).  Thus the integrals of dyadic step functions
against \(\mu_n\) have limits.  Since \(\sup_n\mu_n([0,1])<\infty\), uniform
approximation of continuous functions by dyadic step functions extends these
limits to all \(g\in C([0,1])\).  The resulting linear functional is positive
and bounded by \(\|g\|_\infty\sup_n\mu_n([0,1])\).
By the Riesz representation theorem, this positive bounded functional is
integration against a finite Borel measure \(\mu\).  Hence
\(\mu_n\weak\mu\), and evaluating the convergence at the constant function
\(g\equiv1\) gives \(\mu([0,1])=M\).
\end{proof}

\begin{remark}
\label{rem:tree-cylinder-masses-versus-euclidean-intervals}
The quantities 
\(
C_u=L_uM^{(u)}
\)
are the limiting tree-cylinder masses. They are the limits of the finite-level
masses \(\mu_n(I_u)\), and they are adapted to the exact branching recursion.
If the limiting measure charges dyadic endpoints, then \(C_u\) should not be
identified with the Euclidean half-open interval mass \(\mu(I_u)\) without
accounting for endpoint contributions. The energy theorem below is formulated
in terms of the tree-cylinder square sums precisely to avoid this endpoint
bookkeeping.
\end{remark}

\subsection{Nondegeneracy and the Kahane--Peyri\`ere input}
\label{subsec:vector-nondegeneracy-KP-input}

We now record the nondegeneracy criterion in the notation of the vector cascade.
This is the dyadic vector version of the Kahane--Peyri\`ere theorem, interpreted
through the associated smoothing-transform fixed point; see also
\cite{BigginsKyprianou2005}.

For a nonnegative vector weight 
\(
X=(X_0,X_1)
\), 
define
\[
H^+(X)
=
\E\bigl[X_0\log_2^+X_0+X_1\log_2^+X_1\bigr],
\]
and
\[
\gamma(X)
=
\E\bigl[X_0\log_2 X_0+X_1\log_2 X_1\bigr],
\]
with \(0\log_2 0=0\). Since \(x\log_2 x\) has bounded negative part on
\([0,\infty)\), the quantity \(\gamma(X)\) is a finite real number whenever
\(
H^+(X)<\infty
\).

Let
\[
\kappa_X=\Pbb\bigl((X_0,X_1)\in\{0,1\}^2\bigr).
\]

\begin{theorem}\label{thm:vector-KP-nondegeneracy-paper}
Assume that 
\(
\E(X_0+X_1)=1
\), 
and let 
\(
M=\lim_n M_n
\).

\begin{enumerate}
\item If \(\kappa_X<1\), then the following are equivalent:

\begin{enumerate}
\item \(\Pbb(M>0)>0\);

\item \(\E M=1\);

\item 
\(
H^+(X)<\infty
\)
and 
\(
\gamma(X)<0
\).
\end{enumerate}

\item If 
\(
\kappa_X=1
\),
then the following boundary dichotomy holds. If
\[
X_0+X_1=1
\qquad\text{almost surely},
\]
then
\[
M_n=1\quad(n\ge0),
\qquad
M=1
\quad\text{almost surely}.
\]
If instead
\[
\Pbb(X_0+X_1=1)<1,
\]
then
\[
M=0
\qquad
\text{almost surely}.
\]
\end{enumerate}
\end{theorem}

\begin{remark}
\label{rem:AK-dyadic-identification}
This is the binary specialization of Alsmeyer--Kuhlbusch
\cite[Theorem~1.1]{AlsmeyerKuhlbusch2010}, with
\(
T=(X_0,X_1,0,\ldots)
\)
and 
\(
Z_1=X_0+X_1
\).
Since only two weights are
nonzero and \(\E Z_1=1\), their \(L\log_2L\) condition is equivalent to
\(
H^+(X)<\infty
\).
\end{remark}

\begin{corollary}
\label{cor:energy-admissibility-nontriviality-paper}
If the vector law is energy-admissible in the sense of
Definition~\ref{def:energy-admissible-vector-law}, then
\[
\E M=1
\qquad \text{and}
\qquad
\Pbb(M>0)>0.
\]
In particular, the limiting measure \(\mu\) satisfies
\(
\Pbb(\mu\neq0)>0
\).
\end{corollary}

The following example shows that the balanced condition alone does not imply
nontriviality.

\begin{example}
Let
\[
(X_0,X_1)=
\begin{cases}
(1,1),& \text{with probability }1/2,\\
(0,0),& \text{with probability }1/2.
\end{cases}
\]
Then 
\(
\E X_0=\E X_1=1/2
\),
but
\[
\E[X_0\log_2X_0+X_1\log_2X_1]=0,
\]
so the strict entropy condition fails. The surviving vertices form a critical
Galton--Watson tree with offspring distribution
\[
\Pbb(N=2)=\Pbb(N=0)=1/2,
\]
and hence become extinct almost surely. Thus the limiting cascade is trivial.
\end{example}

\subsection{The vector moment profile}
\label{subsec:vector-moment-profile}

We next record the moment profile associated with the dyadic vector cascade.
Moment profiles of this type are closely related to martingale and growth-rate
methods in branching random walks; see, for example, \cite{Biggins1992}.

\begin{lemma}
\label{lem:vector-q-mass-identity-paper}
For every \(q>0\) and every \(n\ge0\),
\begin{equation}\label{eq:q-mass-identity-paper}
\E\left[\sum_{\lvert u \rvert=n}L_u^q\right]=\rho(q)^n,
\end{equation}
where the identity is understood in the extended sense, and where
\(
\rho(q)^0=1
\).
\end{lemma}

\begin{remark}
\label{rem:q-mass-identity-extended-value}
The finite case follows by induction on \(n\) from the branching
independence; the extended-valued case follows by truncating
\(X_i\) and applying monotone convergence.
\end{remark}

We next record the elementary convexity facts needed to interpret \(D_E(X)\).

\begin{lemma}
\label{lem:rho-convexity-subcritical-interval-paper}
Assume the vector law is energy-admissible. Then the following hold.

\begin{enumerate}[label=\textup{(\roman*)}]
\item For \(0<q\le1\), one has 
\(
\rho(q)<\infty
\).

\item The function 
\(
q\mapsto \log_2\rho(q)
\)
is convex on every interval on which 
\(
\rho(q)<\infty
\).

\item One has
\[
\rho(1)=1,
\qquad
\left.\frac{d}{dq}\right|_{q=1-}\log_2\rho(q)=\gamma(X)<0.
\]

\item For every \(0<q<1\), one has \(\rho(q)>1.\)

\item The set
\[
I_X=\{q\in(1,2):\rho(q)<1\}
\]
is either empty or an initial interval of \((1,2)\), namely one of
\[
(1,q_*),\qquad (1,q_*],\qquad (1,2),
\]
where \(q_*\in(1,2)\) in the first two cases.
\end{enumerate}
\end{lemma}

\begin{proof}
For \(0<q\le1\), 
\(
x^q\le1+x
\),
and hence 
\(
\rho(q)<\infty
\).
Writing
\[
\rho(q)=\int x^q\,d\nu_X(x),\qquad
\nu_X=\E(\delta_{X_0}+\delta_{X_1}),
\]
Hölder's inequality implies the log-convexity of \(\rho\) on its finiteness
intervals.

The balanced assumption implies 
\(
\rho(1)=1
\).
Since 
\(
H^+(X)<\infty
\),
the family
\(
x^q\log_2 x
\),
\(
q\uparrow1
\), 
is dominated by
\(
x\log_2^+x+C\mathbf 1_{\{0<x\le1\}}
\), 
and hence
\[
\left.\frac{d}{dq}\right|_{q=1-}\log_2\rho(q)
=\E[X_0\log_2X_0+X_1\log_2X_1]
=\gamma(X)<0 .
\]
Convexity then implies
\[
\log_2\rho(q)\ge \gamma(X)(q-1)>0,
\]
so 
\(
\rho(q)>1
\).

If 
\(
q_0\in I_X
\), 
convexity between \(1\) and \(q_0\) implies
\(
\rho(q)<1
\)
for all 
\(
1<q\le q_0
\).
Hence \(I_X\) is an initial interval of
\((1,2)\), with the stated endpoint alternatives.
\end{proof}

\begin{proposition}
\label{prop:vector-DE-variational-formula-paper}
Assume the vector law is energy-admissible. Then
\begin{equation}\label{eq:vector-DE-variational-formula-paper}
D_E(X)
=
\sup_{\substack{1<q<2\\ \rho(q)<1}}
\left(-\frac{2}{q}\log_2\rho(q)\right),
\end{equation}
with the empty supremum interpreted as \(0\). In particular,
\(
D_E(X)>0
\)
if and only if
\(
\rho(q)<1
\)
for some 
\(
q\in(1,2)
\).

Equivalently,
\(
D_E(X)>0
\)
if and only if 
\(
\rho(q)<\infty
\) for some \(q>1\).
\end{proposition}

\begin{proof}
Let \(0<s<1\). By Lemma~\ref{lem:rho-convexity-subcritical-interval-paper},
\(
2^{qs/2}\rho(q)>1
\)
for 
\(
0<q\le1
\). 
Hence only 
\(
q\in(1,2)
\)
can satisfy
\(
2^{qs/2}\rho(q)<1
\). 
For such \(q\), this inequality is equivalent to
\(\rho(q)<1\) and 
\(
s<-\frac2q\log_2\rho(q)
\).
Taking the supremum over \(s\) and \(q\) implies the formula and the equivalence
\[D_E(X)>0\iff \rho(q)<1 \quad\text{for some}\quad q\in(1,2).\]

 If 
\(
\rho(q)<1
\),
then 
\(
\rho(q)<\infty
\).
Conversely, if
\(
\rho(q_0)<\infty
\)
for some \(q_0>1\), choose 
\(
r\in(1,2)
\) 
with
\(
r\le q_0
\). 
The right derivative of 
\(
\log_2\rho
\) 
at \(1\) along \([1,r]\)
equals 
\(
\gamma(X)<0
\),
so 
\(
\rho(q)<1
\) 
for all \(q>1\) sufficiently close to
\(1\).
\end{proof}

\begin{corollary}
\label{cor:scalar-reduction-vector-profile-paper}
Let 
\(
W\ge0
\)
satisfy 
\(
\E W=1
\),
let 
\(
W_0,W_1
\)
be independent copies of
\(W\), and set 
\(
X_i=\frac{W_i}{2}
\),
\(i=0,1.\) Then the vector law is balanced
and, for every \(q>0\), we have 
\[
\rho(q)=2^{1-q}\E[W^q],
\]
with the identity
understood in the extended sense.

If, in addition, the resulting vector law is energy-admissible, then
\[
D_E(X)
=
\sup_{1<q<2}
\max\left\{
0,\,
2-\frac2q\bigl(1+\log_2\E[W^q]\bigr)
\right\},
\]
where the corresponding term is interpreted as \(0\) whenever
\(
\E[W^q]=\infty
\).
In particular, this applies under the scalar assumptions
\eqref{eq:minimal-KP-circle}.
\end{corollary}

\begin{remark}
\label{rem:sibling-dependence-rho-paper}
The profile 
\(
\rho(q)=\E[X_0^q+X_1^q]
\)
contains all sibling dependence relevant
to the energy and Fourier lower-bound arguments. The coordinates \(X_0\) and
\(X_1\) may be dependent. What is used repeatedly is independence between the
descendant environments rooted at distinct vertices, together with the identity
\eqref{eq:q-mass-identity-paper}.
\end{remark}

\begin{remark}\label{rem:comparison-microcanonical-vector-model}
A different vector-valued model appears in
\cite{ChenHanQiuWang2025Microcanonical}, where the conservative constraint
\(X_0+X_1=1\) is imposed almost surely.  This is distinct from the balanced
model considered here: we assume only
\(\mathbb E X_0=\mathbb E X_1=1/2\), and do not impose a pointwise conservation
law.
\end{remark}

\section{Square sums and the vector energy theorem}
\label{sec:vector-square-sums-energy-theorem}

This section records the energy theorem in the tree-cylinder form needed for
the Fourier argument.  Throughout this section, the dyadic vector-weight law is
assumed to be energy-admissible.  Thus the limiting measure \(\mu\) exists,
\[
    \mathbb P(M>0)>0,
\]
and the moment profile
\[
    \rho(q)=\E[X_0^q+X_1^q]
\]
has the properties established in
Section~\ref{sec:vector-cascades-energy-profile}.

The proof is divided into two regimes: a subcritical regime, based on
fractional moments of tree-cylinder square sums, and a supercritical regime,
based on a no-plateau lemma and alive-tree amplification.

The exponent appearing here is closely related to the classical scaling-exponent
theory of independent random cascades.  In Molchan's notation
\cite{Molchan1996}, take \(c=2\) and \(w_i=2X_i\).  If \(w^{(r)}\) denotes a
uniformly chosen component of the generator vector, then
\[
    \mathbb E[(w^{(r)})^q]
    =
    2^{q-1}\rho(q),
    \qquad
    \rho(q)=\mathbb E(X_0^q+X_1^q).
\]
Hence Molchan's exponent
\[
    \Phi^{(r)}(q)
    =
    \log_2\mathbb E[(w^{(r)})^q]-q+1
\]
is exactly \(\log_2\rho(q)\).  Thus \(D_{\mathrm E}(X)\) is the
energy/correlation exponent associated with this scaling formalism.  We give
the proof in the tree-cylinder normalization because the Fourier argument below
uses the exact branching recursion for
\[
    \Sigma_n=\sum_{|u|=n}C_u^2,
\]
together with conditioning on non-extinction and the transfer from
tree-cylinder masses to Euclidean dyadic square sums.

\subsection{Deterministic square sums and energy}
\label{subsec:deterministic-square-sums-energy}

We first record the deterministic identifications needed to pass between
tree-cylinder square sums, Euclidean dyadic square sums, and Riesz energy.  The
relation between one-dimensional Riesz energy and dyadic square sums is standard
in the theory of correlation dimensions; see, for example, \cite{Pesin1993}.

\begin{proposition} 
\label{prop:dyadic-square-sum-energy-criterion-paper}
Let \(\nu\) be a nonzero finite Borel measure on \([0,1]\). Then 
\(
\dimE(\nu)=\dimtwo(\nu)
\). 
Equivalently,
\[
\dimE(\nu)
=
\liminf_{n\to\infty}
\frac{\log_2 \Sigma_n(\nu)}{-n}.
\]
\end{proposition}

\begin{remark}
\label{rem:atoms-zero-energy-dimension-paper}
If \(\nu\) has an atom of mass \(a>0\), then 
\(
\Sigma_n(\nu)\ge a^2
\)
for every \(n\), and hence 
\(
\dimtwo(\nu)=0
\).
The diagonal contribution of this atom forces 
\(
I_t(\nu)=\infty
\)
for every \(0<t<1\), so that 
\(
\dimE(\nu)=0
\).
This is consistent with Proposition~\ref{prop:dyadic-square-sum-energy-criterion-paper}.
\end{remark}

For the cascade argument, the square sums with an exact branching recursion are
the tree-cylinder square sums, rather than the Euclidean dyadic sums.

On the event of simultaneous descendant convergence, set
\[
C_u=L_uM^{(u)}
\qquad(u\in\mathcal T).
\]
Then 
\[
C_{\varnothing}=M\quad \text{and}\quad C_u=C_{u0}+C_{u1}\quad (u\in\mathcal T).
\]

\begin{definition}[Tree-cylinder square sums]
\label{def:tree-cylinder-square-sums-paper}
For \(n\ge0\), define 
\[
\Sigma_n
=
\sum_{\lvert u \rvert=n}C_u^2.
\]

More generally, for a vertex 
\(
u\in\mathcal T
\),
define the descendant tree-cylinder square sum by 
\[
\Sigma^{(u)}_n
=
\sum_{\lvert v \rvert=n}\left(C^{(u)}_v\right)^2,
\qquad
C^{(u)}_v=L_v^{(u)}M^{(uv)}.
\]

Thus
\[
\Sigma_n=\Sigma^{(\varnothing)}_n
\qquad\text{and}\qquad
\Sigma_0=M^2.
\]
\end{definition}

\begin{lemma} 
\label{lem:tree-cylinder-square-sum-branching-identity-paper}
For every 
\(
m,n\ge0
\),
\begin{equation}\label{eq:tree-square-sum-branching-identity-paper}
\Sigma_{m+n}
=
\sum_{\lvert u \rvert=m}L_u^2\Sigma^{(u)}_n.
\end{equation}
In particular, 
\[
\Sigma_{n+1}
=
X_0(\varnothing)^2\Sigma^{(0)}_n
+
X_1(\varnothing)^2\Sigma^{(1)}_n.
\]

The descendant square sums \(\Sigma^{(u)}_n\), \(
\lvert u \rvert=m\), are conditionally independent given \(\mathcal F^X_m\), have the same law as \(\Sigma_n\), and are independent of \(\mathcal F^X_m\). 
\end{lemma}

\begin{proof}
Equation~\eqref{eq:tree-square-sum-branching-identity-paper} follows
immediately from the identity \(C_{uv}=L_uC_v^{(u)}\). The asserted independence and equality in law follow from the independence of the descendant environments below distinct level-\(m\) vertices. 
\end{proof}

\begin{lemma}
\label{lem:elementary-tree-square-sum-bounds-paper}
Almost surely, on the event 
\(
\{M>0\}
\),
one has
\begin{equation}\label{eq:elementary-tree-square-sum-bounds-paper}
2^{-n}M^2\le \Sigma_n\le M^2
\qquad(n\ge0).
\end{equation}
Consequently,
\[
0\le
\liminf_{n\to\infty}
\frac{\log_2\Sigma_n}{-n}
\le1
\qquad
\text{almost surely on }\{M>0\}.
\]
\end{lemma}

\begin{proof}
We work on the event where all \(C_u\) are defined and consistent. Iterating \(C_u=C_{u0}+C_{u1}\) implies \(\sum_{\lvert u \rvert=n}C_u=M\). 
Cauchy's inequality implies \eqref{eq:elementary-tree-square-sum-bounds-paper}. On \(\{M>0\}\), taking logarithms and passing to the liminf yields the asserted bound.
\end{proof}

We now relate the tree-cylinder masses to the Euclidean limiting measure.
Let
\[
\partial\mathcal T=\{0,1\}^{\mathbb N},
\qquad
[u]_\partial=\{\omega\in\partial\mathcal T:\omega|_{\lvert u \rvert}=u\}.
\]
The consistent masses \(C_u\) define a finite Borel measure \(\mu_\partial\) on \(\partial\mathcal T\) by
\[
\mu_\partial([u]_\partial)=C_u.
\]
Let
\[
\pi:\partial\mathcal T\to[0,1],
\qquad
\pi(\omega)=\sum_{j=1}^{\infty}\omega_j2^{-j}
\]
be the binary coding map.

\begin{lemma}
\label{lem:boundary-representation-paper}
Almost surely,
\begin{equation}\label{eq:boundary-representation-paper}
\mu=\pi_\#\mu_\partial.
\end{equation}
\end{lemma}

\begin{proof}
For every dyadic cylinder \([u]_\partial\), the measures \(\mu_n\) assign
asymptotic mass \(C_u\) to \(I_u\), while \(\mu_\partial([u]_\partial)=C_u\).
Thus \(\mu_n\) and \(\pi_\#\mu_\partial\) have the same limits on dyadic
step functions.  Uniform approximation of \(C([0,1])\) by dyadic step functions, together
with \(\sup_n\mu_n([0,1])<\infty\), implies
\(\mu=\pi_\#\mu_\partial\).
\end{proof}

\begin{lemma} 
\label{lem:endpoint-transfer-square-sums-paper}
Almost surely, on the event \(\{M>0\}\),
\begin{equation}\label{eq:endpoint-transfer-square-sums-paper}
\liminf_{n\to\infty}
\frac{\log_2\Sigma_n}{-n}
=
\dimtwo(\mu)
=
\liminf_{n\to\infty}
\frac{\log_2\Sigma_n(\mu)}{-n}.
\end{equation}
\end{lemma}

\begin{proof}
The second equality is the definition of \(\dimtwo(\mu)\). It remains to compare the tree and Euclidean square sums.

If \(\mu\) has an atom of mass \(a>0\), then by Lemma~\ref{lem:boundary-representation-paper}, \(\mu_\partial(\pi^{-1}\{x\})=a\) for some \(x\). Since \(\pi^{-1}\{x\}\) has at most two points, there exists a boundary path \(\omega\) such that \(\mu_\partial(\{\omega\})\ge a/2\).
Thus
\[
C_{\omega|n}\ge a/2
\qquad(n\ge0),
\]
and hence \(\Sigma_n\ge a^2/4\). 
Together with \(\Sigma_n\le M^2\), this yields
\[
\liminf_{n\to\infty}\frac{\log_2\Sigma_n}{-n}=0.
\]
By Remark~\ref{rem:atoms-zero-energy-dimension-paper}, \(\dimtwo(\mu)=0\). 

If \(\mu\) is non-atomic, then so is
\(\mu_\partial\); since \(\pi^{-1}(I_u)\) and \([u]_\partial\) differ only
by finitely many endpoint codings, \(\mu(I_u)=C_u\) for every \(u\).  Hence
\(\Sigma_n(\mu)=\Sigma_n\) for all \(n\).
\end{proof}

\subsection{Subcritical square-sum decay}
\label{subsec:subcritical-square-sum-decay}

For \(0<s<1\) and \(0<q<2\), set
\begin{equation}\label{eq:ms-profile}
    m_s(q)=2^{qs/2}\rho(q).
\end{equation}
The subcritical regime is characterized by the condition \(m_s(q)<1\) for some
\(q\in(1,2)\).  We first record a terminal-mass moment criterion in the present
vector notation.  Although only the range \(1<q<2\) is needed for the
square-sum estimate below, the formulation for all \(q>1\) will also be used in
the annular theorem.

\begin{lemma}[Terminal-mass moment criterion]
\label{lem:terminal-mass-moment-criterion}
Let \(q>1\).  Assume \(\E(X_0+X_1)=1\).  If
\[
    \rho(q)=\E[X_0^q+X_1^q]<1,
\]
then
\[
    \sup_{n\ge0}\E[M_n^q]<\infty,
    \qquad
    \E[M^q]<\infty .
\]
\end{lemma}

\begin{proof}
Set \(S=X_0+X_1\). 
We first treat \(1<q\le2\). 
Since
\(S^q\le 2^{q-1}(X_0^q+X_1^q)\), it follows that
\(\E|S-1|^q<\infty\). 
Let \(D_{n+1}=M_{n+1}-M_n\). 
Writing
\(S(u)=X_0(u)+X_1(u)\),
\[
D_{n+1}
=
\sum_{|u|=n}L_u\bigl(S(u)-1\bigr).
\]
Conditionally on \(\mathcal F_n^X\), the variables
\(L_u(S(u)-1)\) \((|u|=n)\), are independent and centered. 
Hence the
conditional von Bahr--Esseen inequality yields
\[
\E\bigl[|D_{n+1}|^q\mid\mathcal F_n^X\bigr]
\le
2\E|S-1|^q\sum_{|u|=n}L_u^q .
\]
Taking expectations and using Lemma~\ref{lem:vector-q-mass-identity-paper},
we obtain
\[
\E|D_{n+1}|^q
\le
2\E|S-1|^q\,\rho(q)^n,
\qquad
\|D_{n+1}\|_q
\le
\bigl(2\E|S-1|^q\bigr)^{1/q}\rho(q)^{n/q}.
\]
Since \(M_n=1+\sum_{k=0}^{n-1}D_{k+1}\), Minkowski's inequality implies
\[
\sup_{n\ge0}\|M_n\|_q
\le
1+
\frac{\bigl(2\E|S-1|^q\bigr)^{1/q}}
{1-\rho(q)^{1/q}}.
\]
Thus \(\sup_{n\ge0}\E[M_n^q]<\infty\). 
Since \(M_n\to M\) a.s., 
Fatou's lemma implies \(\E[M^q]<\infty\).

It remains to consider \(q>2\). 
We first note that \(\rho(q)<1\) implies
\(\rho(p)<1\) for every \(1<p\le q\). 
Let \(\nu_X=\E(\delta_{X_0}+\delta_{X_1})\).  If
\(p=(1-\theta)+\theta q\), \(0<\theta\le1\), then H\"older's inequality gives
\[
    \rho(p)
    =
    \int x^p\,d\nu_X(x)
    \le
    \left(\int x\,d\nu_X(x)\right)^{1-\theta}
    \left(\int x^q\,d\nu_X(x)\right)^\theta
    =
    \rho(q)^\theta<1 .
\]
We argue by induction over \((m,m+1]\), \(m\ge2\). 
Assume the assertion holds at exponent \(q-1\). 
Let \(M_n^{(0)}\) and \(M_n^{(1)}\) be independent copies of \(M_n\), independent of \(X=(X_0,X_1)\). 
By the branching construction,
\[
M_{n+1}
\stackrel{d}=
X_0M_n^{(0)}+X_1M_n^{(1)}.
\]
Using
\[
(x+y)^q
\le
x^q+y^q+C_q\bigl(x^{q-1}y+xy^{q-1}\bigr),
\qquad x,y\ge0,
\]
together with independence and \(\E M_n=1\), 
we obtain
\[
\E[M_{n+1}^q]
\le
\rho(q)\E[M_n^q]
+
C_q\E\!\left[X_0^{q-1}X_1+X_0X_1^{q-1}\right]\E[M_n^{q-1}].
\]
Young's inequality implies
\[
\E\!\left[X_0^{q-1}X_1+X_0X_1^{q-1}\right]
\le
\rho(q)<\infty .
\]
Since \(\rho(q-1)<1\), the induction hypothesis yields
\[
\sup_{n\ge0}\E[M_n^{q-1}]<\infty .
\]
Consequently, for some finite \(C_q'\),
\[
\E[M_{n+1}^q]
\le
\rho(q)\E[M_n^q]+C_q',
\qquad n\ge0 .
\]
Since \(\rho(q)<1\), iteration yields
\[
\sup_{n\ge0}\E[M_n^q]<\infty .
\]
Finally, Fatou's lemma and \(M_n\to M\) a.s. imply
\[
\E[M^q]
\le
\liminf_{n\to\infty}\E[M_n^q]
<\infty.
\]
\end{proof}

\begin{lemma}
\label{lem:terminal-mass-moment-bound-paper}
Let 
\(
q>1
\),
and suppose that 
\(
\rho(q)<1
\). 
Then 
\(
\E[M^q]<\infty
\).
Moreover, for every 
\(
u\in\mathcal T
\),
the descendant terminal mass \(M^{(u)}\) has the same law as \(M\). In particular,
\[
\E[(M^{(u)})^q]<\infty.
\]
\end{lemma}

\begin{proof}
Apply Lemma~\ref{lem:terminal-mass-moment-criterion} to the binary vector \((X_0,X_1)\); descendant cascades have the same law as the original one.
\end{proof}

By Lemma~\ref{lem:terminal-mass-moment-bound-paper}, whenever 
\(
0<s<1
\),
\(
1<q<2
\),
and 
\(
m_s(q)<1
\), 
we have 
\(
\E[M^q]<\infty
\).

\begin{theorem}[Subcritical tree-cylinder square-sum decay]
\label{thm:subcritical-tree-square-sum-decay-paper}
Let \(0<s<1\), \(
1<q<2\), and set \(p=\frac q2\in(0,1)\). 
Assume
\[
m_s(q)=2^{qs/2}\rho(q)<1.
\]
Then, for every \(n\ge0\), 
\[
\E\left[(2^{sn}\Sigma_n)^p\right]
\le
\E[M^q]\,m_s(q)^n .
\]

In particular, \(2^{sn}\Sigma_n\to0\) almost surely.
\end{theorem}

\begin{proof}
By the preceding observation, \(\E[M^q]<\infty\). 
Using \(C_u=L_uM^{(u)}\), independence, \(M^{(u)}\stackrel d=M\), and subadditivity of \(x^p\) for \(p=q/2<1\), we get \(\mathbb E\Sigma_n^p\le \mathbb E[M^q]\rho(q)^n\).
Multiplying by 
\(
2^{snp}=2^{snq/2}
\), 
we obtain
\[
\E\left[(2^{sn}\Sigma_n)^p\right]
\le
\E[M^q]\bigl(2^{sq/2}\rho(q)\bigr)^n
=
\E[M^q]m_s(q)^n.
\]

Since 
\(
m_s(q)<1
\),
Markov's inequality implies, for every 
\(
\varepsilon>0
\),
\[
\sum_{n=0}^\infty
\Pbb(2^{sn}\Sigma_n>\varepsilon)
\le
\varepsilon^{-p}\E[M^q]\sum_{n=0}^\infty m_s(q)^n
<\infty.
\]
By Borel--Cantelli, applied with 
\(
\varepsilon=1/k
\), 
it follows that 
\(
2^{sn}\Sigma_n\to0
\)
almost surely.
\end{proof}

\begin{corollary}
\label{cor:subcritical-tree-square-sum-lower-bound-paper}
If 
\(
0<s<D_E(X)
\), 
then
\begin{equation}\label{eq:subcritical-tree-exponent-lower-bound-paper}
\liminf_{n\to\infty}
\frac{\log_2\Sigma_n}{-n}
\ge s
\qquad
\text{almost surely on }\{M>0\}.
\end{equation}
Consequently,
\begin{equation}\label{eq:tree-exponent-lower-DE-paper}
\liminf_{n\to\infty}
\frac{\log_2\Sigma_n}{-n}
\ge D_E(X)
\qquad
\text{almost surely on }\{M>0\}.
\end{equation}
\end{corollary}

\begin{proof}
If 
\(
D_E(X)=0
\),
the final bound follows from Lemma~\ref{lem:elementary-tree-square-sum-bounds-paper}. Assume that 
\(
D_E(X)>0
\),
and let 
\(
0<s<D_E(X)
\).
By Proposition~\ref{prop:vector-DE-variational-formula-paper}, there exists 
\(
q\in(1,2)
\)
such that 
\(
m_s(q)<1
\).
Hence Theorem~\ref{thm:subcritical-tree-square-sum-decay-paper} implies 
\(
2^{sn}\Sigma_n\to0
\)
almost surely. Thus, almost surely, 
\(
\Sigma_n\le 2^{-sn}
\)
for all sufficiently large \(n\). On 
\(
\{M>0\}
\), Lemma~\ref{lem:elementary-tree-square-sum-bounds-paper} ensures that 
\(
\Sigma_n>0
\),
and therefore
\[
\liminf_{n\to\infty}
\frac{\log_2\Sigma_n}{-n}
\ge s.
\]
Letting 
\(
s\uparrow D_E(X)
\)
along a countable sequence yields \eqref{eq:tree-exponent-lower-DE-paper}.
\end{proof}

\begin{corollary}
\label{cor:subcritical-energy-correlation-lower-bound-paper}
Almost surely on 
\(
\{M>0\}
\), 
\[
\dimtwo(\mu)=\dimE(\mu)\ge D_E(X).
\]
\end{corollary}

\begin{proof}
By Lemma~\ref{lem:endpoint-transfer-square-sums-paper} and Corollary~\ref{cor:subcritical-tree-square-sum-lower-bound-paper},
\[
\dimtwo(\mu)
=
\liminf_{n\to\infty}\frac{\log_2\Sigma_n}{-n}
\ge D_E(X)
\]
almost surely on 
\(
\{M>0\}
\).
Since \(\mu\) is then a nonzero finite Borel measure, Proposition~\ref{prop:dyadic-square-sum-energy-criterion-paper} applies and implies 
\(
\dimE(\mu)=\dimtwo(\mu)
\).
\end{proof}

\subsection{The no-plateau lemma and the supercritical obstruction}
\label{subsec:no-plateau-supercritical-obstruction}

We now prove the supercritical half of the tree-cylinder square-sum estimate.
This is the part of the energy argument closest to the classical
scaling-exponent theory.  The proof below isolates the finite-block obstruction
needed in the present normalization.

Let
\[
    N_X=\one_{\{X_0>0\}}+\one_{\{X_1>0\}}
\]
denote the number of positive children of a single vertex.  The first step is
a no-plateau statement for
\[
    m_s(q)=2^{qs/2}\rho(q),
\]
showing that, above \(D_{\mathrm E}(X)\), this profile is uniformly separated
from \(1\).

\begin{lemma}
\label{lem:alive-offspring-supercritical-paper}
Under energy-admissibility,
\(
\E N_X>1
\).
\end{lemma}

\begin{proof}
Let 
\(
\mathcal T_+=\{u:L_u>0\}
\).
Its offspring law is \(N_X\). Since energy-admissibility implies 
\(
\mathbb P(M>0)>0
\),
the alive tree survives with positive probability.

If 
\(
\E N_X<1
\),
then, for the alive population \(Z_n\),
\[
\mathbb P(Z_n>0)\le \E Z_n=(\E N_X)^n\to0,
\]
so the survival probability is zero. If 
\(
\E N_X=1
\),
positive survival probability forces the critical Galton--Watson law to be degenerate, namely 
\(
N_X=1
\)
almost surely. Thus exactly one coordinate of \(X\) is positive almost surely. With 
\(
S=X_0+X_1
\),
balance implies 
\(
\E S=1
\),
and
\[
X_0\log_2X_0+X_1\log_2X_1=S\log_2S.
\]
By Jensen's inequality,
\[
\E[S\log_2S]\ge \E S\log_2\E S=0,
\]
contradicting the strict entropy condition in energy-admissibility. Hence 
\(
\E N_X>1
\).
\end{proof}

\begin{lemma}
\label{lem:small-q-gap-paper}
For every compact interval \(K\subset(0,1)\), there exist \(q_0>0\) and
\(\eta>0\) such that
\[
    m_s(q)\ge 1+\eta
    \qquad
    (s\in K,\;0<q\le q_0).
\]
\end{lemma}

\begin{proof}
For 
\(
0<q\le1
\),
one has 
\(
X_0^q+X_1^q\le 2+X_0+X_1,
\)
which is integrable. Since 
\(
X_i^q\to\one_{\{X_i>0\}}
\)
as 
\(
q\downarrow0
\),
dominated convergence and Lemma~\ref{lem:alive-offspring-supercritical-paper} imply
\[
\rho(q)\to \E N_X>1.
\]
Moreover, 
\(
2^{qs/2}\to1
\)
uniformly for 
\(
s\in K
\).
Hence 
\(
m_s(q)=2^{qs/2}\rho(q)\to\E N_X>1
\)
uniformly on \(K\), establishing the claimed gap.
\end{proof}

\begin{lemma}[No plateau above the energy threshold]
\label{lem:no-plateau-above-energy-threshold-paper}
If
\(
D_E(X)<s<1
\),
then 
\[
\inf_{0<q<2}m_s(q)>1.
\]
\end{lemma}

\begin{proof}
Suppose that 
\(
D_E(X)<s<1
\)
and 
\(
\inf_{0<q<2}m_s(q)=1
\). 
Set
\[
s_0=\frac{D_E(X)+s}{2}.
\]
Then 
\(
D_E(X)<s_0<s
\),
and by the definition of \(D_E(X)\), 
\(
\inf_{0<q<2}m_{s_0}(q)\ge1
\). 
By Lemma~\ref{lem:small-q-gap-paper}, there exist 
\(
q_0>0
\)
and 
\(
\eta>0
\)
such that
\[
m_s(q)\ge1+\eta
\qquad(0<q\le q_0).
\]
Thus there exists a sequence 
\(
q_j\in[q_0,2)
\) 
such that 
\(
m_s(q_j)\to1
\).
Since 
\(
m_{s_0}(q)=2^{-q(s-s_0)/2}m_s(q)
\),
we obtain
\[
\inf_{0<q<2}m_{s_0}(q)
\le
\liminf_j m_{s_0}(q_j)
\le
2^{-q_0(s-s_0)/2}<1,
\]
contradicting the preceding lower bound. Hence 
\(
\inf_{0<q<2}m_s(q)>1
\).
\end{proof}

For 
\(
s\in(0,1)
\),
set
\[
A_i^{(s)}=2^sX_i^2,
\quad i=0,1, \qquad \text{and}\qquad Z_n^{(s)}=2^{sn}\Sigma_n.
\]
The one-step branching identity yields
\begin{equation}\label{eq:normalized-square-sum-recursion-paper}
Z_{n+1}^{(s)}
=
A_0^{(s)} Z_n^{(0,s)}
+
A_1^{(s)} Z_n^{(1,s)},
\end{equation}
where 
\(
Z_n^{(0,s)}
\)
and 
\(
Z_n^{(1,s)}
\)
are independent copies of \(Z_n^{(s)}\), independent of \(X\).

\begin{lemma}
\label{lem:kappa-supercritical-gap-paper}
Let \(D_{\mathrm E}(X)<s<1\).  For \(0\le\theta\le1\), define
\[
\kappa_s(\theta)
=
\begin{cases}
\E\bigl[\one_{\{A^{(s)}_0>0\}}+\one_{\{A^{(s)}_1>0\}}\bigr],
& \theta=0,\\[2mm]
\E\bigl[(A^{(s)}_0)^\theta+(A^{(s)}_1)^\theta\bigr],
& 0<\theta\le1 .
\end{cases}
\]
Then
\[
    \inf_{0\le\theta\le1}\kappa_s(\theta)>1 .
\]
\end{lemma}

\begin{proof}
For \(0<\theta<1\), setting \(q=2\theta\) gives
\(\kappa_s(\theta)=m_s(q)\), so the interior bound follows from
Lemma~\ref{lem:no-plateau-above-energy-threshold-paper}.  The endpoint
\(\theta=0\) follows from Lemma~\ref{lem:alive-offspring-supercritical-paper}.
For \(\theta=1\), we pass to the limit \(q\uparrow2\) in \(m_s(q)\), using
truncation if necessary.  Hence the same strict lower bound persists at the
right endpoint.
\end{proof}

The next lemma isolates the finite-block weighted-growth input used in the
supercritical argument.  It is a standard change-of-measure consequence of the
branching-random-walk formalism, but we record the proof in the present notation
in order to keep track of possible zero weights and of the explicit parameters
\(N\), \(\tau\), and \(\beta\), in particular the product condition
\(\beta\tau>1\), which is used below in the tree-cylinder argument.

\begin{lemma}[Finite-block weighted-growth lemma]
\label{lem:finite-block-weighted-growth-paper}
Let 
\(
(A_0,A_1)
\)
be a pair of almost surely finite nonnegative random variables, and define the possibly extended-valued function 
\(
\kappa:[0,1]\to[0,\infty]
\)
by
\[
\kappa(\theta)
=
\begin{cases}
\displaystyle
\mathbb{E}\left[\one_{\{A_0>0\}}+\one_{\{A_1>0\}}\right],
& \theta=0,\\[8pt]
\displaystyle
\mathbb{E}\left[A_0^\theta+A_1^\theta\right],
& 0<\theta\le1.
\end{cases}
\]
Assume that 
\[
\inf_{0\le\theta\le1}\kappa(\theta)>1.
\]

Let 
\(
\mathcal T=\bigcup_{n\ge0}\{0,1\}^n
\)
be the full binary tree, with root \(\varnothing\). Let
\[
\bigl\{(A_0(v),A_1(v)):v\in\mathcal T\bigr\}
\]
be an i.i.d. family of copies of 
\(
(A_0,A_1)
\).
For 
\(
u=(u_1,\ldots,u_n)\in\{0,1\}^n
\),
write 
\(
u|_{j-1}=(u_1,\ldots,u_{j-1})
\),
with 
\(
u|_0=\varnothing
\),
and define
\[
A_u
=
\prod_{j=1}^{n}A_{u_j}(u|_{j-1}),
\qquad
A_\varnothing=1.
\]
Then there exist 
\(
N\ge1
\), 
\(
\tau>0
\),
and 
\(
\beta>1
\)
such that 
\[
\beta\tau>1 \quad \text{and} \quad \mathbb{E}\#\{u\in\{0,1\}^N:A_u\ge\tau\}>\beta.
\]
\end{lemma}

\begin{proof}
Let
\[
\alpha_0=\inf_{0\le\theta\le1}\kappa(\theta)>1.
\]
Choose 
\(
\delta_0>0
\)
such that
\[
1+2\delta_0<\alpha_0 .
\]

For 
\(
m\ge2
\),
set
\[
A_i^{[m]}
=
A_i\one_{\{m^{-1}\le A_i\le m\}},
\qquad i=0,1,
\]
and define
\[
\kappa_m(\theta)
=
\begin{cases}
\displaystyle
\mathbb{E}\left[
\one_{\{A_0^{[m]}>0\}}+\one_{\{A_1^{[m]}>0\}}
\right],
& \theta=0,\\[8pt]
\displaystyle
\mathbb{E}\left[
\bigl(A_0^{[m]}\bigr)^\theta+
\bigl(A_1^{[m]}\bigr)^\theta
\right],
& 0<\theta\le1.
\end{cases}
\]
For each fixed \(\theta\in[0,1]\), monotone convergence yields \(\kappa_m(\theta)\uparrow\kappa(\theta)\), with the case \(\theta=0\) understood through the indicators of \(\{A_i>0\}\). Since \(\inf_{0\le\theta\le1}\kappa(\theta)>1+2\delta_0\), continuity of the truncated functions \(\kappa_m\) and compactness of \([0,1]\) allow us to choose a single \(m_0\) such that
\[
\inf_{0\le\vartheta\le1}\kappa_{m_0}(\vartheta)>1+\delta_0.
\]

We claim that it is enough to prove the lemma for the truncated pair
\(
(A_0^{[m_0]},A_1^{[m_0]}).
\)
Indeed, for each vertex 
\(
v\in\mathcal T
\),
set
\[
A_i^{[m_0]}(v)
=
A_i(v)\one_{\{m_0^{-1}\le A_i(v)\le m_0\}},
\qquad i=0,1.
\]
Then
\(
\{(A_0^{[m_0]}(v),A_1^{[m_0]}(v)):v\in\mathcal T\}
\)
is an i.i.d. family of copies of 
\(
(A_0^{[m_0]},A_1^{[m_0]})
\).
If
\[
A_u^{[m_0]}
=
\prod_{j=1}^{\lvert u \rvert}A_{u_j}^{[m_0]}(u|_{j-1}),
\]
then
\[
0\le A_u^{[m_0]}\le A_u
\]
for every 
\(
u\in\mathcal T
\).
Consequently, for every 
\(
\tau>0
\),
\[
\{u:A_u^{[m_0]}\ge\tau\}
\subseteq
\{u:A_u\ge\tau\}.
\]
Thus any lower bound for the expected number of truncated products exceeding \(\tau\) is also a lower bound for the corresponding original products.

Therefore, from now on, replace 
\(
(A_0,A_1)
\) 
by the truncated pair 
\(
(A_0^{[m_0]},A_1^{[m_0]})
\), 
and write it again as 
\(
(A_0,A_1)
\).
We also write \(\kappa\) for the corresponding function associated with this new pair. Then
\[
A_i\in\{0\}\cup[m_0^{-1},m_0],
\qquad i=0,1,
\]
and, by the preceding display,
\begin{equation}\label{eq:truncated-uniform-supercritical-paper}
\inf_{0\le\theta\le1}\kappa(\theta)>1.
\end{equation}

Define a finite measure \(\nu\) on \(\mathbb R\) by
\[
\nu(B)
=
\mathbb{E}\sum_{i=0}^1
\one_{\{A_i>0,\ \ln A_i\in B\}},
\qquad B\in\mathcal B(\mathbb R).
\]
Since 
\(
A_i\in\{0\}\cup[m_0^{-1},m_0]
\), 
the measure \(\nu\) is supported on 
\(
[-\ln m_0,\ln m_0]
\).
For 
\(
\theta\in\mathbb R
\), 
define
\[
\Lambda(\theta)
=
\ln\int e^{\theta x}\,\nu(dx).
\]
This is well-defined because \(\nu\) has compact support and positive total mass. Moreover, for 
\(
0\le\theta\le1
\),
\[
\int e^{\theta x}\,\nu(dx)
=
\mathbb{E}\sum_{i=0}^1
\one_{\{A_i>0\}}e^{\theta\ln A_i}.
\]
For 
\(
0<\theta\le1
\),
this equals 
\(
\mathbb{E}\left[A_0^\theta+A_1^\theta\right]
\), 
while for 
\(
\theta=0
\)
it equals 
\(
\mathbb{E}\left[\one_{\{A_0>0\}}+\one_{\{A_1>0\}}\right]
\).

By \eqref{eq:truncated-uniform-supercritical-paper},
\[
\Lambda(\theta)=\ln\kappa(\theta)>0,
\qquad 0\le\theta\le1.
\]

Since \(\nu\) has compact support and positive total mass, \(\Lambda\) is finite, convex, and \(C^\infty\) on \(\mathbb R\).

We next use a simple geometric consequence of convexity. Since \(
\Lambda>0
\)
on \([0,1]\), one can choose 
\(
\theta_0\in(0,1)
\)
such that the tangent line to \(\Lambda\) at \(\theta_0\) is positive at both endpoints \(0\) and \(1\). Indeed, this is immediate by taking a minimizer of \(\Lambda\) on \([0,1]\): if the minimizer lies in the interior, the tangent is horizontal, while if it lies at an endpoint, one moves slightly into the interval.

Set
\[
r=\Lambda'(\theta_0),
\qquad
\gamma=\Lambda(\theta_0)-\theta_0\Lambda'(\theta_0).
\]
The values of the tangent line
\[
t\mapsto \Lambda(\theta_0)+(t-\theta_0)\Lambda'(\theta_0)
\]
at \(t=0\) and \(t=1\) are respectively \(\gamma\) and 
\(
\gamma+r
\).
Thus
\(
\gamma>0
\),
\(
\gamma+r>0
\).
Define a probability measure \(Q\) on \(\mathbb R\) by
\[
Q(dx)
=
e^{\theta_0x-\Lambda(\theta_0)}\,\nu(dx).
\]
Indeed,
\[
Q(\mathbb R)
=
e^{-\Lambda(\theta_0)}
\int e^{\theta_0x}\,\nu(dx)
=
1.
\]
Let 
\(
Y_1,Y_2,\ldots
\)
be i.i.d. random variables with law \(Q\), and write
\(
S_N=Y_1+\cdots+Y_N
\).
Since \(Q\) has compact support, \(Y_1\) is integrable, and by the strong law of large numbers,
\[
\frac{S_N}{N}\longrightarrow \mathbb E_QY_1
\qquad Q\text{-a.s.}
\]
Moreover,
\[
\mathbb E_QY_1
=
\int x\,Q(dx)
=
e^{-\Lambda(\theta_0)}\int xe^{\theta_0x}\,\nu(dx)
=
\Lambda'(\theta_0)
=
r.
\]
Therefore
\[
\frac{S_N}{N}\longrightarrow r
\qquad Q\text{-a.s.}
\]

Choose \(\eta>0\) sufficiently small that 
\[
\gamma-\theta_0\eta>0,
\qquad
\gamma+r-\eta-\theta_0\eta>0.
\]

Then choose 
\(
\delta>0
\) 
sufficiently small that
\begin{equation}\label{eq:delta-choice-finite-block-paper}
\gamma-\theta_0\eta-2\delta>0,
\qquad
\gamma+r-\eta-\theta_0\eta-2\delta>0.
\end{equation}

For 
\(
N\ge1
\), 
define
\[
Z_N
=
\#\left\{
u\in\{0,1\}^N:
A_u\ge e^{N(r-\eta)}
\right\}.
\]
Equivalently,
\[
Z_N
=
\sum_{\lvert u \rvert=N}
\one_{\{A_u>0\}}
\one_{[N(r-\eta),\infty)}(\ln A_u).
\]

We shall use the following many-to-one identity: for every nonnegative Borel function 
\(
f:\mathbb R\to[0,\infty]
\)
and every 
\(
N\ge1
\),
\begin{equation}\label{eq:many-to-one-nu-convolution-paper}
\mathbb{E}\sum_{\lvert u \rvert=N}
\one_{\{A_u>0\}}f(\ln A_u)
=
\int f(x)\,\nu^{*N}(dx).
\end{equation}
Indeed, the case 
\(
N=1
\)
is precisely the definition of \(\nu\), and the general case follows by induction on \(N\), using the independence of the weights on different vertices and the definition of convolution.

Taking
\[
f(x)=\one_{[N(r-\eta),\infty)}(x)
\]
in \eqref{eq:many-to-one-nu-convolution-paper}, we obtain
\[
\mathbb{E}Z_N
=
\nu^{*N}\bigl([N(r-\eta),\infty)\bigr).
\]

Let 
\(
I_N=[N(r-\eta),N(r+\eta)]
\). 
Then
\[
\mathbb{E}Z_N
\ge
\nu^{*N}(I_N).
\]
We next estimate 
\(
\nu^{*N}(I_N)
\).
Since 
\(
Q(dx)=e^{\theta_0x-\Lambda(\theta_0)}\,\nu(dx)
\), 
we have
\[
\nu(dx)=e^{\Lambda(\theta_0)-\theta_0x}\,Q(dx).
\]
Hence, by the definition of convolution and by the fact that 
\(
(Y_1,\ldots,Y_N)
\)
has law 
\(
Q^{\otimes N}
\),
\[
\begin{aligned}
\nu^{*N}(I_N)
&=
e^{N\Lambda(\theta_0)}
\mathbb E_Q\left[
e^{-\theta_0S_N}\one_{\{S_N\in I_N\}}
\right] \\
&\ge
e^{N\Lambda(\theta_0)}
e^{-\theta_0N(r+\eta)}
Q(S_N\in I_N).
\end{aligned}
\]
By the strong law,
\[
Q(S_N\in I_N)\longrightarrow1.
\]
Consequently, for all sufficiently large \(N\), 
\(
Q(S_N\in I_N)\ge e^{-N\delta}
\).
For such \(N\),
\[
\mathbb{E}Z_N
\ge
\exp\left(N(\Lambda(\theta_0)-\theta_0r-\theta_0\eta-\delta)\right).
\]
Since 
\(
\gamma=\Lambda(\theta_0)-\theta_0r
\),
this becomes
\[
\mathbb{E}Z_N
\ge
\exp\left(N(\gamma-\theta_0\eta-\delta)\right).
\]

Choose such an \(N\), and set
\[
\tau=e^{N(r-\eta)},
\qquad
\beta=\exp\left(N(\gamma-\theta_0\eta-2\delta)\right).
\]
By \eqref{eq:delta-choice-finite-block-paper}, 
\(
\beta>1
\).
Moreover, 
\(
\mathbb{E}Z_N>\beta
\).
Since
\[
Z_N=\#\{u\in\{0,1\}^N:A_u\ge\tau\},
\]
we have
\[
\mathbb{E}\#\{u\in\{0,1\}^N:A_u\ge\tau\}>\beta.
\]
Finally,
\[
\beta\tau
=
\exp\left(
N(\gamma+r-\eta-\theta_0\eta-2\delta)
\right)>1
\]
by \eqref{eq:delta-choice-finite-block-paper}. This proves the lemma for the truncated pair, and hence for the original pair.
\end{proof}

\begin{lemma}
\label{lem:bounded-gw-growth-paper}
Let 
\(
(Z_k)_{k\ge0}
\)
be a Galton--Watson process with 
\(
Z_0=1
\),
offspring distribution \(\xi\), and mean \(m=\mathbb{E}\xi>1\).
Assume that \(\xi\) is bounded. Then, for every 
\(
1<\beta<m
\),
\(
Z_k\ge \beta^k
\)
for all sufficiently large \(k\), almost surely on the survival event
\[
S=\{Z_k>0\text{ for all }k\ge0\}.
\]
\end{lemma}

\begin{proof}
Since \(\xi\) is bounded, \(\mathbb E[\xi\log^+\xi]<\infty\). By the
Kesten--Stigum theorem for supercritical Galton--Watson processes,
\[
\frac{Z_k}{m^k}\longrightarrow Y_\infty
\]
almost surely, where \(Y_\infty>0\) almost surely on the survival event \(S\);
see \cite[Theorem~A]{LyonsPemantlePeres1995}, or the original theorem of
Kesten and Stigum~\cite{KestenStigum1966}. Hence, on \(S\),
\[
\frac{Z_k}{\beta^k}
=
\frac{Z_k}{m^k}\left(\frac{m}{\beta}\right)^k
\longrightarrow \infty. 
\]
Therefore \(Z_k\ge \beta^k\) for all sufficiently large
\(k\), almost surely on \(S\).
\end{proof}

\begin{proposition}[Positive-probability supercritical divergence]
\label{prop:positive-probability-supercritical-divergence-paper}
Let 
\(
D_E(X)<s<1
\).
Then
\begin{equation}\label{eq:positive-probability-supercritical-divergence-paper}
\mathbb P\left(
\limsup_{n\to\infty}2^{sn}\Sigma_n=\infty
\right)>0.
\end{equation}
\end{proposition}

\begin{proof}
Recall that 
\(
A_i^{(s)}=2^sX_i^2
\), 
\(i=0,1\), and, for 
\(
u\in\mathcal T
\),
\[
A_u^{(s)}
=
\prod_{j=1}^{\lvert u \rvert}
A_{u_j}^{(s)}(u|_{j-1}).
\]
Thus 
\(
A_u^{(s)}=2^{s\lvert u \rvert}L_u^2
\),
and
\begin{equation}\label{eq:scaled-square-sum-as-A-sum-paper}
2^{sn}\Sigma_n
=
\sum_{\lvert u \rvert=n}
A_u^{(s)}\bigl(M^{(u)}\bigr)^2.
\end{equation}

Using the convention \(a^0=\one_{\{a>0\}}\), Lemma~\ref{lem:kappa-supercritical-gap-paper}
gives
\[
    \inf_{0\le\theta\le1}
    \E\bigl[(A^{(s)}_0)^\theta+(A^{(s)}_1)^\theta\bigr]>1 .
\]
Applying Lemma~\ref{lem:finite-block-weighted-growth-paper} to
\((A^{(s)}_0,A^{(s)}_1)\), we obtain \(N\ge1\), \(\tau>0\), and \(\beta>1\)
such that
\[
    \beta\tau>1,
    \qquad
    m:=\E\#\{u\in\{0,1\}^N:A^{(s)}_u\ge\tau\}>\beta .
\]

Define an embedded \(N\)-block Galton--Watson process. The root is good; if \(u\) is good, then \(uv\), 
\(
v\in\{0,1\}^N
\), 
is good if and only if
\[
\frac{A_{uv}^{(s)}}{A_u^{(s)}}\ge\tau.
\]
Let \(\mathcal G_k\) be the set of good vertices in generation \(k\). The offspring distribution has mean 
\(
m>\beta>1
\)
and is bounded by \(2^N\); hence the process survives with positive probability. Denote its survival event by \(S_{\mathrm{good}}\). By Lemma~\ref{lem:bounded-gw-growth-paper}, on \(S_{\mathrm{good}}\),
\begin{equation}\label{eq:good-tree-growth-eventual-paper}
\#\mathcal G_k\ge\beta^k
\end{equation}
for all sufficiently large \(k\), almost surely. Moreover, every 
\(
u\in\mathcal G_k
\)
satisfies
\[
\lvert u \rvert=kN,
\qquad
A_u^{(s)}\ge\tau^k.
\]

Since 
\(
\mathbb P(M>0)>0
\),
choose \(a>0\) with
\[
p_a:=\mathbb P(M^2\ge a)>0.
\]
For each \(k\), conditionally on \(\mathcal F^X_{kN}\), the variables
\[
\one_{\{(M^{(u)})^2\ge a\}},
\qquad u\in\mathcal G_k,
\]
are independent Bernoulli variables with parameter \(p_a\). Hence, by the binomial Chernoff bound, there exists 
\(
c_a>0
\)
such that
\[
\mathbb P(B_k\mid\mathcal F^X_{kN})
\le
\exp(-c_a\#\mathcal G_k),
\]
where
\[
B_k
=
\left\{
\#\left\{
u\in\mathcal G_k:
(M^{(u)})^2\ge a
\right\}
<
\frac{p_a}{2}\#\mathcal G_k
\right\}.
\]

For 
\(
L\ge1
\),
set
\[
C_L
=
\left\{
\#\mathcal G_k\ge\beta^k
\text{ for every } k\ge L
\right\}.
\]
For 
\(
k\ge L
\),
\[
\mathbb P(B_k\cap C_L)
\le
\mathbb E\left[
\one_{\{\#\mathcal G_k\ge\beta^k\}}
\mathbb P(B_k\mid\mathcal F^X_{kN})
\right]
\le
\exp(-c_a\beta^k).
\]
Thus 
\(
\sum_{k=L}^{\infty}\mathbb P(B_k\cap C_L)<\infty
\),
and Borel--Cantelli implies
\[
\mathbb P(B_k\ \text{i.o.}\ \cap C_L)=0.
\]
Since, by \eqref{eq:good-tree-growth-eventual-paper},
\[
S_{\mathrm{good}}
\subseteq
\bigcup_{L=1}^{\infty}C_L
\]
up to a null event, we have
\[
\mathbb P(B_k\ \text{i.o.}\ \cap S_{\mathrm{good}})=0.
\]
Consequently, on 
\(
S_{\mathrm{good}}
\),
almost surely, for all sufficiently large \(k\),
\[
\#\left\{
u\in\mathcal G_k:
(M^{(u)})^2\ge a
\right\}
\ge
\frac{p_a}{2}\#\mathcal G_k
\ge
\frac{p_a}{2}\beta^k.
\]

Along \(n=kN\), from \eqref{eq:scaled-square-sum-as-A-sum-paper}, we obtain
\[
\begin{aligned}
2^{skN}\Sigma_{kN}
&=
\sum_{\lvert u \rvert=kN}
A_u^{(s)}\bigl(M^{(u)}\bigr)^2  \\
&\ge
a\tau^k
\#\left\{
u\in\mathcal G_k:
(M^{(u)})^2\ge a
\right\}
\ge
\frac{a p_a}{2}(\beta\tau)^k .
\end{aligned}
\]
Since \(\beta\tau>1\), the right-hand side tends to \(+\infty\). Therefore
\[
\limsup_{n\to\infty}2^{sn}\Sigma_n=\infty
\]
on \(S_{\mathrm{good}}\), up to a null event. Since 
\(
\mathbb P(S_{\mathrm{good}})>0
\),
\eqref{eq:positive-probability-supercritical-divergence-paper} follows.
\end{proof}

\subsection{Alive-tree amplification}
\label{subsec:alive-tree-amplification-paper}

The finite-block argument above gives divergence with positive probability.  We
now amplify this to an almost sure statement on the non-extinction event
\(\{M>0\}\).  This step uses only the alive tree and the independence of
descendant cascades, and is included explicitly because zero weights are
allowed and the final energy theorem is conditioned on non-extinction.

Let
\[
\mathcal T_+=\{u\in\mathcal T:L_u>0\}
\]
be the alive tree. Its generation-\(n\) population size is
\[
Z_n^+=\#\{u\in\{0,1\}^n:L_u>0\}.
\]

\begin{lemma}[Alive-tree growth]
\label{lem:alive-tree-growth-paper}
The alive tree is a supercritical Galton--Watson tree. On its survival event,
\(
Z_n^+\to\infty
\)
almost surely.
Moreover,
\[
\{M>0\}\subseteq\{\mathcal T_+\text{ survives}\}
\]
up to null events.
\end{lemma}

\begin{proof}
Each alive vertex has offspring law
\[
N_X=\one_{\{X_0>0\}}+\one_{\{X_1>0\}},
\]
and offspring numbers at distinct vertices are independent. Hence \(\mathcal T_+\) is a Galton--Watson tree. By Lemma~\ref{lem:alive-offspring-supercritical-paper}, 
\(
\E N_X>1
\);
thus it is supercritical, and its population tends to infinity almost surely on its survival event.

If the alive tree becomes extinct, then \(L_u=0\) for all sufficiently deep vertices \(u\). Hence \(M_n=0\) eventually, and therefore \(M=0\). Thus
\[
\{M>0\}\subseteq\{\mathcal T_+\text{ survives}\}
\]
up to null events.
\end{proof}

For a vertex \(u\), define the descendant normalized square-sum limsup by
\[
Y_s^{(u)}
=
\limsup_{n\to\infty}2^{sn}\Sigma_n^{(u)}.
\]
For each fixed \(m\), conditionally on \(\mathcal F_m^X\), the family
\[
\{Y_s^{(u)}:\lvert u \rvert=m\}
\]
is independent, and each member has the same law as
\[
Y_s=\limsup_{n\to\infty}2^{sn}\Sigma_n.
\]

\begin{proposition}[Subtree amplification for supercritical divergence]
\label{prop:subtree-amplification-supercritical-divergence-paper}
Let 

\(
D_E(X)<s<1
\).
Then \[\limsup_{n\to\infty}2^{sn}\Sigma_n=\infty\qquad \text{almost surely on } \{M>0\}.\]
\end{proposition}

\begin{proof}
Set
\[
Y_s=\limsup_{n\to\infty}2^{sn}\Sigma_n .
\]
By Proposition~\ref{prop:positive-probability-supercritical-divergence-paper},
\[
p_s:=\mathbb P(Y_s=\infty)>0,
\qquad
q_s:=\mathbb P(Y_s<\infty)<1.
\]

Fix \(m\ge0\).  For every \(n\ge0\), the branching identity gives
\[
2^{s(m+n)}\Sigma_{m+n}
=
\sum_{|u|=m}
2^{sm}L_u^2
\left(
2^{sn}\Sigma_n^{(u)}
\right).
\]
For \(|u|=m\), set
\[
    Y_s^{(u)}=\limsup_{n\to\infty}2^{sn}\Sigma_n^{(u)} .
\]
Conditionally on \(\mathcal F^X_m\), the variables
\[
    \{Y_s^{(u)}: |u|=m\}
\]
are independent copies of \(Y_s\).  Moreover, if \(L_u>0\) and
\(Y_s^{(u)}=\infty\), then the corresponding term in the preceding branching
identity has positive coefficient \(2^{sm}L_u^2\), and hence \(Y_s=\infty\).
Therefore
\[
\{Y_s<\infty\}
\subset
\bigcap_{\substack{|u|=m\\ L_u>0}}
\{Y_s^{(u)}<\infty\}.
\]
Since the alive set
\[
    \{u: |u|=m,\ L_u>0\}
\]
is \(\mathcal F^X_m\)-measurable and has cardinality \(Z_m^+\), conditional
independence gives
\[
    \mathbb P(Y_s<\infty \mid \mathcal F^X_m)\le q_s^{Z_m^+}.
\]

Let 
\(
A=\{Y_s<\infty\}
\)
and 
\(
B=\{M>0\}
\).
For \(k\ge1\),
\[
\mathbb P(A\cap B)
\le
\mathbb P(A\cap\{Z_m^+\ge k\})
+
\mathbb P(B\cap\{Z_m^+<k\})
\le
q_s^k+\mathbb P(B\cap\{Z_m^+<k\}).
\]
By Lemma~\ref{lem:alive-tree-growth-paper}, 
\(
Z_m^+\to\infty
\) 
almost surely on \(B\). Letting 
\(
m\to\infty
\),
we obtain
\[
\mathbb P(A\cap B)\le q_s^k.
\]
Finally, let 
\(
k\to\infty
\).
Since 
\(
q_s<1
\),
\[
\mathbb P(Y_s<\infty,\ M>0)=0.
\]
Thus 
\(
Y_s=\infty
\)
almost surely on 
\(
\{M>0\}
\).
\end{proof}

\begin{corollary}
\label{cor:supercritical-tree-square-sum-upper-bound-paper}
Almost surely on \(\{M>0\}\), 
\[
\liminf_{n\to\infty}
\frac{\log_2\Sigma_n}{-n}
\le
D_E(X).
\]
\end{corollary}

\begin{proof}
If 
\(
D_E(X)=1
\), 
the claim follows from Lemma~\ref{lem:elementary-tree-square-sum-bounds-paper}. Assume that 
\(
D_E(X)<1
\),
and let
\(
s\in(D_E(X),1)
\).
By Proposition~\ref{prop:subtree-amplification-supercritical-divergence-paper},
\[
\limsup_{n\to\infty}2^{sn}\Sigma_n=\infty
\qquad\text{almost surely on }\{M>0\}.
\]
Hence, on this event, infinitely many \(n\) satisfy 
\(
\Sigma_n\ge2^{-sn}
\),
and therefore
\(
\liminf_{n\to\infty}
\frac{\log_2\Sigma_n}{-n}
\le s
\).
Letting 
\(
s\downarrow D_E(X)
\)
along a countable sequence proves the claim.
\end{proof}

\subsection{The vector energy theorem}
\label{subsec:vector-energy-theorem-proof}

We now combine the subcritical and supercritical estimates.  The result is first
stated in the tree-cylinder normalization and then transferred to the Euclidean
dyadic square sums and Riesz energy.

\begin{theorem}[Tree-cylinder square-sum theorem]
\label{thm:tree-cylinder-square-sum-theorem-paper}
Assume that the dyadic vector-weight law is energy-admissible.  Then, almost
surely on \(\{M>0\}\),
\begin{equation}\label{eq:tree-cylinder-square-sum-formula-paper}
    \liminf_{n\to\infty}
    \frac{\log_2\Sigma_n}{-n}
    =
    D_{\mathrm E}(X).
\end{equation}
\end{theorem}

\begin{proof}
The lower bound
\[
    \liminf_{n\to\infty}
    \frac{\log_2\Sigma_n}{-n}
    \ge D_{\mathrm E}(X)
\]
is Corollary~\ref{cor:subcritical-tree-square-sum-lower-bound-paper}.  The reverse
bound
\[
    \liminf_{n\to\infty}
    \frac{\log_2\Sigma_n}{-n}
    \le D_{\mathrm E}(X)
\]
is Corollary~\ref{cor:supercritical-tree-square-sum-upper-bound-paper}.  Intersecting
the corresponding full-probability events on \(\{M>0\}\) gives
\eqref{eq:tree-cylinder-square-sum-formula-paper}.
\end{proof}

\begin{theorem}[Vector energy theorem]
\label{thm:vector-energy-theorem-paper}
Assume that the dyadic vector-weight law is energy-admissible.  Then, almost
surely on \(\{M>0\}\),
\begin{equation}\label{eq:vector-energy-formula-paper}
    \dim_{\mathrm E}(\mu)=\dim_2(\mu)=D_{\mathrm E}(X).
\end{equation}
\end{theorem}

\begin{proof}
On \(\{M>0\}\), Lemma~\ref{lem:endpoint-transfer-square-sums-paper}
identifies the tree-cylinder square-sum exponent with the Euclidean dyadic
correlation dimension:
\[
    \dim_2(\mu)
    =
    \liminf_{n\to\infty}
    \frac{\log_2\Sigma_n}{-n}.
\]
Theorem~\ref{thm:tree-cylinder-square-sum-theorem-paper} therefore gives
\[
    \dim_2(\mu)=D_{\mathrm E}(X).
\]
Since \(\mu\) is a nonzero finite Borel measure on \([0,1]\) on this event,
Proposition~\ref{prop:dyadic-square-sum-energy-criterion-paper} yields
\[
    \dim_{\mathrm E}(\mu)=\dim_2(\mu).
\]
This proves \eqref{eq:vector-energy-formula-paper}.
\end{proof}

\begin{corollary}
\label{cor:energy-finiteness-infinite-paper}
Assume that the dyadic vector-weight law is energy-admissible.  Then, almost
surely on \(\{M>0\}\), the following assertions hold:
\[
    I_t(\mu)<\infty
    \qquad\text{for every }0<t<D_{\mathrm E}(X),
\]
and
\[
    I_t(\mu)=\infty
    \qquad\text{for every }D_{\mathrm E}(X)<t<1.
\]
If \(D_{\mathrm E}(X)=0\), the first assertion is vacuous.  If
\(D_{\mathrm E}(X)=1\), the second assertion is vacuous.
\end{corollary}

\begin{proof}
By Theorem~\ref{thm:vector-energy-theorem-paper},
\(\dim_{\mathrm E}(\mu)=D_{\mathrm E}(X)\) almost surely on \(\{M>0\}\).
If \(0<t<D_{\mathrm E}(X)\), choose \(u\) with
\(t<u<D_{\mathrm E}(X)\) and \(I_u(\mu)<\infty\).  Since
\(|x-y|^{-t}\le |x-y|^{-u}\) on \([0,1]^2\), it follows that
\(I_t(\mu)<\infty\).  Conversely, if \(D_{\mathrm E}(X)<t<1\) and
\(I_t(\mu)<\infty\), then the definition of energy dimension would give
\(\dim_{\mathrm E}(\mu)\ge t>D_{\mathrm E}(X)\), a contradiction.
\end{proof}

\section{Dense-grid Fourier lower bound for vector cascades}
\label{sec:vector-dense-grid-fourier-lower-bound}

This section proves the Fourier lower bound for the interval vector cascade.  The
energy theorem of Section~\ref{sec:vector-square-sums-energy-theorem} enters
only through the moment profile: for each \(q\in(1,2)\) with \(\rho(q)<1\), we
prove a Fourier decay estimate at every exponent
\[
    0<\beta<-\frac1q\log_2\rho(q).
\]
More precisely, almost surely on \(\{M>0\}\), there exists a finite random
constant \(C_\beta\) such that
\(
    \lvert\widehat{\mu}(\xi)\rvert
    \le C_\beta\lvert\xi\rvert^{-\beta}
\)
for all sufficiently large \(\lvert\xi\rvert\).  Optimizing over \(q\) then gives
\(
    \dim_{\mathrm F}(\mu)\ge D_{\mathrm E}(X).
\)

The argument is directly vector-valued.  It uses the balanced centering
\(\mathbb E\mu_n=\mathcal L\), centered Fourier profiles on dense frequency
grids, a mesoscopic star equation, and a vector-valued
\(\ell^r\)-contraction estimate.  The \(q\)-level annular decay theorem,
Theorem~\ref{thm:q-level-annular-decay-paper}, is obtained by iterating the
resulting grid-norm recursion through a ladder argument.

\subsection{Centered profiles and dense grids}
\label{subsec:centered-profiles-dense-grids-paper}

Let \(\mathcal L\) denote Lebesgue measure on \([0,1]\), and write
\[
    \widehat{\mathcal L}(\xi)
    =
    \int_0^1 e^{-2\pi i\xi x}\,dx .
\]
For \(n\ge0\) and \(\eta\in[1,2]\), define the centered Fourier profile
\[
    H_n(\eta)
    =
    \widehat\mu(2^n\eta)-\widehat{\mathcal L}(2^n\eta).
\]
The balance condition identifies \(\mathcal L\) as the deterministic centering:
\[
    \mathbb E\mu_n=\mathcal L .
\]

For a vertex \(u\), let \(\mu^{(u)}\) denote the descendant limiting cascade rooted
at \(u\), and set
\[
H_n^{(u)}(\eta)
=
\widehat{\mu^{(u)}}(2^n\eta)-\wh{\mathcal L}(2^n\eta).
\]
Then \(H_n^{(u)}\) has the same distribution as \(H_n\), and descendant profiles
rooted at distinct vertices in the same generation are conditionally independent
given the environment up to that generation.

For 
\(
u=(u_1,\ldots,u_m)
\),
let
\[
S_u(x)=a_u+2^{-m}x,\qquad x\in[0,1],
\]
be the affine branch corresponding to the dyadic word \(u\).

\begin{remark}
\label{rem:no-dyadic-endpoint-mass-paper}
The only coding ambiguity in the Euclidean realization occurs at dyadic
endpoints,
\[
    E_{\mathrm{dyad}}
    =
    \{k2^{-n}:n\ge0,\ 0\le k\le 2^n\}.
\]
These points carry no limiting cascade mass.  Indeed, for a fixed
\(\omega\in\partial\mathcal T\),
\[
    \mu_\partial(\{\omega\})
    =
    \lim_{n\to\infty} C_{\omega|n}
    \le C_{\omega|n},
\]
and
\[
    \mathbb E C_{\omega|n}
    =
    \mathbb E L_{\omega|n}\,\mathbb E M
    \le 2^{-n}.
\]
Hence \(\mu_\partial(\{\omega\})=0\) almost surely.  Since each dyadic endpoint
has at most two binary codings and \(E_{\mathrm{dyad}}\) is countable, we have
\[
    \mu(E_{\mathrm{dyad}})=0
\]
almost surely.

Henceforth we work on the full-probability event on which
\(\mu(E_{\mathrm{dyad}})=0\) and the simultaneous descendant construction holds;
on this event dyadic endpoint ambiguities may be ignored in Euclidean interval
notation.
\end{remark}

\begin{lemma}[Star equation in Fourier form]
\label{lem:star-equation-fourier-paper}
For every 
\(
m\ge0
\)
and every 
\(
\xi\in\R
\),
\begin{equation}\label{eq:fourier-star-equation-paper}
\wh\mu(\xi)
=
\sum_{\lvert u \rvert=m}
L_u e^{-2\pi i a_u\xi}
\widehat{\mu^{(u)}}(2^{-m}\xi)
\end{equation}
almost surely on the simultaneous descendant construction event. Moreover,
\begin{equation}\label{eq:fourier-lebesgue-star-equation-paper}
\wh{\mathcal L}(\xi)
=
\sum_{\lvert u \rvert=m}
2^{-m}e^{-2\pi i a_u\xi}
\wh{\mathcal L}(2^{-m}\xi).
\end{equation}
Consequently, for \(n\ge m\) and \(\eta\in[1,2]\),
\begin{equation}\label{eq:centered-profile-mesoscopic-identity-paper}
H_n(\eta)
=
\sum_{\lvert u \rvert=m}
L_u e^{-2\pi i a_u2^n\eta}
H_{n-m}^{(u)}(\eta)
+
\wh{\mathcal L}(2^{n-m}\eta)
\sum_{\lvert u \rvert=m}
\bigl(L_u-2^{-m}\bigr)e^{-2\pi i a_u2^n\eta}.
\end{equation}
\end{lemma}

\begin{proof}
For a continuous function 
\(
e^{-2\pi i\xi x}
\),
the descendant decomposition of the limiting
measure implies
\[
\int e^{-2\pi i\xi x}\,d\mu(x)
=
\sum_{\lvert u \rvert=m}
L_u\int e^{-2\pi i\xi(a_u+2^{-m}y)}\,d\mu^{(u)}(y),
\]
which is precisely \eqref{eq:fourier-star-equation-paper}.

The identity \eqref{eq:fourier-lebesgue-star-equation-paper} follows from the
same computation applied to Lebesgue measure, using the deterministic
decomposition
\[
\mathcal L=\sum_{\lvert u \rvert=m}2^{-m}(S_u)_\#\mathcal L.
\]
Subtracting \eqref{eq:fourier-lebesgue-star-equation-paper} from
\eqref{eq:fourier-star-equation-paper} with \(\xi=2^n\eta\), we arrive at
\eqref{eq:centered-profile-mesoscopic-identity-paper}.
\end{proof}

\begin{lemma}
\label{lem:lebesgue-profile-estimates-paper}
There exists an absolute constant \(C<\infty\) such that, for every \(n\ge0\)
and every \(\eta,\eta'\in[1,2]\),
\[
    |\widehat{\mathcal L}(2^n\eta)|\le C2^{-n}
\]
and
\[
    |\widehat{\mathcal L}(2^n\eta)
      -\widehat{\mathcal L}(2^n\eta')|
    \le C|\eta-\eta'|.
\]
\end{lemma}

\begin{proof}
For \(\xi\neq0\),
\[
    \widehat{\mathcal L}(\xi)
    =
    \frac{e^{-2\pi i\xi}-1}{-2\pi i\xi},
\]
which gives the first estimate.  Differentiating this expression gives
\[
    |\widehat{\mathcal L}'(\xi)|\le C|\xi|^{-1}
    \qquad (|\xi|\ge1).
\]
Since \(2^n\eta\ge1\) for \(n\ge0\) and \(\eta\in[1,2]\), the mean-value theorem
gives the second estimate.
\end{proof}

\begin{lemma}
\label{lem:centered-profile-lipschitz-paper}
There exists an absolute constant \(C<\infty\) such that, almost surely, for
every \(n\ge0\) and every \(\eta,\eta'\in[1,2]\),
\[
    |H_n(\eta)-H_n(\eta')|
    \le C(M+1)2^n|\eta-\eta'|.
\]
\end{lemma}

\begin{proof}
Since \(\mu\) is supported on \([0,1]\),
\[
\begin{aligned}
|\widehat\mu(2^n\eta)-\widehat\mu(2^n\eta')|
&\le
\int_0^1
\left|
e^{-2\pi i2^n\eta x}
-
e^{-2\pi i2^n\eta' x}
\right|\,d\mu(x)  \\
&\le
C M 2^n|\eta-\eta'|.
\end{aligned}
\]
The same bound for \(\widehat{\mathcal L}\) follows from the trivial Lipschitz
estimate, or from Lemma~\ref{lem:lebesgue-profile-estimates-paper}.  Combining
the two estimates proves the claim.
\end{proof}

We now fix the dense grids used in the annular argument.

\begin{definition}[Dense grids]
\label{def:dense-grids-paper}
Let \(\alpha>1\). For \(n\ge1\), let 
\(
\Gamma_n(\alpha)\subset[1,2]
\)
be a
finite set with mesh at most \(2^{-\alpha n}\):
\[
\sup_{\eta\in[1,2]}\operatorname{dist}(\eta,\Gamma_n(\alpha))
\le
2^{-\alpha n},
\]
and with cardinality
\[
\#\Gamma_n(\alpha)\le C2^{\alpha n}
\]
for an absolute constant \(C\).
For a function 
\(
f:\Gamma_n(\alpha)\to\mathbb C
\)
and 
\(
1\le r<\infty
\),
write
\[
\lVert f\rVert_{\ell^r(\Gamma_n(\alpha))}
=
\left(\sum_{\eta\in\Gamma_n(\alpha)}\lvert f(\eta)\rvert^r\right)^{1/r}.
\]
\end{definition}

\begin{lemma}\label{lem:dense-grid-passage-paper}
Let 
\(
0<\beta<1
\)
and 
\(
\varepsilon>0
\).
Assume that, almost surely, there is a
finite random constant \(A_\beta\) such that
\begin{equation}\label{eq:grid-centered-bound-assumption-paper}
\max_{\eta\in\Gamma_n(1+\beta+\varepsilon)}\lvert H_n(\eta) \rvert
\le
A_\beta2^{-\beta n}
\end{equation}
for all sufficiently large \(n\). Then, almost surely, there is a finite
random constant \(C_\beta\) such that
\[
\sup_{2^n\le\lvert\xi\rvert\le2^{n+1}}\lvert\wh\mu(\xi)\rvert
\le
C_\beta2^{-\beta n}
\]
for all sufficiently large \(n\).
\end{lemma}

\begin{proof}
Work on an event where 
\(
M<\infty
\),
where Lemma~\ref{lem:centered-profile-lipschitz-paper} holds for all 
\(
n,\eta,\eta'
\),
and
where \eqref{eq:grid-centered-bound-assumption-paper} holds eventually. Let
\(
2^n\le \xi\le 2^{n+1}
\),
write 
\(
\xi=2^n\eta
\)
with 
\(
\eta\in[1,2]
\),
and
choose 
\(
\eta_n\in\Gamma_n(1+\beta+\varepsilon)
\)
with
\[
\lvert\eta-\eta_n\rvert\le 2^{-(1+\beta+\varepsilon)n}.
\]
By Lemma~\ref{lem:centered-profile-lipschitz-paper},
\[
\lvert H_n(\eta)-H_n(\eta_n)\rvert
\le
C(M+1)2^n2^{-(1+\beta+\varepsilon)n}
=
C(M+1)2^{-(\beta+\varepsilon)n}.
\]
Hence, eventually,
\[
\lvert H_n(\eta) \rvert
\le
A_\beta2^{-\beta n}
+
C(M+1)2^{-(\beta+\varepsilon)n}
\le
C_\beta2^{-\beta n}.
\]
Moreover,
\[
\lvert\widehat\mu(2^n\eta)\rvert
\le
\lvert H_n(\eta) \rvert+\lvert\widehat{\mathcal L}(2^n\eta)\rvert
\le
C_\beta2^{-\beta n}+C2^{-n}
\le
C_\beta2^{-\beta n},
\]
where the last inequality uses \(\beta<1\). This proves the bound for positive frequencies.  Negative frequencies follow
from
\(
    |\widehat\mu(-\xi)|=|\widehat\mu(\xi)|.
\)
\end{proof}

\subsection{Vector-valued contraction}
\label{subsec:vector-valued-contraction-paper}

The following estimate is the only Banach-space input in the interval Fourier
argument.  It is a standard Rademacher-type contraction estimate in finite
\(\ell^r\) spaces; see, for instance, \cite{Pisier2016}.  We record it in the
form needed below, where \(1<q<2\) is the moment exponent and \(r\ge2\) is the
auxiliary grid-norm exponent.

\begin{lemma}[Vector-valued Rademacher contraction]
\label{lem:rademacher-type-ellr-paper}
Let \(q\in(1,2)\), let \(r\ge2\), and let \(G\) and \(J\) be finite sets. There
exists a constant 
\(
C_q<\infty
\), 
depending only on \(q\), such that for every
finite deterministic family 
\(
(z_j)_{j\in J}\subset\mathbb C^G
\),
\begin{equation}\label{eq:rademacher-type-ellr-paper}
\E_\varepsilon
\left\lVert
\sum_{j\in J}\varepsilon_jz_j
\right\rVert_{\ell^r(G)}^q
\le
C_q r^{q/2}
\sum_{j\in J}\lVert z_j \rVert_{\ell^r(G)}^q,
\end{equation}
where 
\(
(\varepsilon_j)_{j\in J}
\) 
are independent Rademacher signs.
\end{lemma}

\begin{proof}
For \(p\ge2\), the scalar Khintchine inequality implies
\[
\E_\varepsilon
\left|
\sum_{j\in J}\varepsilon_j c_j
\right|^p
\le
(C\sqrt p)^p
\left(\sum_{j\in J}\lvert c_j\rvert^2\right)^{p/2}
\]
for every finite complex family \((c_j)\). Applying this with \(p=r\), summing
over \(g\in G\), and using Minkowski's inequality in \(\ell^{r/2}(G)\), we
obtain
\[
\E_\varepsilon
\left\|
\sum_{j\in J}\varepsilon_jz_j
\right\|_{\ell^r(G)}^r
\le
(C\sqrt r)^r
\left(\sum_{j\in J}\lVert z_j \rVert_{\ell^r(G)}^2\right)^{r/2}.
\]
Since \(q\le r\), monotonicity of \(L^p\)-norms and the assumption \(q<2\)
yield
\[
\E_\varepsilon
\left\|
\sum_{j\in J}\varepsilon_jz_j
\right\|_{\ell^r(G)}^q
\le
C_q r^{q/2}
\left(\sum_{j\in J}\lVert z_j \rVert_{\ell^r(G)}^2\right)^{q/2}
\le
C_q r^{q/2}
\sum_{j\in J}\lVert z_j \rVert_{\ell^r(G)}^q.
\]
This proves \eqref{eq:rademacher-type-ellr-paper}.
\end{proof}

\begin{proposition}\label{prop:vector-valued-contraction-paper}
Let \(q\in(1,2)\), let \(r\ge2\), let \(G\) be a finite set, and let
\(J\) be a finite index set. Let \((Z_j)_{j\in J}\) be independent mean-zero
random vectors in \(\mathbb C^G\) such that
\[
\E\lVert Z_j \rVert_{\ell^r(G)}^q<\infty
\qquad(j\in J).
\]
Let \(\mathcal A\) be a \(\sigma\)-field independent of \(\sigma(Z_j:j\in J)\),
and let \((a_j)_{j\in J}\) be complex-valued \(\mathcal A\)-measurable random
variables. Then, with both sides understood in the extended sense,
\begin{equation}\label{eq:vector-valued-contraction-paper}
\E\left[
\left\|
\sum_{j\in J}a_jZ_j
\right\|_{\ell^r(G)}^q
\middle|\mathcal A
\right]
\le
C_q r^{q/2}
\sum_{j\in J}\lvert a_j\rvert^q
\E\lVert Z_j \rVert_{\ell^r(G)}^q
\end{equation}
almost surely.
\end{proposition}

\begin{proof}
Condition on \(\mathcal A\).  The coefficients \(a_j\) are then deterministic,
and the conditional law of the independent mean-zero vectors \(Z_j\) is their
original law.  The standard symmetrization argument yields
\[
\E\left\|\sum_j a_jZ_j\right\|_{\ell^r(G)}^q
\le
\E\E_\varepsilon
\left\|\sum_j\varepsilon_j a_j(Z_j-Z_j')\right\|_{\ell^r(G)}^q,
\]
where \((Z_j')\) is an independent copy of \((Z_j)\). Lemma~\ref{lem:rademacher-type-ellr-paper}
and
\[
\|a_j(Z_j-Z_j')\|_{\ell^r(G)}^q
\le
2^{q-1}|a_j|^q
\bigl(\|Z_j\|_{\ell^r(G)}^q+\|Z_j'\|_{\ell^r(G)}^q\bigr)
\]
then yield \eqref{eq:vector-valued-contraction-paper}.
\end{proof}

\begin{corollary}\label{cor:contraction-with-cascade-coefficients-paper}
Let 
\(
q\in(1,2)
\),
\(r\ge2\), and let \(G\) be a finite set. Fix \(m\ge0\).
Suppose that the family
\[
(Z_u)_{\lvert u \rvert=m}
\]
is independent of \(\mathcal F^X_m\), that the vectors \(Z_u\) are independent
and coordinatewise mean zero, and that all \(Z_u\) have the same distribution
as a random vector \(Z\) with
\[
\E\lVert Z \rVert_{\ell^r(G)}^q<\infty.
\]
Then
\begin{equation}\label{eq:contraction-with-cascade-coefficients-paper}
\E\left[
\left\|
\sum_{\lvert u \rvert=m}L_u Z_u
\right\|_{\ell^r(G)}^q
\middle|\mathcal F^X_m
\right]
\le
C_q r^{q/2}
\left(\sum_{\lvert u \rvert=m}L_u^q\right)
\E\lVert Z \rVert_{\ell^r(G)}^q.
\end{equation}
Consequently,
\begin{equation}\label{eq:contraction-with-rho-paper}
\E\left[
\left\|
\sum_{\lvert u \rvert=m}L_u Z_u
\right\|_{\ell^r(G)}^q
\right]
\le
C_q r^{q/2}\rho(q)^m
\E\lVert Z \rVert_{\ell^r(G)}^q.
\end{equation}
\end{corollary}

\begin{proof}
Apply Proposition~\ref{prop:vector-valued-contraction-paper} with
\(a_u=L_u\) and 
\(
\mathcal A=\mathcal F^X_m
\) 
to obtain
\eqref{eq:contraction-with-cascade-coefficients-paper}. Taking expectations
and using Lemma~\ref{lem:vector-q-mass-identity-paper},
\[
\E\sum_{\lvert u \rvert=m}L_u^q=\rho(q)^m
\] 
gives 
\eqref{eq:contraction-with-rho-paper}.
\end{proof}

\begin{remark}\label{rem:role-of-contraction-estimate-paper}

In the mesoscopic decomposition, the centered profile on a dense grid is written
as a sum of independent mean-zero descendant contributions after conditioning
on \(\mathcal F^X_m\).  Corollary~\ref{cor:contraction-with-cascade-coefficients-paper}
then produces the random factor
\[
    \sum_{|u|=m}L_u^q,
\]
whose expectation is \(\rho(q)^m\).  This is the point at which the vector
moment profile enters the Fourier argument; no independence between siblings is
used.
\end{remark}

\subsection{Mesoscopic decomposition}
\label{subsec:mesoscopic-decomposition}

We now turn the star equation into an exact recursion for the centered
profiles.  After a mesoscopic generation is exposed, the descendant centered
profiles are independent and mean zero, while the remaining term is a
lower-frequency forcing term coming from the Lebesgue centering.

\begin{lemma}
\label{lem:centered-profile-mean-zero-paper}
Assume the vector law is energy-admissible.  Then, for every \(\xi\in\R\),
\[
    \E[\widehat{\mu}(\xi)]
    =
    \widehat{\mathcal L}(\xi).
\]
Consequently, for every \(n\ge0\) and every \(\eta\in[1,2]\),
\[
    \E[H_n(\eta)]=0.
\]
The same statements hold for every descendant copy \(\mu^{(u)}\).
\end{lemma}

\begin{proof}
For each \(n\), the balanced condition gives
\(
    \E\mu_n=\mathcal L,
\)
and hence
\(
    \E[\widehat{\mu_n}(\xi)]=\widehat{\mathcal L}(\xi).
\)
By Theorem~\ref{thm:vector-limiting-measure-construction-paper},
\(\mu_n\weak\mu\) almost surely, so
\(
    \widehat{\mu_n}(\xi)\to\widehat{\mu}(\xi)
\)
almost surely.  Energy-admissibility and
Theorem~\ref{thm:vector-KP-nondegeneracy-paper} imply uniform integrability of
the total-mass martingale and \(\E M=1\).  Since
\(
    |\widehat{\mu_n}(\xi)|\le M_n,
\)
the family \(\{\widehat{\mu_n}(\xi):n\ge0\}\) is uniformly integrable.  Therefore
the convergence holds in \(L^1\), and
\[
    \E[\widehat{\mu}(\xi)]
    =
    \lim_{n\to\infty}\E[\widehat{\mu_n}(\xi)]
    =
    \widehat{\mathcal L}(\xi).
\]
Thus \(\E[H_n(\eta)]=0\), by the definition of \(H_n\).  The descendant
statement follows because each descendant cascade has the same law as the
original one.
\end{proof}

Let \(0<m\le k\).  Applying
Lemma~\ref{lem:star-equation-fourier-paper} with \(n=k\), we obtain the exact
decomposition
\[
    H_k(\eta)=U_{k,m}(\eta)+R_{k,m}(\eta),
    \qquad \eta\in[1,2],
\]
where
\begin{equation}\label{eq:Uk-mesoscopic-paper}
    U_{k,m}(\eta)
    =
    \sum_{|u|=m}
    L_u e^{-2\pi i a_u2^k\eta}H^{(u)}_{k-m}(\eta),
\end{equation}
and
\begin{equation}\label{eq:Rk-mesoscopic-paper}
    R_{k,m}(\eta)
    =
    \widehat{\mathcal L}(2^{k-m}\eta)
    \sum_{|u|=m}
    (L_u-2^{-m})e^{-2\pi i a_u2^k\eta}.
\end{equation}
The term \(U_{k,m}\) is the centered descendant contribution.  The term
\(R_{k,m}\) is deterministic in phase but random in its coarse coefficients; it
is the forcing term created by centering with respect to Lebesgue measure.
For \(0<m\le k\), we call
\[
    H_k=U_{k,m}+R_{k,m}
\]
the mesoscopic decomposition of the centered profile at depth \(m\).

\begin{lemma}
\label{lem:mesoscopic-descendant-independence-paper}
Fix \(0<m\le k\), and let \(G\subset[1,2]\) be finite.  For \(|u|=m\), set
\[
    Z_u(\eta)
    =
    e^{-2\pi i a_u2^k\eta}H^{(u)}_{k-m}(\eta),
    \qquad \eta\in G.
\]
Conditional on \(\mathcal F^X_m\), the family
\[
    \{Z_u:|u|=m\}
\]
is independent.  Each \(Z_u\) is independent of \(\mathcal F^X_m\), has the law
of
\[
    \{H_{k-m}(\eta):\eta\in G\}
\]
up to deterministic unimodular phases, and satisfies
\[
    \E[Z_u\mid\mathcal F^X_m]=0
\]
coordinatewise.
\end{lemma}

\begin{proof}
For \(|u|=m\), the profile \(H^{(u)}_{k-m}\) is constructed entirely from the
descendant environment below \(u\).  These descendant environments are
independent of \(\mathcal F^X_m\), and are independent for distinct vertices
\(u\) of generation \(m\).  Multiplication by
\(e^{-2\pi i a_u2^k\eta}\) is deterministic and unimodular, and hence does not
affect independence or norms.  Finally,
Lemma~\ref{lem:centered-profile-mean-zero-paper} gives
\[
    \E[H^{(u)}_{k-m}(\eta)]=0
\]
for each \(\eta\in G\), which proves the coordinatewise conditional
mean-zero property.
\end{proof}

The forcing term is controlled by the mass martingale and the Lebesgue profile estimate.

\begin{lemma}
\label{lem:forcing-term-bound-paper}
Let \(q\in(1,2)\) and assume \(\rho(q)<1\).  Then there exists
\(C_q<\infty\) such that, for all \(0<m\le k\), all \(r\ge2\), and every finite
set \(G\subset[1,2]\),
\begin{equation}\label{eq:forcing-term-bound-paper}
    \E\|R_{k,m}\|_{\ell^r(G)}^q
    \le
    C_q\,2^{-q(k-m)}(\#G)^{q/r}.
\end{equation}
\end{lemma}

\begin{proof}
By Lemma~\ref{lem:lebesgue-profile-estimates-paper},
\[
    |\widehat{\mathcal L}(2^{k-m}\eta)|
    \le C2^{-(k-m)}
    \qquad(\eta\in[1,2]).
\]
Also,
\[
\left|
\sum_{|u|=m}(L_u-2^{-m})e^{-2\pi i a_u2^k\eta}
\right|
\le
\sum_{|u|=m}L_u+\sum_{|u|=m}2^{-m}
=
M_m+1.
\]
Hence
\[
    \|R_{k,m}\|_{\ell^r(G)}
    \le
    C2^{-(k-m)}(M_m+1)(\#G)^{1/r}.
\]
Taking \(q\)-th moments gives
\[
    \E\|R_{k,m}\|_{\ell^r(G)}^q
    \le
    C2^{-q(k-m)}(\#G)^{q/r}\E[(M_m+1)^q].
\]
Since \(\rho(q)<1\), Lemma~\ref{lem:terminal-mass-moment-criterion} gives
\[
    \sup_{m\ge0}\E[M_m^q]<\infty.
\]
Thus \(\sup_{m\ge0}\E[(M_m+1)^q]<\infty\), and the estimate follows.
\end{proof}

\subsection{The grid-norm recursion}
\label{subsec:grid-norm-recursion-paper}

We now combine the mesoscopic decomposition with the vector-valued contraction
estimate.  The recursion is formulated with a grid envelope: after one
mesoscopic step, the descendant profile has Fourier scale \(k-m\), but it is
still evaluated on a grid chosen for the ambient annular scale \(n\).

Fix \(\alpha>1\), and set
\[
    r_n=\max\{2,n\}.
\]
Let \(\mathcal G_n(\alpha)\) be the collection of all finite sets
\(G\subset[1,2]\) satisfying
\begin{equation}\label{eq:grid-envelope-cardinality-paper}
    \#G\le C_\alpha 2^{\alpha n},
\end{equation}
where \(C_\alpha\) is chosen large enough so that
\(\Gamma_n(\alpha)\in\mathcal G_n(\alpha)\).

Fix \(q\in(1,2)\).  For \(0\le k\le n\), define
\begin{equation}\label{eq:grid-envelope-Akn-paper}
    A^{(q,\alpha)}_{k,n}
    =
    \sup_{G\in\mathcal G_n(\alpha)}
    \E\left[\|H_k\|_{\ell^{r_n}(G)}^q\right].
\end{equation}
When \(q\) and \(\alpha\) are fixed, we write \(A_{k,n}\) for
\(A^{(q,\alpha)}_{k,n}\).

\begin{lemma}\label{lem:crude-grid-norm-bound-paper}
Let \(q\in(1,2)\) and assume \(\rho(q)<1\).  For every \(\alpha>1\) there exists 
\(
C_{q,\alpha}<\infty
\) 
such that 
\(
A_{k,n}\le C_{q,\alpha}
\) 
for all 
\(
0\le k\le n
\).
\end{lemma}

\begin{proof}
Since \(\rho(q)<1\), Lemma~\ref{lem:terminal-mass-moment-criterion}
gives
\[
    \E[M^q]<\infty.
\]
For every \(k\ge0\) and \(\eta\in[1,2]\),
\[
    |H_k(\eta)|
    \le |\widehat{\mu}(2^k\eta)|+
        |\widehat{\mathcal L}(2^k\eta)|
    \le M+1.
\]
Hence, for \(G\in\mathcal G_n(\alpha)\),
\[
    \E\|H_k\|_{\ell^{r_n}(G)}^q
    \le
    (\#G)^{q/r_n}\E[(M+1)^q].
\]
By \eqref{eq:grid-envelope-cardinality-paper},
\[
    (\#G)^{q/r_n}
    \le
    (C_\alpha 2^{\alpha n})^{q/\max\{2,n\}}
    \le C_{q,\alpha}.
\]
Thus \(A_{k,n}\le C_{q,\alpha}\) for all \(0\le k\le n\).
\end{proof}

 \begin{proposition}[Grid-norm recursion]
\label{prop:grid-norm-recursion-paper}
Let \(q\in(1,2)\) and assume \(\rho(q)<1\).  Fix \(\alpha>1\).  Then there
exists \(C_{q,\alpha}<\infty\) such that, for all integers
\(0<m\le k\le n\),
\begin{equation}\label{eq:grid-norm-recursion-paper}
    A_{k,n}
    \le
    C_{q,\alpha} r_n^{q/2}\rho(q)^m A_{k-m,n}
    +
    C_{q,\alpha}2^{-q(k-m)}.
\end{equation}
\end{proposition}

\begin{proof}
Fix \(G\in\mathcal G_n(\alpha)\).  By the mesoscopic decomposition,
\[
    H_k=U_{k,m}+R_{k,m}
\]
on \(G\).  Therefore
\[
    \|H_k\|_{\ell^{r_n}(G)}^q
    \le
    2^{q-1}\|U_{k,m}\|_{\ell^{r_n}(G)}^q
    +
    2^{q-1}\|R_{k,m}\|_{\ell^{r_n}(G)}^q.
\]
We estimate the descendant term and the forcing term separately.

For the descendant term, set
\[
    Z_u(\eta)
    =
    e^{-2\pi i a_u2^k\eta}H^{(u)}_{k-m}(\eta),
    \qquad \eta\in G,\quad |u|=m.
\]
By Lemma~\ref{lem:mesoscopic-descendant-independence-paper}, conditional on
\(\mathcal F^X_m\), the vectors \(Z_u\) are independent, mean zero, and
independent of \(\mathcal F^X_m\).  Moreover,
\[
    \|Z_u\|_{\ell^{r_n}(G)}
    =
    \|H^{(u)}_{k-m}\|_{\ell^{r_n}(G)}.
\]
Applying  Corollary~\ref{cor:contraction-with-cascade-coefficients-paper}, gives
\[
    \E\left[
        \|U_{k,m}\|_{\ell^{r_n}(G)}^q
        \,\middle|\, \mathcal F^X_m
    \right]
    \le
    C_q r_n^{q/2}
    \left(\sum_{|u|=m}L_u^q\right)
    \E\|H_{k-m}\|_{\ell^{r_n}(G)}^q .
\]
Since \(G\in\mathcal G_n(\alpha)\), the last expectation is at most
\(A_{k-m,n}\).  Taking expectations and using
\[
    \E\sum_{|u|=m}L_u^q=\rho(q)^m
\]
yields
\[
    \E\|U_{k,m}\|_{\ell^{r_n}(G)}^q
    \le
    C_q r_n^{q/2}\rho(q)^m A_{k-m,n}.
\]

For the forcing term, Lemma~\ref{lem:forcing-term-bound-paper} gives
\[
    \E\|R_{k,m}\|_{\ell^{r_n}(G)}^q
    \le
    C_q2^{-q(k-m)}(\#G)^{q/r_n}.
\]
By \eqref{eq:grid-envelope-cardinality-paper},
\[
    (\#G)^{q/r_n}\le C_{q,\alpha}.
\]
Thus
\[
    \E\|R_{k,m}\|_{\ell^{r_n}(G)}^q
    \le
    C_{q,\alpha}2^{-q(k-m)}.
\]
Combining the two estimates and taking the supremum over
\(G\in\mathcal G_n(\alpha)\) proves
\eqref{eq:grid-norm-recursion-paper}.

\end{proof}

\begin{remark}\label{rem:ambient-index-grid-envelope-paper}
The grid \(G\) used to control the annulus at scale \(n\) has cardinality of
order \(2^{\alpha n}\).  After one mesoscopic step, the descendant profile has
Fourier scale \(k-m\), but it is still evaluated on the same ambient grid
\(G\).  This is why the envelope \(A_{k,n}\) has two indices: \(k\) records the
Fourier scale of the profile, while \(n\) records the size of the grid on which
the profile is measured.
\end{remark}

 \subsection{The ladder iteration estimate}
\label{subsec:ladder-estimate-paper}

We now iterate the grid-norm recursion from
Proposition~\ref{prop:grid-norm-recursion-paper}.  Fix
\(q\in(1,2)\) with \(\rho(q)<1\), and define
\begin{equation}\label{eq:beta-X-q-paper}
    \beta_X(q)
    =
    -\frac1q\log_2\rho(q)>0.
\end{equation}
Thus \(\rho(q)=2^{-q\beta_X(q)}\).

\begin{lemma}
\label{lem:beta-X-upper-bound-paper}
For every \(q\in(1,2)\) with \(\rho(q)<1\),
\[
    0<\beta_X(q)<1.
\]
In fact,
\(
    \beta_X(q)\le \frac{q-1}{q}<\frac12 .
\)
\end{lemma}

\begin{proof}
The positivity follows from \(\rho(q)<1\).  For the upper bound, the balanced
condition and Jensen's inequality give
\[
    \E[X_i^q]\ge (\E X_i)^q=2^{-q},
    \qquad i=0,1.
\]
Hence
\[
    \rho(q)=\E[X_0^q+X_1^q]\ge 2^{1-q}.
\]
Therefore
\[
    -\log_2\rho(q)\le q-1,
\]
and so
\[
    \beta_X(q)
    =
    -\frac1q\log_2\rho(q)
    \le
    \frac{q-1}{q}
    <
    \frac12,
\]
because \(q<2\).
\end{proof}

The ladder proof uses a decay exponent strictly smaller than \(\beta_X(q)\). To
avoid conflict with the dense-grid parameter, we denote the Fourier decay
exponent in this subsection by \(\beta\), and the grid parameter by \(\alpha_{\mathrm{gr}}\).

\begin{proposition}[Ladder iteration]
\label{prop:ladder-estimate-paper}
Fix \(q\in(1,2)\) with \(\rho(q)<1\), let
\(0<\beta<\beta_X(q)\), and fix \(\alpha_{\mathrm{gr}}>1\).  Then there exist
constants
\[
    C_{\mathrm{lad}}<\infty,\qquad
    \tau_{\mathrm{lad}}<\infty,\qquad
    n_0\ge1,
\]
depending on \(q,\beta,\alpha_{\mathrm{gr}}\) and on the law of the cascade,
such that, for every \(n\ge n_0\) and every \(0\le k\le n\),
\begin{equation}\label{eq:ladder-iteration-bound-paper}
    A^{(q,\alpha_{\mathrm{gr}})}_{k,n}
    \le
    C_{\mathrm{lad}} n^{\tau_{\mathrm{lad}}}2^{-q\beta k}.
\end{equation}
In particular,
\begin{equation}\label{eq:ladder-terminal-bound-paper}
    A^{(q,\alpha_{\mathrm{gr}})}_{n,n}
    \le
    C_{\mathrm{lad}} n^{\tau_{\mathrm{lad}}}2^{-q\beta n}
    \qquad(n\ge n_0).
\end{equation}
\end{proposition}

\begin{proof}
Write \(A_{k,n}=A^{(q,\alpha_{\mathrm{gr}})}_{k,n}\).  Choose
\(\gamma\in(\beta,1)\), and set
\[
    k'=\lceil \gamma k\rceil .
\]
Since \(\rho(q)=2^{-q\beta_X(q)}\), Proposition~\ref{prop:grid-norm-recursion-paper}
gives a constant \(C_0<\infty\) such that, whenever \(k-k'\ge1\),
\begin{equation}\label{eq:ladder-one-step-paper}
    A_{k,n}
    \le
    C_0 r_n^{q/2}2^{-q\beta_X(q)(k-k')}A_{k',n}
    +
    C_0 2^{-qk'} .
\end{equation}
Since \(r_n\le Cn\), we may absorb the harmless constant and use
\(r_n^{q/2}\le Cn^{q/2}\).

We have
\[
    k-k'
    \ge (1-\gamma)k-1.
\]
Thus there exists \(c_\gamma>0\) such that
\[
    k-k'\ge c_\gamma k
\]
for all sufficiently large \(k\).  Choose \(K>0\) so large that
\begin{equation}\label{eq:ladder-K-choice-paper}
    q(\beta_X(q)-\beta)c_\gamma K>\frac q2+2 .
\end{equation}
Then, whenever \(k\ge K\log_2 n\) and \(n\) is sufficiently large,
\[
    n^{q/2}2^{-q(\beta_X(q)-\beta)(k-k')}
    \le n^{-2}.
\]

We prove \eqref{eq:ladder-iteration-bound-paper} by induction on \(k\), with
\(n\) fixed.  For \(0\le k\le K\log_2 n\), Lemma~\ref{lem:crude-grid-norm-bound-paper}
gives \(A_{k,n}\le C_1\).  Choose
\[
    \tau_{\mathrm{lad}}>q\beta K
\]
and then increase \(C_{\mathrm{lad}}\) and \(n_0\) so that
\[
    C_1
    \le
    C_{\mathrm{lad}}n^{\tau_{\mathrm{lad}}}2^{-q\beta k}
\]
for all \(0\le k\le K\log_2 n\) and \(n\ge n_0\).

Now assume \(k>K\log_2 n\), and suppose the bound has been proved at all
smaller scales.  For \(n\ge n_0\), we have \(k-k'\ge1\) and \(k'<k\).  Applying
\eqref{eq:ladder-one-step-paper} and the induction hypothesis at \(k'\), we get
\[
\begin{aligned}
A_{k,n}
&\le
C_0 C_{\mathrm{lad}} n^{\tau_{\mathrm{lad}}+q/2}
2^{-q\beta_X(q)(k-k')}2^{-q\beta k'}
+
C_0 2^{-qk'}  \\
&=
C_0 C_{\mathrm{lad}} n^{\tau_{\mathrm{lad}}}
2^{-q\beta k}
\left[
    n^{q/2}2^{-q(\beta_X(q)-\beta)(k-k')}
\right]
+
C_0 2^{-qk'} .
\end{aligned}
\]
By the choice of \(K\), the bracketed factor is at most \(n^{-2}\).  After
increasing \(n_0\), the first term is therefore at most
\[
    \frac12 C_{\mathrm{lad}}n^{\tau_{\mathrm{lad}}}2^{-q\beta k}.
\]
For the forcing term, since \(k'\ge \gamma k\) and \(\gamma>\beta\),
\[
    2^{-qk'}\le 2^{-q\beta k}.
\]
Increasing \(C_{\mathrm{lad}}\), if necessary, gives
\[
    C_0 2^{-qk'}
    \le
    \frac12 C_{\mathrm{lad}}n^{\tau_{\mathrm{lad}}}2^{-q\beta k}.
\]
Combining the two estimates closes the induction.  The terminal bound
\eqref{eq:ladder-terminal-bound-paper} is the case \(k=n\).
\end{proof}

\begin{corollary}\label{cor:dense-grid-decay-centered-profiles-paper}
 
Fix \(q\in(1,2)\) with \(\rho(q)<1\), and let
\(0<\beta<\beta_X(q)\).  Choose a grid parameter
\(\alpha_{\mathrm{gr}}>1\).  Let \(\Gamma_n(\alpha_{\mathrm{gr}})\) be the dense
grid from Definition~\ref{def:dense-grids-paper}.  Then, almost surely,
\begin{equation}\label{eq:grid-pointwise-decay-paper}
    \max_{\eta\in\Gamma_n(\alpha_{\mathrm{gr}})}
    |H_n(\eta)|
    \le
    2^{-\beta n}
\end{equation}
for all sufficiently large \(n\).
\end{corollary}

\begin{proof}
Choose \(\beta'\) with
\[
    \beta<\beta'<\beta_X(q).
\]
By Proposition~\ref{prop:ladder-estimate-paper}, applied with \(\beta'\), and
since \(\Gamma_n(\alpha_{\mathrm{gr}})\in\mathcal G_n(\alpha_{\mathrm{gr}})\),
\[
    \E\|H_n\|_{\ell^{r_n}(\Gamma_n(\alpha_{\mathrm{gr}}))}^q
    \le
    C n^{\tau}2^{-q\beta'n}
\]
for all sufficiently large \(n\).  If
\[
    \max_{\eta\in\Gamma_n(\alpha_{\mathrm{gr}})}|H_n(\eta)|>2^{-\beta n},
\]
then
\[
    \|H_n\|_{\ell^{r_n}(\Gamma_n(\alpha_{\mathrm{gr}}))}^q
    >
    2^{-q\beta n}.
\]
Markov's inequality gives
\[
    \mathbb P\left(
    \max_{\eta\in\Gamma_n(\alpha_{\mathrm{gr}})}
    |H_n(\eta)|>2^{-\beta n}
    \right)
    \le
    C n^\tau 2^{-q(\beta'-\beta)n}.
\]
The right-hand side is summable in \(n\), so Borel--Cantelli proves
\eqref{eq:grid-pointwise-decay-paper}.
\end{proof}

\subsection{Annular Fourier decay and optimization over q}
\label{subsec:annular-decay-optimization-paper}

We now pass from dense-grid bounds for the centered profiles to annular Fourier
decay for \(\mu\), and then optimize over the admissible moment exponent \(q\).

\begin{theorem}
\label{thm:q-level-annular-decay-paper}
Fix \(q\in(1,2)\) with \(\rho(q)<1\), and let
\[
    0<\beta<\beta_X(q)
    =
    -\frac1q\log_2\rho(q).
\]
Then, almost surely, there exists a finite random constant
\(C_\beta(\omega)<\infty\) such that
\begin{equation}\label{eq:q-level-annular-decay-paper}
    \sup_{2^n\le |\xi|\le 2^{n+1}}
    |\widehat{\mu}(\xi)|
    \le
    C_\beta(\omega)2^{-\beta n}
\end{equation}
for all sufficiently large \(n\).
\end{theorem}

\begin{proof}
By Lemma~\ref{lem:beta-X-upper-bound-paper}, \(\beta<1\).  Choose
\(\varepsilon>0\) and set
\[
    \alpha_{\mathrm{gr}}=1+\beta+\varepsilon.
\]
By
Corollary~\ref{cor:dense-grid-decay-centered-profiles-paper},  almost surely,
\[
    \max_{\eta\in\Gamma_n(\alpha_{\mathrm{gr}})}
    |H_n(\eta)|
    \le
    2^{-\beta n}
\]
for all sufficiently large \(n\). Lemma~\ref{lem:dense-grid-passage-paper} 
upgrades this dense-grid estimate to the annular bound
\eqref{eq:q-level-annular-decay-paper} for positive frequencies.  Negative
frequencies follow from
\(
    |\widehat{\mu}(-\xi)|=|\widehat{\mu}(\xi)|.
\)
\end{proof}

\begin{corollary}\label{cor:q-level-fourier-dimension-lower-bound-paper}

If \(q\in(1,2)\) and \(\rho(q)<1\), then
\[
    \dim_{\mathrm F}(\mu)
    \ge
    2\beta_X(q)
    =
    -\frac2q\log_2\rho(q)
\]
almost surely on \(\{M>0\}\).
\end{corollary}

\begin{proof}
For every \(0<\beta<\beta_X(q)\), Theorem~\ref{thm:q-level-annular-decay-paper}
gives
\[
    |\widehat{\mu}(\xi)|
    =
    O(|\xi|^{-\beta})
    \qquad (|\xi|\to\infty).
\]
Thus \(2\beta\) is an admissible Fourier-dimension exponent.  Letting
\(\beta\uparrow\beta_X(q)\) gives the claim.
\end{proof}

\begin{theorem}[Vector Fourier lower bound]
\label{thm:vector-fourier-lower-bound}
Assume that the dyadic vector-weight law is energy-admissible.  Then
\[
    \dim_{\mathrm F}(\mu)\ge D_{\mathrm E}(X)
\]
almost surely on \(\{M>0\}\).  Equivalently, on the same event, for every
\(0<\sigma<D_{\mathrm E}(X)\), there exists a finite random constant
\(C_\sigma(\omega)<\infty\) such that
\begin{equation}\label{eq:vector-fourier-decay-sigma-paper}
    |\widehat{\mu}(\xi)|
    \le
    C_\sigma(\omega)|\xi|^{-\sigma/2}
\end{equation}
for all sufficiently large \(|\xi|\).
\end{theorem}

\begin{proof}
If \(D_{\mathrm E}(X)=0\), there is nothing to prove.  Fix
\(0<\sigma<D_{\mathrm E}(X)\).  By
Proposition~\ref{prop:vector-DE-variational-formula-paper},  there exists
\(q\in(1,2)\) such that
\[
    \rho(q)<1,
    \qquad
    -\frac2q\log_2\rho(q)>\sigma.
\]
Equivalently, \(\beta_X(q)>\sigma/2\).  Choose
\[
    \frac{\sigma}{2}<\beta<\beta_X(q).
\]
Theorem~\ref{thm:q-level-annular-decay-paper} gives, almost surely, a finite
random constant \(C_\beta(\omega)\) such that
\[
    |\widehat{\mu}(\xi)|
    \le
    C_\beta(\omega)|\xi|^{-\beta}
    \le
    C_\beta(\omega)|\xi|^{-\sigma/2}
\]
for all sufficiently large \(|\xi|\).  This proves
 \eqref{eq:vector-fourier-decay-sigma-paper}.

Intersecting the corresponding probability-one events over rational
\(\sigma\in(0,D_{\mathrm E}(X))\) gives
\[
    \dim_{\mathrm F}(\mu)\ge D_{\mathrm E}(X)
\]
almost surely on \(\{M>0\}\).
\end{proof}

\begin{remark}\label{rem:comparison-LQT-scalar}
In the scalar canonical case, related Fourier-dimension lower bounds are also
available from the general theory of Lin--Qiu--Tan
\cite[Theorem~1.6]{LinQiuTan2025}.  We do not use that input here.  The proof
above is formulated directly for the dyadic vector cascade: sibling weights may
be dependent, and the argument uses the vector moment profile
\[
    \rho(q)=\E(X_0^q+X_1^q)
\]
together with independence of descendant environments.
\end{remark}

\section{The exact interval theorem}
\label{sec:exact-interval-theorem}

We now combine the tree-cylinder energy theorem with the dense-grid Fourier
lower bound.

\begin{theorem}[Exact interval formula]
\label{thm:exact-interval-formula-paper}
Assume that the dyadic vector-weight law is energy-admissible.  Let \(\mu\) be
the limiting dyadic vector cascade measure on \([0,1]\), and set
\(M=\mu([0,1])\).  Then, almost surely on \(\{M>0\}\),
\begin{equation}\label{eq:exact-interval-vector-theorem-paper}
    \dim_{\mathrm F}(\mu)
    =
    \dim_{\mathrm E}(\mu)
    =
    \dim_2(\mu)
    =
    D_{\mathrm E}(X).
\end{equation}
In particular, on the same event, if \(D_{\mathrm E}(X)>0\), then for every
\(0<\sigma<D_{\mathrm E}(X)\) there is a finite random constant
\(C_\sigma(\omega)<\infty\) such that
\begin{equation}\label{eq:exact-interval-vector-decay-paper}
    |\widehat{\mu}(\xi)|
    \le
    C_\sigma(\omega)|\xi|^{-\sigma/2}
\end{equation}
for all sufficiently large \(|\xi|\).  If \(D_{\mathrm E}(X)=0\), the decay
assertion is vacuous, and all three dimensions in
\eqref{eq:exact-interval-vector-theorem-paper} are equal to \(0\).
\end{theorem}

\begin{proof}
On \(\{M>0\}\), Theorem~\ref{thm:vector-energy-theorem-paper} gives
\[
    \dim_{\mathrm E}(\mu)=\dim_2(\mu)=D_{\mathrm E}(X),
\]
while Theorem~\ref{thm:vector-fourier-lower-bound} gives
\[
    \dim_{\mathrm F}(\mu)\ge D_{\mathrm E}(X).
\]
Since \(\mu\) is then a nonzero finite Borel measure on \([0,1]\),
Proposition~\ref{prop:fourier-energy-upper-bound-paper}  gives
\[
    \dim_{\mathrm F}(\mu)\le \dim_{\mathrm E}(\mu).
\]
Combining the three inequalities proves
\eqref{eq:exact-interval-vector-theorem-paper}.  The decay estimate
\eqref{eq:exact-interval-vector-decay-paper} is exactly the decay formulation in
Theorem~\ref{thm:vector-fourier-lower-bound}.
\end{proof}

\begin{corollary}[Canonical scalar formula]
\label{cor:scalar-dyadic-interval-formula-paper}
Let \(W\ge0\) satisfy
\[
    \E W=1,\qquad
    \E[W\log_2^+W]<\infty,\qquad
    \E[W\log_2 W]<1.
\]
Let \(W_0,W_1\) be independent copies of \(W\), define
\[
    X_i=\frac{W_i}{2},
    \qquad i=0,1,
\]
and let \(\mu\) be the corresponding scalar dyadic Mandelbrot cascade measure
on \([0,1]\).  Set \(M=\mu([0,1])\).  Then, almost surely on \(\{M>0\}\),
\begin{equation}\label{eq:scalar-dyadic-interval-formula-paper}
    \dim_{\mathrm F}(\mu)
    =
    \dim_{\mathrm E}(\mu)
    =
    \dim_2(\mu)
    =
    D^+(W),
\end{equation}
where
\begin{equation}\label{eq:scalar-Dplus-paper}
D^+(W)
=
\sup_{1<q<2}
\max\left\{
0,\,
2-\frac2q\bigl(1+\log_2\E[W^q]\bigr)
\right\},
\end{equation}
with the convention that the corresponding term is interpreted as \(0\)
whenever \(\E[W^q]=\infty\).  In particular, on the same event, if
\(D^+(W)>0\), then for every \(0<\sigma<D^+(W)\) there are finite random
constants \(C_\sigma(\omega),R_\sigma(\omega)<\infty\) such that
\begin{equation}\label{eq:scalar-dyadic-interval-decay-paper}
    |\widehat{\mu}(\xi)|
    \le
    C_\sigma(\omega)|\xi|^{-\sigma/2}
    \qquad (|\xi|\ge R_\sigma(\omega)).
\end{equation}
If \(D^+(W)=0\), the decay assertion is vacuous.
\end{corollary}

\begin{proof}
The vector law \(X=(W_0/2,W_1/2)\) is balanced.  Moreover,
\[
    \E[X_0\log_2^+X_0+X_1\log_2^+X_1]<\infty
\]
is equivalent to
\[
    \E[W\log_2^+W]<\infty,
\]
and
\[
    \E[X_0\log_2X_0+X_1\log_2X_1]
    =
    \E[W\log_2W]-1.
\]
Thus the scalar assumptions imply energy-admissibility of the associated vector
law.  Also,
\[
    \rho(q)=2^{1-q}\E[W^q],
\]
and Corollary~\ref{cor:scalar-reduction-vector-profile-paper} gives
\(
    D_{\mathrm E}(X)=D^+(W).
\)
Theorem~\ref{thm:exact-interval-formula-paper} therefore gives
\[
    \dim_{\mathrm F}(\mu)=\dim_{\mathrm E}(\mu)=\dim_2(\mu)
    =
    D_{\mathrm E}(X)=D^+(W)
\]
almost surely on \(\{M>0\}\), and its decay formulation gives
\eqref{eq:scalar-dyadic-interval-decay-paper}.
\end{proof}

\begin{corollary}[Pure heavy-tail case]
\label{cor:interval-heavy-tail-zero-branch-paper}
In the scalar setting of
Corollary~\ref{cor:scalar-dyadic-interval-formula-paper}, assume further
that
\[
    \E[W^q]=\infty
    \qquad\text{for every }q>1.
\]
Then \(D^+(W)=0\), and hence
\[
    \dim_{\mathrm F}(\mu)
    =
    \dim_{\mathrm E}(\mu)
    =
    \dim_2(\mu)
    =
    0
\]
almost surely on \(\{M>0\}\).
\end{corollary}

\begin{proof}
Every term in the definition of \(D^+(W)\) is interpreted as \(0\), and hence
\(D^+(W)=0\).  The conclusion follows from
Corollary~\ref{cor:scalar-dyadic-interval-formula-paper}.
\end{proof}

\begin{corollary}[Positive Fourier dimension criterion]
\label{cor:interval-positive-dimensional-branch-paper}
Assume that the dyadic vector-weight law is energy-admissible.  Then
\(D_{\mathrm E}(X)>0\) if and only if
\[
    \rho(q)<1
    \qquad\text{for some }q\in(1,2).
\]
Equivalently, \(D_{\mathrm E}(X)>0\) if and only if
\[
    \rho(q)<\infty
    \qquad\text{for some }q>1.
\]
When these equivalent conditions hold,
\[
    \dim_{\mathrm F}(\mu)>0
\]
almost surely on \(\{M>0\}\).
\end{corollary}

\begin{proof}
The equivalences are  Proposition~\ref{prop:vector-DE-variational-formula-paper}.  If
\(D_{\mathrm E}(X)>0\), then
 Theorem~\ref{thm:exact-interval-formula-paper}  gives
\[
    \dim_{\mathrm F}(\mu)=D_{\mathrm E}(X)>0
\]
almost surely on \(\{M>0\}\).
\end{proof}

\begin{remark}
\label{rem:no-interval-endpoint-decay-claim-paper}
Theorem~\ref{thm:exact-interval-formula-paper} gives Fourier decay for every
strict exponent \(0<\sigma<D_{\mathrm E}(X)\), but it does not assert a uniform
decay estimate at the endpoint \(\sigma=D_{\mathrm E}(X)\).  This is the usual
distinction between a dimension identity and an endpoint regularity statement.
\end{remark}

This completes the interval part of the paper.  The remaining sections treat
the circle cascade, where the relevant obstruction is the endpoint local
exponent \(A_{\mathrm{loc}}(W)\), rather than the interval energy dimension.

\section{Circle cascades and the curved-support endpoint obstruction}
\label{sec:circle-endpoint-obstruction}

We prove the upper-bound direction of Theorem~\ref{thm:main-circle-endpoint-formula},
\[
\dimF(\mu_\circ)\le \Aloc(W)
\qquad
\text{almost surely on }\{\mu_\circ(\mathbb S^1)>0\}.
\]
The argument has two components:
\[
\alphamin(\mu_\circ)=\Aloc(W)
\qquad
\text{almost surely on }\{\mu_\circ(\mathbb S^1)>0\},
\]
together with the deterministic curved-support obstruction 
\(
\dimF(\nu)\le\alphamin(\nu)
\) 
for every nonzero finite Borel measure \(\nu\) supported on \(\mathbb S^1\). On the circle, this obstruction arises from quadratic stationary phase in the normal direction: a ball of radius \(r\) carrying comparatively large mass forces a large one-dimensional Fourier average at frequency scale \(r^{-2}\).

\subsection{The circle cascade}
\label{subsec:circle-cascade-construction-paper}

Let 
\(
\mathbb S^1=\{x\in\R^2:\lvert x \rvert=1\},
\) 
and let \(\sigma\) denote normalized arclength measure on \(\mathbb S^1\). We use the parametrization
\begin{equation*}
f:[0,1)\to\mathbb S^1,
\qquad
f(t)=(\cos 2\pi t,\sin 2\pi t),
\end{equation*}
only to impose the dyadic filtration.

Let \(W\ge0\) satisfy \(\E W=1\). Attach independent copies \(W_v\) of \(W\) to all nonempty binary words \(v\). For \(v\in\{0,1\}^n\), set
\begin{equation*}
Q_v=\prod_{j=1}^n W_{v|j},
\qquad
Q_\varnothing=1.
\end{equation*}
Let 
\(
J_v=[a_v,a_v+2^{-\lvert v \rvert})\subset[0,1),
\) 
where 
\(
a_v=\sum_{j=1}^{\lvert v \rvert}v_j2^{-j},
\)
and set \(\mathcal I_v=f(J_v)\). The level-\(n\) circle cascade measure is
\begin{equation}\label{eq:circle-level-measure-paper}
d\mu_{\circ,n}(x)
=
\sum_{\lvert v \rvert=n}Q_v\one_{\mathcal I_v}(x)\,d\sigma(x).
\end{equation}
Equivalently, 
\(
\mu_{\circ,n}(\mathcal I_v)=2^{-n}Q_v
\) 
for \(\lvert v \rvert=n\). Its total mass is
\begin{equation*}
Y_n=\mu_{\circ,n}(\mathbb S^1)
=
2^{-n}\sum_{\lvert v \rvert=n}Q_v.
\end{equation*}
Let 
\(
\mathcal F_0^W=\{\varnothing,\Omega\}
\) 
and 
\(
\mathcal F_n^W=\sigma(W_v:1\le \lvert v \rvert\le n)
\) 
for \(n\ge1\).

\begin{proposition}
\label{prop:circle-cascade-construction-paper}
Assume that \(W\ge0\) and \(\E W=1\). Then \((Y_n)_{n\ge0}\) is a nonnegative martingale. Moreover, there exists a finite random Borel measure \(\mu_\circ\) on \(\mathbb S^1\) such that
\begin{equation*}
\mu_{\circ,n}\weak\mu_\circ
\qquad
\text{almost surely}.
\end{equation*}
Its total mass satisfies
\begin{equation*}
\mu_\circ(\mathbb S^1)=Y:=\lim_{n\to\infty}Y_n
\end{equation*}
almost surely. In particular, 
\(
\{\mu_\circ\neq0\}=\{Y>0\}
\) 
up to null events.
\end{proposition}

The circle cascade is the pushforward, under \(f\), of the scalar dyadic cascade on the parameter interval; hence no separate compactness argument is required.

The nontriviality criterion is the classical Kahane--Peyri\`ere condition. We adopt the convention \(0\log_2 0=0\), and all expectations involving \(W\log_2 W\) are understood in the extended sense.

\begin{theorem}[Kahane--Peyri\`ere nondegeneracy on the circle]
\label{thm:circle-KP-nondegeneracy-paper}
Assume that \(W\ge0\) and \(\E W=1\). If
\begin{equation*}
\E[W\log_2^+W]<\infty
\qquad\text{and}\qquad
\E[W\log_2 W]<1,
\end{equation*}
then \((Y_n)\) is uniformly integrable, \(\E Y=1\), and \(\Pbb(Y>0)>0\). If \(\E[W\log_2 W]\ge1\), then \(Y=0\) almost surely.
\end{theorem}

\begin{remark}\label{rem:circle-KP-classical-source}
This is the classical dyadic Kahane--Peyri\`ere theorem for Mandelbrot martingales \cite{KahanePeyriere1976}; see also \cite{Heurteaux2016}. The proof depends only on the binary tree, the normalization \(\E W=1\), and the total-mass martingale, all of which are identical for the dyadic arc construction on \(\mathbb S^1\) and the dyadic interval construction on \([0,1)\).
\end{remark}

Write \(\mathcal S_\circ=\{Y>0\}\). By Proposition~\ref{prop:circle-cascade-construction-paper}, this agrees with 
\(
\{\mu_\circ(\mathbb S^1)>0\}
\) 
up to null events.

For \(v\in\mathcal T\), let \(\mu_\circ^{(v)}\) be the descendant circle cascade generated by \(\{W_{vu}:u\in\mathcal T,\ u\neq\varnothing\}\). For \(m\ge0\), define
\begin{equation*}
Y_m^{(v)}
=
2^{-m}
\sum_{\lvert u \rvert=m}
\prod_{j=1}^m W_{v(u|j)}.
\end{equation*}
Then \((Y_m^{(v)})_{m\ge0}\) is a nonnegative martingale with the same law as \((Y_m)\); denote its limit by \(Y^{(v)}\). By Proposition~\ref{prop:circle-cascade-construction-paper}, after intersecting countably many probability-one events, we may assume that
\begin{equation*}
\mu_\circ^{(v)}(\mathbb S^1)=Y^{(v)}
\qquad(v\in\mathcal T).
\end{equation*}
For each \(m\ge0\), the family \((\mu_\circ^{(v)})_{\lvert v \rvert=m}\) is independent, is independent of \(\mathcal F_m^W\), and consists of measures with the same law as \(\mu_\circ\); the same assertions hold for \((Y^{(v)})_{\lvert v \rvert=m}\). Since \(\mathcal T\) is countable, we may also assume that all descendant martingales converge. For every \(n\ge \lvert v \rvert\),
\begin{equation*}
\mu_{\circ,n}(\mathcal I_v)
=
2^{-\lvert v \rvert}Q_vY_{n-\lvert v \rvert}^{(v)},
\end{equation*}
and hence
\begin{equation*}
\lim_{n\to\infty}\mu_{\circ,n}(\mathcal I_v)
=
2^{-\lvert v \rvert}Q_vY^{(v)}
\qquad(v\in\mathcal T).
\end{equation*}
By the standard dyadic endpoint estimate, we further restrict to the full-probability event on which \(\mu_\circ\) charges no dyadic endpoint. Then every half-open arc \(\mathcal I_v\) is a \(\mu_\circ\)-continuity set, and
\begin{equation}\label{eq:circle-cylinder-mass-decomposition-paper}
\mu_\circ(\mathcal I_v)
=
2^{-\lvert v \rvert}Q_vY^{(v)}
\qquad(v\in\mathcal T)
\end{equation}
simultaneously for all \(v\in\mathcal T\).

\subsection{The endpoint local exponent}
\label{subsec:endpoint-local-exponent-paper}

Recall
\begin{equation*}
\Aloc(W)
=
\sup_{q>1}
\max\left\{
0,\,
\frac{q-1-\log_2\E[W^q]}{q}
\right\},
\end{equation*}
where the term is interpreted as \(0\) when \(\E[W^q]=\infty\). Write \(\tau(q)=q-1-\log_2\E[W^q]\) whenever \(\E[W^q]<\infty\). Then
\begin{equation*}
\Aloc(W)
=
\sup_{\substack{q>1\\ \E[W^q]<\infty}}
\max\left\{0,\frac{\tau(q)}{q}\right\}.
\end{equation*}
By Jensen's inequality, \(\E[W^q]\ge(\E W)^q=1\) for \(q>1\), and therefore \(0\le\Aloc(W)\le1\).

For a nonzero finite Borel measure \(\nu\) on \(\mathbb S^1\), define
\begin{equation*}
\alphamin(\nu)
=
\inf_{x\in\spt\nu}
\liminf_{r\downarrow0}
\frac{\log_2\nu(B(x,r))}{\log_2 r},
\end{equation*}
where \(B(x,r)\) denotes the Euclidean ball in \(\R^2\). For the zero measure, set \(\alphamin(0)=0\).

\begin{lemma}
\label{lem:circle-uniform-cylinder-upper-bound-paper}
Assume the minimal Kahane--Peyri\`ere regime. If \(0<\beta<\Aloc(W)\), then, almost surely, there exists a finite random constant \(C_\beta\) such that
\begin{equation}\label{eq:circle-uniform-cylinder-upper-bound-paper}
\mu_\circ(\mathcal I_v)\le C_\beta 2^{-\beta \lvert v \rvert}
\end{equation}
for every finite binary word \(v\).
\end{lemma}

\begin{proof}
Choose \(q>1\) such that \(\E[W^q]<\infty\) and \(\tau(q)/q>\beta\). Then \(\tau(q)>0\), equivalently \(2^{1-q}\E[W^q]<1\), and the standard terminal-mass moment criterion implies \(\E[Y^q]<\infty\). By \eqref{eq:circle-cylinder-mass-decomposition-paper}, the independence of \(Q_v\) and \(Y^{(v)}\), and the identity \(Y^{(v)}\stackrel{d}=Y\),
\begin{equation*}
\E[\mu_\circ(\mathcal I_v)^q]
=
2^{-q\lvert v \rvert}\E[Q_v^q]\E[Y^q]
=
2^{-q\lvert v \rvert}\bigl(\E[W^q]\bigr)^{\lvert v \rvert}\E[Y^q].
\end{equation*}
Thus
\begin{equation*}
\E\sum_{\lvert v \rvert=n}\mu_\circ(\mathcal I_v)^q
=
\E[Y^q]\,2^{-n\tau(q)}.
\end{equation*}
Therefore
\begin{equation*}
\Pbb\left(\max_{\lvert v \rvert=n}\mu_\circ(\mathcal I_v)>2^{-\beta n}\right)
\le
2^{\beta q n}
\E\sum_{\lvert v \rvert=n}\mu_\circ(\mathcal I_v)^q
\le
C2^{-n(\tau(q)-\beta q)}.
\end{equation*}
Since \(\tau(q)-\beta q>0\), the right-hand side is summable. The Borel--Cantelli lemma implies \(\max_{\lvert v \rvert=n}\mu_\circ(\mathcal I_v)\le2^{-\beta n}\) for all sufficiently large \(n\), almost surely. Enlarging the random constant to cover the finitely many initial generations proves \eqref{eq:circle-uniform-cylinder-upper-bound-paper}.
\end{proof}

\begin{lemma}
\label{lem:circle-cylinder-to-local-dimension-paper}
Let \(\nu\) be a finite Borel measure on \(\mathbb S^1\). Suppose that, for some \(\beta\ge0\), there is a constant \(C<\infty\) such that
\begin{equation*}
\nu(\mathcal I_v)\le C2^{-\beta \lvert v \rvert}
\qquad(v\in\mathcal T).
\end{equation*}
Then
\begin{equation*}
\liminf_{r\downarrow0}
\frac{\log_2\nu(B(x,r))}{\log_2 r}\ge\beta
\qquad
(x\in\spt\nu).
\end{equation*}
\end{lemma}

\begin{proof}
If \(2^{-(n+1)}<r\le2^{-n}\), then \(B(x,r)\cap\mathbb S^1\) meets at most a bounded number of level-\(n\) dyadic arcs. Hence
\begin{equation*}
\nu(B(x,r))\le C'2^{-\beta n}\le C''r^\beta.
\end{equation*}
For \(x\in\spt\nu\), divide \(\log_2\nu(B(x,r))\le \log_2 C''+\beta\log_2 r\) by \(\log_2 r<0\) and let \(r\downarrow0\) to obtain the claim.
\end{proof}

\begin{lemma}
\label{lem:alive-tree-supercritical-support-witness}
Assume the minimal Kahane--Peyri\`ere regime. Then \(\Pbb(W>0)>1/2\). Hence
\begin{equation*}
\mathcal T^{+}:=\{v\in\mathcal T:Q_v>0\}
\end{equation*}
is a supercritical Galton--Watson tree with offspring law \(\operatorname{Binomial}(2,\Pbb(W>0))\). Moreover, with
\begin{equation*}
A_n:=\#\{v\in\{0,1\}^{n}:Q_v>0\},
\end{equation*}
one has \(A_n\to\infty\) almost surely on \(\{\mu_\circ(\mathbb S^1)>0\}\).
\end{lemma}

\begin{remark}
Lemma~\ref{lem:alive-tree-supercritical-support-witness} is the specialization of Lemmas~\ref{lem:alive-offspring-supercritical-paper} and~\ref{lem:alive-tree-growth-paper} to the normalized binary weights \(X_i=W_i/2\), \(i=0,1\). For \(v\in\{0,1\}^n\), the corresponding product is \(L_v=2^{-n}Q_v\), so the alive tree there is precisely \(\mathcal T^+=\{v\in\mathcal T:Q_v>0\}\).

Thus
\begin{equation*}
\E N_X=2\Pbb(W>0)>1,
\end{equation*}
hence \(\Pbb(W>0)>1/2\), and Lemma~\ref{lem:alive-tree-growth-paper} implies that
\begin{equation*}
A_n=\#\{v\in\{0,1\}^n:Q_v>0\}\to\infty
\end{equation*}
almost surely on \(\{\mu_\circ(\mathbb S^1)>0\}\). Replacing dyadic intervals by dyadic arcs does not affect this tree-indexed positivity argument.
\end{remark}

\begin{lemma}
\label{lem:biggins-chernoff-block-count-exponent}
Assume the minimal Kahane--Peyri\`ere regime. For \(a\in\mathbb R\) and \(n\ge1\), set
\begin{equation*}
N_n(a)=\#\{u\in\{0,1\}^{n}:Q_u\ge 2^{an}\}.
\end{equation*}
Then:
\begin{enumerate}
\item For every \(a\in\mathbb R\),
\begin{equation*}
\Lambda(a):=\lim_{n\to\infty}\frac1n\log_2\mathbb E N_n(a)
=
\inf_{q\ge0}
\left\{
1+\log_2\mathbb E(W^q)-qa
\right\},
\end{equation*}
where \(W^0:=\mathbf 1_{\{W>0\}}\), and for \(q>0\) the term is \(+\infty\) if \(\mathbb E(W^q)=+\infty\).

\item If \(a<1-\Aloc(W)\), then \(\Lambda(a)>0\).

\item The function \(\Lambda\) is finite and continuous on \((-\infty,1-\Aloc(W))\).
\end{enumerate}
\end{lemma}

\begin{proof}
Apply Biggins' Chernoff theorem for branching random walks \cite[Theorem~1]{Biggins1977} to the displacement
\(-\ln W\), with edges for which \(W=0\) treated as killed. The
one-generation transform is
\[
m(0)=2\Pbb(W>0),\qquad
m(\theta)=2\mathbb E(W^\theta)\quad(\theta>0),
\]
with the value \(+\infty\) allowed.  By
Lemma~\ref{lem:alive-tree-supercritical-support-witness}, \(m(0)>1\), and
\(m(1)=2\mathbb EW=2<\infty\).  Since \(N_n(a)\) counts the particles with total displacement at most
\(-an\ln2\), Biggins' theorem yields
\[
\lim_{n\to\infty}
\bigl(\mathbb E N_n(a)\bigr)^{1/n}
=
\inf_{\theta\ge0} e^{-a\theta\ln2}m(\theta).
\]
Taking \(\log_2\) yields
\[
\Lambda(a)=
\inf_{q\ge0}
\left\{
1+\log_2\mathbb E(W^q)-qa
\right\},
\]
with \(W^0=\mathbf 1_{\{W>0\}}\) and with the stated extended-value
convention.

For \(q>0\), define
\begin{equation*}
\Psi(q)=\frac{1+\log_2\mathbb E(W^q)}{q},
\end{equation*}
with value \(+\infty\) if \(\mathbb E(W^q)=+\infty\). We first show that
\begin{equation*}
\inf_{q>0}\Psi(q)=1-\Aloc(W).
\end{equation*}
For \(0<q\le1\), put \(\varphi(q)=\log_2\mathbb E(W^q)\). Then \(\varphi\) is convex, \(\varphi(1)=0\), and \(\varphi'_-(1)=\mathbb E(W\log_2 W)\). Hence, for \(0<q<1\),
\begin{equation*}
\varphi(q)\ge(q-1)\mathbb E(W\log_2 W).
\end{equation*}
Since the minimal Kahane--Peyri\`ere regime ensures \(\mathbb E(W\log_2 W)<1\), it follows that \(\Psi(q)>1\) for \(0<q<1\), while \(\Psi(1)=1\). Thus
\begin{equation*}
\inf_{0<q\le1}\Psi(q)=1.
\end{equation*}
Together with the definition of \(\Aloc(W)\), this implies
\begin{equation*}
1-\Aloc(W)
=
\min\left\{1,\inf_{q>1}\Psi(q)\right\}
=
\inf_{q>0}\Psi(q).
\end{equation*}

Fix \(a<1-\Aloc(W)\), and set \(\eta=1-\Aloc(W)-a>0\). For \(q>0\),
\begin{equation*}
1+\log_2\mathbb E(W^q)-qa
=
q(\Psi(q)-a)
\ge q\eta
\end{equation*}
in the extended sense. Also, \(\mathbb E(W^q)\to\Pbb(W>0)\) as \(q\downarrow0\), and Lemma~\ref{lem:alive-tree-supercritical-support-witness} ensures that \(1+\log_2\Pbb(W>0)>0\). Hence there exist \(q_0,c_0>0\) such that
\begin{equation*}
1+\log_2\mathbb E(W^q)-qa\ge c_0
\qquad(0\le q\le q_0),
\end{equation*}
whereas for \(q\ge q_0\) the previous bound yields the lower bound \(q_0\eta\). Therefore
\begin{equation*}
\Lambda(a)\ge\min\{c_0,q_0\eta\}>0.
\end{equation*}

Finally, for each \(q\ge0\), the map \(a\mapsto 1+\log_2\mathbb E(W^q)-qa\) is affine; hence \(\Lambda\) is concave. On \((-\infty,1-\Aloc(W))\), the positivity just proved and the \(q=0\) term imply
\begin{equation*}
0<\Lambda(a)\le1+\log_2\Pbb(W>0)<\infty.
\end{equation*}
Thus \(\Lambda\) is finite and concave on this open interval, and is therefore continuous there.
\end{proof}

\begin{theorem}
\label{thm:support-mass-witness}
Assume the minimal Kahane--Peyri\`ere regime. Let \(\beta>\Aloc(W)\), let \(\Lambda\) be the block-count exponent from Lemma~\ref{lem:biggins-chernoff-block-count-exponent}, and let \(C\ge0\) be finite almost surely on \(\mathcal S_\circ\). Define
\begin{equation*}
E_{\beta,C}(\mu_\circ)
:=
\left\{
x\in\operatorname{spt}\mu_\circ:
\mu_\circ(\mathcal I_{n_k}(x))
\ge C(\omega)2^{-\beta n_k}
\text{ for some }n_k\uparrow\infty
\right\}.
\end{equation*}
Then, almost surely on \(\mathcal S_\circ\),
\begin{equation*}
\dimH E_{\beta,C}(\mu_\circ)
\ge
\Lambda(1-\beta)
>
0.
\end{equation*}
\end{theorem}

\begin{proof}
Fix \(0<\delta<\beta-\Aloc(W)\), and set \(a=1-\beta+\delta<1-\Aloc(W)\). By the descendant construction in Subsection~\ref{subsec:circle-cascade-construction-paper}, \(Y^{(v)}=\mu_\circ^{(v)}(\mathbb S^1)\stackrel{d}{=}\mu_\circ(\mathbb S^1)\), and
\begin{equation*}
Y^{(v)}
=
\frac12 W_{v0}Y^{(v0)}
+
\frac12 W_{v1}Y^{(v1)},
\qquad
\mu_\circ(\mathcal I_v)=2^{-\lvert v \rvert}Q_vY^{(v)}
\end{equation*}
by the first-step descendant decomposition and \eqref{eq:circle-cylinder-mass-decomposition-paper}.

By Theorem~\ref{thm:circle-KP-nondegeneracy-paper}, \(\mathbb P(Y^{(u)}>0)>0\), and since \(W_u\) is independent of \(Y^{(u)}\) with \(\mathbb P(W_u>0)>0\), choose \(c>0\) such that
\begin{equation*}
p_c:=\mathbb P(W_uY^{(u)}\ge 2c)>0,
\end{equation*}
independently of \(u\).

Let \(0<\varepsilon<\Lambda(a)/2\). Choose \(h\ge1\) such that
\begin{equation*}
\frac1h\log_2\mathbb E N_h(a)>\Lambda(a)-\varepsilon,
\qquad
\frac1h\log_2\left(\frac{p_c}{2}\right)>-\varepsilon.
\end{equation*}
Then
\begin{equation*}
m_h:=\frac{p_c}{2}\mathbb E N_h(a)
>
2^{h(\Lambda(a)-2\varepsilon)}
>
1.
\end{equation*}

For \(v,u\in\mathcal T\), set
\begin{equation*}
Q_u^{(v)}=\prod_{j=1}^{\lvert u \rvert}W_{v(u|j)},
\qquad Q_{\varnothing}^{(v)}=1.
\end{equation*}
If \(Q_v>0\), then \(Q_u^{(v)}=Q_{vu}/Q_v\). Define
\begin{equation*}
\mathcal C(v)
=
\left\{
v0u:
u\in\{0,1\}^{h-1},\
W_{v1}Y^{(v1)}\ge 2c,\
Q_{0u}^{(v)}\ge 2^{ah}
\right\},
\qquad
\xi_v=\#\mathcal C(v).
\end{equation*}
The side-witness event \(\{W_uY^{(u)}\ge 2c\}\) is independent of the \(v0\)-subtree and has probability \(p_c\). By symmetry,
\begin{equation*}
\mathbb E\xi_v
=
p_c\,
\mathbb E\#\{u\in\{0,1\}^{h-1}:Q_{0u}^{(v)}\ge 2^{ah}\}
=
\frac{p_c}{2}\mathbb E N_h(a)
=
m_h>1.
\end{equation*}

Starting from \(v\), set
\begin{equation*}
Z_0^{(v)}=\{v\},
\qquad
Z_{k+1}^{(v)}=\bigcup_{w\in Z_k^{(v)}}\mathcal C(w).
\end{equation*}
Then \((\#Z_k^{(v)})_{k\ge0}\) is a Galton--Watson process with offspring law \(\xi_v\): different current particles use disjoint subtrees; if \(z\in\mathcal C(w)\), the event \(z\in\mathcal C(w)\) uses only the finite path from \(w\) to \(z\) and the side subtree \(w1\), and no weights strictly below \(z\); the laws are identical by the tree-indexed i.i.d. construction. Its mean is \(m_h>1\). Therefore
\begin{equation*}
\rho_h
:=
\mathbb P\left(
Z_k^{(\varnothing)}\neq\varnothing
\text{ for every }k\ge0
\right)
>0.
\end{equation*}

We next show that, on non-extinction, some alive vertex starts a surviving witnessed block process. Let
\begin{equation*}
\mathcal F_n=\sigma(W_w:1\le \lvert w \rvert\le n),
\qquad
H_v=\bigcap_{k=0}^{\infty}\{Z_k^{(v)}\neq\varnothing\},
\quad v\in\{0,1\}^{n}.
\end{equation*}
Conditional on \(\mathcal F_n\), the events \(H_v\) are independent and \(\mathbb P(H_v\mid\mathcal F_n)=\rho_h\). Put
\begin{equation*}
B_n
=
\bigcap_{v\in\{0,1\}^{n}}
\left(
\{Q_v=0\}\cup H_v^c
\right),
\qquad
G=\bigcap_{n=0}^{\infty}B_n.
\end{equation*}
Then
\begin{equation*}
\mathbb P(B_n\mid\mathcal F_n)=(1-\rho_h)^{A_n},
\qquad
\mathbb E(\mathbf 1_G\mid\mathcal F_n)
\le (1-\rho_h)^{A_n}.
\end{equation*}
By Lemma~\ref{lem:alive-tree-supercritical-support-witness}, \(A_n\to\infty\) almost surely on \(\{\mu_\circ(\mathbb S^1)>0\}\), while \(\mathbb E(\mathbf 1_G\mid\mathcal F_n)\to\mathbf 1_G\) almost surely by martingale convergence. Hence \(G\) has probability zero on the non-extinction event. Thus, almost surely on the non-extinction event, there is an alive vertex \(v_0\) such that \(Q_{v_0}>0\) and \(Z_k^{(v_0)}\neq\varnothing\) for every \(k\ge0\).

Define
\begin{equation*}
\partial\mathcal T_{v_0}^{\mathrm{wit}}
:=
\left\{
(w_k)_{k=0}^{\infty}:
w_0=v_0,\ 
w_{k+1}\in\mathcal C(w_k)\ \text{for all }k\ge0
\right\},
\end{equation*}
with \(h\)-block ultrametric
\begin{equation*}
d_{\mathrm{tree}}(\gamma,\gamma')
=
2^{-h\lvert\gamma\wedge\gamma'\rvert_{\mathrm{GW}}},
\end{equation*}
where \(\lvert\gamma\wedge\gamma'\rvert_{\mathrm{GW}}\) denotes the common initial block length. Let \(\Pi_{v_0}\) be the dyadic coding map
\begin{equation*}
\Pi_{v_0}(\gamma)
\in
\bigcap_{k=0}^{\infty}\overline{\mathcal I_{w_k}}.
\end{equation*}
If \(D\) is the countable set of dyadic endpoints on \(\mathbb S^1\), set
\begin{equation*}
K_{v_0}
:=
\Pi_{v_0}\bigl(\partial\mathcal T_{v_0}^{\mathrm{wit}}\bigr)\setminus D.
\end{equation*}
By the Galton--Watson boundary dimension theorem and the metric-change remark \cite[Theorem~0(ii) and Remark~6]{Liu1996}, on survival,
\begin{equation*}
\dimH K_{v_0}
=
\frac{\log_2 m_h}{h}
=
\frac1h\log_2\mathbb E N_h(a)
+
\frac1h\log_2\left(\frac{p_c}{2}\right)
>
\Lambda(a)-2\varepsilon.
\end{equation*}
Here deleting \(D\) does not change Hausdorff dimension, and the symbolic \(h\)-block metric and the Euclidean metric are dimensionally equivalent because each \(h\)-block cylinder projects to a dyadic arc of comparable length, while each arc at that scale intersects only \(O_h(1)\) such arcs.

It remains to show that \(K_{v_0}\subseteq E_{\beta,C}(\mu_\circ)\). Let \(x\in K_{v_0}\). Then \(x\notin D\) corresponds to a chain
\begin{equation*}
v_0\prec v_1\prec v_2\prec\cdots,
\qquad
v_{k+1}\in\mathcal C(v_k),
\qquad
\lvert v_k \rvert=\lvert v_0 \rvert+kh,
\end{equation*}
with \(\mathcal I_{\lvert v_k \rvert}(x)=\mathcal I_{v_k}\). The relative-product condition yields
\begin{equation*}
\frac{Q_{v_{k+1}}}{Q_{v_k}}\ge 2^{ah},
\qquad
Q_{v_k}\ge Q_{v_0}2^{akh},
\end{equation*}
and the side witness ensures
\begin{equation*}
Y^{(v_k)}
=
\frac12W_{v_k0}Y^{(v_k0)}
+
\frac12W_{v_k1}Y^{(v_k1)}
\ge c.
\end{equation*}
Therefore
\begin{equation*}
\mu_\circ(\mathcal I_{v_k})
=
2^{-\lvert v_k \rvert}Q_{v_k}Y^{(v_k)}
\ge
c\,2^{-\lvert v_k \rvert}Q_{v_0}2^{akh}.
\end{equation*}
Since \(\lvert v_k \rvert=\lvert v_0 \rvert+kh\) and \(a=1-\beta+\delta\),
\begin{equation*}
\frac{\mu_\circ(\mathcal I_{v_k})}{2^{-\beta\lvert v_k \rvert}}
\ge
cQ_{v_0}2^{(\beta-1)\lvert v_0 \rvert}2^{\delta kh}
\to+\infty.
\end{equation*}
As \(C(\omega)<\infty\) on non-extinction, for all large \(k\),
\begin{equation*}
\mu_\circ(\mathcal I_{\lvert v_k \rvert}(x))
=
\mu_\circ(\mathcal I_{v_k})
\ge
C(\omega)2^{-\beta\lvert v_k \rvert}.
\end{equation*}
The cylinders \(\mathcal I_{v_k}\) shrink to \(x\) and have positive \(\mu_\circ\)-mass for all large \(k\), so \(x\in\operatorname{spt}\mu_\circ\). Thus \(K_{v_0}\subseteq E_{\beta,C}(\mu_\circ)\).

Consequently,
\begin{equation*}
\dimH E_{\beta,C}(\mu_\circ)
\ge
\dimH K_{v_0}
>
\Lambda(1-\beta+\delta)-2\varepsilon.
\end{equation*}
Letting \(\varepsilon\downarrow0\), then \(\delta\downarrow0\), and using the continuity from Lemma~\ref{lem:biggins-chernoff-block-count-exponent}, we obtain
\begin{equation*}
\dimH E_{\beta,C}(\mu_\circ)
\ge
\Lambda(1-\beta)>0.
\end{equation*}
\end{proof}

\begin{theorem}
\label{thm:exact-dimension-support-mass-witnesses}
Assume the minimal Kahane--Peyri\`ere regime. Let \(\beta>\Aloc(W)\), let \(\Lambda\) be the block-count exponent in Lemma~\ref{lem:biggins-chernoff-block-count-exponent}, and let \(C\) be a nonnegative random variable with \(0<C<\infty\) almost surely on \(\mathcal S_\circ\). Then, almost surely on \(\mathcal S_\circ\),
\begin{equation*}
\dimH E_{\beta,C}(\mu_\circ)=\Lambda(1-\beta).
\end{equation*}
Equivalently,
\begin{equation*}
\dimH E_{\beta,C}(\mu_\circ)
=
\inf_{q\ge0}
\left\{
1+\log_2\mathbb E(W^q)-q(1-\beta)
\right\},
\end{equation*}
where \(W^0=\mathbf 1_{\{W>0\}}\), and the value inside the infimum is interpreted as \(+\infty\) whenever \(\mathbb E(W^q)=+\infty\).
\end{theorem}

\begin{proof}
The lower bound follows from Theorem~\ref{thm:support-mass-witness}. We prove the reverse inequality.

Fix \(\gamma>\Aloc(W)\), and set
\begin{equation*}
E_{\gamma}(\mu_\circ)
:=
\left\{
x\in\operatorname{spt}\mu_\circ:
\mu_\circ(\mathcal I_{n_k}(x))
\ge 2^{-\gamma n_k}
\text{ for some }n_k\uparrow\infty
\right\}.
\end{equation*}
For \(n\ge1\), let
\begin{equation*}
H_n(\gamma)
=
\left\{
v\in\{0,1\}^{n}:
\mu_\circ(\mathcal I_v)\ge 2^{-\gamma n}
\right\}.
\end{equation*}
Then
\begin{equation*}
E_\gamma(\mu_\circ)
\subseteq
\limsup_{n\to\infty}
\bigcup_{v\in H_n(\gamma)}\mathcal I_v.
\end{equation*}

Let
\begin{equation*}
\mathcal A
=
(0,1]
\cup
\left\{
q>1:
\mathbb E(W^q)<\infty,\
2^{1-q}\mathbb E(W^q)<1
\right\}.
\end{equation*}
For \(q\in\mathcal A\), the \(L^q\)-boundedness theorem for Mandelbrot cascade martingales implies that \(\mathbb E(Y^q)<\infty\), where \(Y=\mu_\circ(\mathbb S^1)\). Put
\begin{equation*}
\Phi_\gamma(q)=1+\log_2\mathbb E(W^q)-q(1-\gamma).
\end{equation*}
Since
\begin{equation*}
\mu_\circ(\mathcal I_v)=2^{-n}Q_vY^{(v)},
\end{equation*}
with \(Y^{(v)}\) an independent copy of \(Y\), independent of \(Q_v\),
\begin{equation*}
\mathbb E\sum_{\lvert v \rvert=n}\mu_\circ(\mathcal I_v)^q
=
\mathbb E(Y^q)\,
2^{n(1-q)}
\bigl(\mathbb E(W^q)\bigr)^n.
\end{equation*}
For \(s>0\),
\begin{equation*}
\mathbf 1_{\{\mu_\circ(\mathcal I_v)\ge 2^{-\gamma n}\}}
\le
2^{q\gamma n}\mu_\circ(\mathcal I_v)^q,
\qquad
\operatorname{diam}(\mathcal I_v)\le C_1 2^{-n},
\end{equation*}
and therefore
\begin{equation*}
\mathbb E
\sum_{v\in H_n(\gamma)}
\operatorname{diam}(\mathcal I_v)^s
\le
C_1^s\mathbb E(Y^q)\,
2^{-n(s-\Phi_\gamma(q))}.
\end{equation*}
If \(s>\Phi_\gamma(q)\), summing over \(n\) and applying Fubini yields
\begin{equation*}
\sum_{n=1}^{\infty}
\sum_{v\in H_n(\gamma)}
\operatorname{diam}(\mathcal I_v)^s
<\infty
\end{equation*}
almost surely. The limsup cover then implies
\begin{equation*}
\dimH E_\gamma(\mu_\circ)\le \Phi_\gamma(q),
\qquad q\in\mathcal A.
\end{equation*}

It remains to remove the restriction \(q\in\mathcal A\). If \(q>1\), \(\mathbb E(W^q)<\infty\), and \(2^{1-q}\mathbb E(W^q)\ge1\), then
\begin{equation*}
\Phi_\gamma(q)
=
q\gamma+1-q+\log_2\mathbb E(W^q)
\ge q\gamma>\gamma=\Phi_\gamma(1),
\end{equation*}
so such \(q\) cannot decrease the infimum; if \(\mathbb E(W^q)=+\infty\), the value is \(+\infty\). Since \(q=1\in\mathcal A\) and \(\mathbb E(W^q)\to\mathbb P(W>0)\) as \(q\downarrow0\), the endpoint \(q=0\) may also be included. Hence
\begin{equation*}
\inf_{q\in\mathcal A}\Phi_\gamma(q)
=
\inf_{q\ge0}
\left\{
1+\log_2\mathbb E(W^q)-q(1-\gamma)
\right\}.
\end{equation*}
By Lemma~\ref{lem:biggins-chernoff-block-count-exponent}, the last infimum is \(\Lambda(1-\gamma)\). Thus
\begin{equation*}
\dimH E_\gamma(\mu_\circ)\le \Lambda(1-\gamma)
\end{equation*}
almost surely.

Now let \(C\) be as in the statement. For each rational \(\eta>0\), on \(\{0<C(\omega)<\infty\}\),
\begin{equation*}
C(\omega)2^{-\beta n}\ge 2^{-(\beta+\eta)n}
\end{equation*}
eventually, and hence
\begin{equation*}
E_{\beta,C}(\mu_\circ)\subseteq E_{\beta+\eta}(\mu_\circ).
\end{equation*}
The deterministic upper bound with \(\gamma=\beta+\eta\) yields
\begin{equation*}
\dimH E_{\beta,C}(\mu_\circ)
\le
\Lambda(1-\beta-\eta)
\end{equation*}
almost surely. Taking a countable intersection over rational \(\eta>0\), letting \(\eta\downarrow0\), and using the continuity of \(\Lambda\) from Lemma~\ref{lem:biggins-chernoff-block-count-exponent}, we obtain
\begin{equation*}
\dimH E_{\beta,C}(\mu_\circ)\le\Lambda(1-\beta).
\end{equation*}
Together with the lower bound, this proves the equality. The variational formula follows from Lemma~\ref{lem:biggins-chernoff-block-count-exponent}.
\end{proof}

\begin{remark}
The closest result among the references cited here is \cite[Theorem~7]{Heurteaux2016}. It proves a stronger multifractal level-set statement, but under stronger regularity assumptions: in addition to nondegeneracy, it assumes \(\mathbb P(W=0)=0\)
and the existence of all real moments \(\mathbb E(W^q)<\infty
\) for every \(q\in\mathbb R\).
\end{remark}

\begin{theorem}
\label{thm:minimum-local-dimension}
Assume the minimal Kahane--Peyri\`ere regime. Then, almost surely on \(\mathcal S_\circ\),
\begin{equation}\label{eq:minimum-local-dimension-paper}
\alphamin(\mu_\circ)=\Aloc(W).
\end{equation}
\end{theorem}

\begin{proof}
We first prove the lower bound. If \(\Aloc(W)=0\), the claim is immediate, since, for every finite positive measure \(\nu\),
\begin{equation*}
\inf_{x\in\operatorname{spt}\nu}
\liminf_{r\downarrow0}
\frac{\log_2\nu(B(x,r))}{\log_2 r}
\ge0.
\end{equation*}
Assume \(\Aloc(W)>0\), and choose \(0<\beta_j<\Aloc(W)\) with \(\beta_j\uparrow\Aloc(W)\). For each \(j\), Lemma~\ref{lem:circle-uniform-cylinder-upper-bound-paper} and Lemma~\ref{lem:circle-cylinder-to-local-dimension-paper} imply, on a full-probability event, that
\begin{equation*}
\liminf_{r\downarrow0}
\frac{\log_2\mu_\circ(B(x,r))}{\log_2 r}
\ge \beta_j
\qquad
\text{for every }x\in\operatorname{spt}\mu_\circ.
\end{equation*}
Intersecting these countably many full-probability events yields \(\alphamin(\mu_\circ)\ge\beta_j\) for all \(j\), hence \(\alphamin(\mu_\circ)\ge \Aloc(W)\)
almost surely on \(\mathcal S_\circ\).

For the upper bound, choose \(\beta_j>\Aloc(W)\) with \(\beta_j\downarrow\Aloc(W)\). By Theorem~\ref{thm:exact-dimension-support-mass-witnesses},
\begin{equation*}
\dimH E_{\beta_j,1}(\mu_\circ)=\Lambda(1-\beta_j)>0
\end{equation*}
almost surely on \(\mathcal S_\circ\). Hence, after a countable intersection, for every \(j\) there exist \(x_j\in\spt\mu_\circ\) and \(n_{j,k}\uparrow\infty\) such that
\begin{equation*}
\mu_\circ(\mathcal I_{n_{j,k}}(x_j))
\ge
2^{-\beta_j n_{j,k}}.
\end{equation*}
For fixed \(j\), \(\mathcal I_{n_{j,k}}(x_j)\) is contained in a Euclidean ball \(B(x_j,C2^{-n_{j,k}})\) with deterministic \(C<\infty\). Set \(r_{j,k}=C2^{-n_{j,k}}\). Then
\begin{equation*}
\mu_\circ(B(x_j,r_{j,k}))
\ge
2^{-\beta_j n_{j,k}}
=
C^{-\beta_j}r_{j,k}^{\beta_j},
\end{equation*}
and hence
\begin{equation*}
\log_2\mu_\circ(B(x_j,r_{j,k}))
\ge
\beta_j\log_2 r_{j,k}-\beta_j\log_2 C.
\end{equation*}
Since \(\log_2 r_{j,k}<0\),
\begin{equation*}
\frac{\log_2\mu_\circ(B(x_j,r_{j,k}))}{\log_2 r_{j,k}}
\le
\beta_j-\frac{\beta_j\log_2 C}{\log_2 r_{j,k}}.
\end{equation*}
Letting \(k\to\infty\) yields
\begin{equation*}
\liminf_{r\downarrow0}
\frac{\log_2\mu_\circ(B(x_j,r))}{\log_2 r}
\le
\beta_j.
\end{equation*}
Since \(x_j\in\operatorname{spt}\mu_\circ\), \(\alphamin(\mu_\circ)\le\beta_j\). Letting \(j\to\infty\) yields \(\alphamin(\mu_\circ)\le \Aloc(W)\) almost surely on \(\mathcal S_\circ\). Combining the two bounds proves \eqref{eq:minimum-local-dimension-paper}.
\end{proof}

\subsection{The curved-support upper bound}
\label{subsec:curved-support-upper-bound-paper}

We prove the deterministic upper bound \(\dimF(\nu)\le\alphamin(\nu)\) for finite measures supported on the circle.

\begin{lemma}
\label{lem:radial-L2-local-mass-paper}
Let \(\nu\) be a finite Borel measure on \(\mathbb S^1\). There exist constants \(c,C>0\) such that for every \(x_0\in\mathbb S^1\) and all sufficiently small \(r>0\), if \(\Lambda=r^{-2}\), then
\begin{equation}\label{eq:radial-L2-local-mass-paper}
\int_{\R}\lvert\wh\nu(Rx_0)\rvert^2e^{-\pi(R/\Lambda)^2}\,dR
\ge
c\Lambda\,\nu(B(x_0,cr))^2.
\end{equation}
\end{lemma}

\begin{proof}
Let \(x_0=f(t_0)\), and pull \(\nu\) back to \([0,1)\); we denote the resulting measure again by \(\nu\). Set
\begin{equation*}
a(t)=f(t)\cdot x_0=\cos 2\pi(t-t_0).
\end{equation*}
For \(t,u\) sufficiently close to \(t_0\),
\begin{equation*}
\lvert a(t)-a(t_0)\rvert\le C\lvert t-t_0\rvert^2,
\qquad
\lvert a(u)-a(t_0)\rvert\le C\lvert u-t_0\rvert^2.
\end{equation*}
Let \(\psi(R)=e^{-\pi R^2}\). Its Fourier transform is again \(e^{-\pi\rho^2}\), which is positive and bounded below in a neighborhood of the origin. Expanding the square and applying Fubini,
\begin{equation*}
\begin{aligned}
\int_{\R}\lvert\wh\nu(Rx_0)\rvert^2\psi(R/\Lambda)\,dR
&=
\int_{\R}\iint e^{-2\pi iR(a(t)-a(u))}\,d\nu(t)d\nu(u)\,\psi(R/\Lambda)\,dR\\
&=
\Lambda\iint e^{-\pi\Lambda^2(a(t)-a(u))^2}\,d\nu(t)d\nu(u).
\end{aligned}
\end{equation*}
If \(t,u\) both lie in a sufficiently small arc of length \(cr\) around \(t_0\), then
\begin{equation*}
\Lambda \lvert a(t)-a(u)\rvert\le C r^{-2}r^2\le1
\end{equation*}
after choosing \(c\) sufficiently small. On this set the Gaussian factor is bounded below by a positive absolute constant. Restricting the double integral to the square of this arc, and using the comparability of chord distance and arclength at small scales, proves \eqref{eq:radial-L2-local-mass-paper}.
\end{proof}

\begin{proposition}
\label{prop:curved-support-upper-bound}
Let \(\nu\) be a nonzero finite Borel measure supported on \(\mathbb S^1\). Then
\begin{equation}\label{eq:curved-support-upper-bound-paper}
\dimF(\nu)\le\alphamin(\nu).
\end{equation}
\end{proposition}

\begin{proof}
Let \(t\) be an admissible Fourier decay exponent for \(\nu\), and assume \(0<t<1\). Thus
\begin{equation*}
\lvert\wh\nu(\xi)\rvert\le C_t(1+\lvert\xi\rvert)^{-t/2}
\qquad(\xi\in\R^2).
\end{equation*}
We prove that \(\alphamin(\nu)\ge t\). Fix \(0<s<t\) and \(x_0\in\operatorname{spt}\nu\). For small \(r>0\), set \(\Lambda=r^{-2}\). Since \(x_0\in\mathbb S^1\),
\begin{equation*}
\begin{aligned}
\int_{\R}\lvert\wh\nu(Rx_0)\rvert^2e^{-\pi(R/\Lambda)^2}\,dR
&\le
C_t^2\int_{\R}(1+\lvert R\rvert)^{-t}e^{-\pi(R/\Lambda)^2}\,dR\\
&\le
C_s\Lambda^{1-s}.
\end{aligned}
\end{equation*}
On the other hand, Lemma~\ref{lem:radial-L2-local-mass-paper} implies
\begin{equation*}
\int_{\R}\lvert\wh\nu(Rx_0)\rvert^2e^{-\pi(R/\Lambda)^2}\,dR
\ge
c\Lambda\,\nu(B(x_0,cr))^2.
\end{equation*}
Combining the two estimates,
\begin{equation*}
\nu(B(x_0,cr))^2\le C_s\Lambda^{-s}=C_sr^{2s},
\end{equation*}
and hence \(\nu(B(x_0,cr))\le C_sr^s\). After changing constants, \(\nu(B(x_0,\rho))\le C_s\rho^s\) for all sufficiently small \(\rho>0\). Since \(\log_2\rho<0\),
\begin{equation*}
\liminf_{\rho\downarrow0}
\frac{\log_2\nu(B(x_0,\rho))}{\log_2\rho}
\ge s.
\end{equation*}
As \(x_0\in\operatorname{spt}\nu\) was arbitrary, \(\alphamin(\nu)\ge s\). Letting \(s\uparrow t\) yields \(\alphamin(\nu)\ge t\). Thus every admissible Fourier decay exponent \(t<1\) is bounded above by \(\alphamin(\nu)\).

It remains to exclude admissible exponents larger than \(1\). Since \(\nu\) is supported on \(\mathbb S^1\), one has \(\dimF(\nu)\le1\). Taking the supremum over admissible Fourier decay exponents and using the preceding conclusion for all \(t<1\), we obtain the desired claim.
\end{proof}

\begin{corollary}
\label{cor:endpoint-upper-bound-paper}
Assume the minimal Kahane--Peyri\`ere regime. Then, almost surely on \(\mathcal S_\circ\),
\begin{equation}\label{eq:endpoint-upper-bound-paper}
\dimF(\mu_\circ)\le \Aloc(W).
\end{equation}
\end{corollary}

\begin{proof}
On \(\mathcal S_\circ\), the measure \(\mu_\circ\) is a nonzero finite Borel measure on \(\mathbb S^1\). Proposition~\ref{prop:curved-support-upper-bound} and Theorem~\ref{thm:minimum-local-dimension} imply the desired claim.
\end{proof}

\section{A finite-\texorpdfstring{\(r\)}{r} annular Fourier decay theorem}
\label{sec:finite-r-annular-theorem}

This section proves the annular Fourier estimate that is the main analytic input for the circle endpoint theorem. The estimate is formulated for an auxiliary scalar weight \(U\), which need not coincide with the endpoint weight \(W\), and requires only the existence of a moment \(r>1\) giving a strict gap beyond the target exponent \(s\).

Let \(0<s<1\), \(r>1\), \(\delta>0\), and let 
\(U\ge0\) 
satisfy
\[
\E U=1,\qquad
\E[U^r]<\infty,\qquad
2^{1-r}\E[U^r]\le 2^{-r(s+\delta)}.
\]
Attach independent copies \(U_v\) of \(U\) to all nonempty binary words \(v\). For \(v\in\{0,1\}^n\), set
\begin{equation*}
Q_v=\prod_{j=1}^{n}U_{v|j},
\qquad
Q_{\varnothing}=1,
\end{equation*}
where \(v|j\) denotes the truncation of \(v\) to its first \(j\) symbols. Let \(\widetilde\nu_\ell\) be the level-\(\ell\) dyadic cascade measure on \([0,1)\), defined by
\[
d\widetilde\nu_\ell(t)
=
\sum_{\lvert v \rvert=\ell}Q_v\one_{J_v}(t)\,dt,
\qquad
\widetilde\nu_\ell(J_v)=2^{-\ell}Q_v\quad(\lvert v \rvert=\ell).
\]
Let \(\widetilde\nu\) be the almost sure weak limit of \(
(\widetilde\nu_\ell)_{\ell\ge0}\),
and define
\[
\nu_\circ=f_\#\widetilde\nu,
\qquad
f(t)=(\cos 2\pi t,\sin 2\pi t).
\]
Then, for \(\xi\in\R^2\),
\[
\wh{\nu_\circ}(\xi)
=
\int_0^1 e^{-2\pi i \xi\cdot f(t)}\,d\widetilde\nu(t).
\]

\subsection{Statement and parameter choices}
\label{subsec:finite-r-statement-parameters-paper}

\begin{theorem}
\label{thm:finite-r-annular}
Let \(0<s<1\), \(r>1\), and \(\delta>0\). Let \(U\ge0\) be a random variable with \(\E U=1\) and
\[
2^{1-r}\E[U^r]\le2^{-r(s+\delta)}.
\]
Let \(\nu_\circ=f_\#\widetilde\nu\) be the circle cascade generated by \(U\).  Then there exist constants \(C<\infty\), \(c>0\), and \(\eta>0\), depending only on \(s,r,\delta\) and on the law of \(U\), such that, for every \(n\ge1\),
\begin{equation}\label{eq:finite-r-annular-theorem-paper}
\Pbb\left(
\sup_{2^n\le\lvert\xi\rvert\le2^{n+1}}
\lvert\wh{\nu_\circ}(\xi)\rvert
>
C2^{-sn/2}
\right)
\le
C\exp(-c2^{\eta n})+C2^{-cn}.
\end{equation}
In particular,
\[
\lvert\wh{\nu_\circ}(\xi)\rvert=O(\lvert\xi\rvert^{-s/2})
\qquad (\lvert\xi\rvert\to\infty)
\]
almost surely.
\end{theorem}

The remainder of this section is devoted to the proof of Theorem~\ref{thm:finite-r-annular}. Choose \(
0<\delta_1<\min\{\delta,1-s\}\), 
and set \(a=s+\delta_1.\)
Then \(s<a<1\). The exponent \(a\) will serve as the local-mass exponent. The surplus \(a-s\) absorbs the losses arising from stationary tubes, phase bins, martingale estimates, compensator bounds, and the passage from a grid to the full annulus.

For \(k\ge0\) and \(n\ge1\), define the local-mass threshold
\begin{equation}\label{eq:local-mass-threshold-paper}
T_{k,n}=2^{\varepsilon n}2^{-ak},
\end{equation}
where \(\varepsilon>0\) will be chosen below. We also set
\begin{equation}\label{eq:cap-growth-factor-paper}
L_n=2^{\kappa n},
\end{equation}
where \(\kappa>0\) will be chosen below. For the \(r\)-tail compensator estimates, introduce 
\[\beta_{\mathrm{comp}}
=
r(s+\delta)-(r-1)(s+\delta_1)
=
s+r\delta-(r-1)\delta_1\quad \text{and}\quad \gamma_{\mathrm{comp}}=\min\{\beta_{\mathrm{comp}},1\}.\]

Whenever \(r>1\), \(0<\delta_1<\delta\), and \(s<1\), one has
\[
\beta_{\mathrm{comp}}>s,
\qquad
\gamma_{\mathrm{comp}}>s.
\]

\begin{lemma}
\label{lem:compatible-parameter-choice-paper}
The parameters \(
\delta_1,\varepsilon,\kappa>0\) can be chosen so that
\[
0<\delta_1<\min\{\delta,1-s\},
\qquad
20\varepsilon+\kappa<\frac{\delta_1}{2},
\]
and, with \(a=s+\delta_1\) and
\[
\vartheta=1-\frac{s}{2a}-\frac{8\varepsilon}{a},
\]
one has \(\vartheta>0\) and 
\[\beta_{\mathrm{comp}}\vartheta>\frac{s}{2},
\qquad
\gamma_{\mathrm{comp}}\vartheta>\frac{s}{2}.\]

\end{lemma}

\begin{proof}
Choose \(
0<\delta_1<\min\{\delta,1-s\}\), 
and set \(a=s+\delta_1\). Then \(s<a<1\), and since \(\delta_1<\delta\),
\[
\beta_{\mathrm{comp}}
=
s+r\delta-(r-1)\delta_1>s,
\qquad
\gamma_{\mathrm{comp}}=\min\{\beta_{\mathrm{comp}},1\}>s.
\]
For \(\varepsilon=0\), \(\vartheta_0=1-s/(2a)>1/2\), so
\[
\beta_{\mathrm{comp}}\vartheta_0>\frac{s}{2},
\qquad
\gamma_{\mathrm{comp}}\vartheta_0>\frac{s}{2}.
\]
By continuity, choose \(\varepsilon>0\) sufficiently small so that the same strict inequalities hold with
\[
\vartheta=1-\frac{s}{2a}-\frac{8\varepsilon}{a},
\]
while also ensuring that \(\vartheta>0\) and \(
20\varepsilon<\delta_1/2\).
Finally, choose
\[
0<\kappa<\frac{\delta_1}{2}-20\varepsilon.
\]
Then \(
20\varepsilon+\kappa<\delta_1/2\), as required.
\end{proof}

\begin{remark}
\label{rem:use-of-finite-r-parameters-paper}
The parameter \(\delta_1\) creates the local-mass surplus \(a-s\), \(\varepsilon\) is the annular slack in the local-mass thresholds and grid passage, and \(\kappa\) controls cap growth.  The inequalities
\[
\beta_{\mathrm{comp}}\vartheta>\frac{s}{2},
\qquad
\gamma_{\mathrm{comp}}\vartheta>\frac{s}{2}
\]
are used only in the compensator estimates, where they convert the generation gain from the \(r\)-tail bound into decay at the annular frequency scale.
\end{remark}

\subsection{Good events for local mass bounds}
\label{subsec:local-mass-good-events-paper}

Let \(\mathcal D_k^\circ=\{J_v:\lvert v \rvert=k\}\) denote the collection of level-\(k\) half-open dyadic intervals in \([0,1)\).

\begin{definition}
\label{def:local-mass-good-events-paper}
For \(n\ge1\), let \(
\mathcal G^{\mathrm{pre}}_n\) be the event that, for all \(k\ge0\), \(J\in\mathcal D_k^\circ\), and \(\ell\ge k\),
\[
\widetilde\nu_\ell(J)\le T_{k,n}.
\]
Let \(\mathcal G^{\mathrm{lim}}_n\) be the event that, for all \(k\ge0\) and \(J\in\mathcal D_k^\circ\),
\[
\widetilde\nu(J)\le T_{k,n}.
\]
Finally, set
\[
\mathcal G_n=\mathcal G^{\mathrm{pre}}_n\cap\mathcal G^{\mathrm{lim}}_n.
\]
\end{definition}

The prelimit event is stronger than the level-matched condition \(\widetilde\nu_k(J)\le T_{k,n}\), since the later martingale estimates require control of later prelimit masses on earlier dyadic intervals.

Let \(\widetilde Y_\ell=\widetilde\nu_\ell([0,1))\) and \(
\widetilde Y=\widetilde\nu([0,1))\).

\begin{lemma}
\label{lem:uniform-r-moment-total-mass-circle-paper}
Under the finite-\(r\) witnessed hypothesis,
\[
\sup_{\ell\ge0}\E[\widetilde Y_\ell^r]<\infty
\qquad\text{and}\qquad
\E[\widetilde Y^r]<\infty.
\]
\end{lemma}

\begin{proof}
The scalar cascade can be viewed as a dyadic vector cascade with normalized child coefficients \(U_0/2\) and \(U_1/2\), where \(U_0\) and \(U_1\) are independent copies of \(U\). By the finite-\(r\) witnessed hypothesis,
\[
\E\left[\left(\frac{U_0}{2}\right)^r+\left(\frac{U_1}{2}\right)^r\right]
=
2^{1-r}\E[U^r]
\le
2^{-r(s+\delta)}
<1.
\]
In addition,
\[
\E\left[\left(\frac{U_0+U_1}{2}\right)^r\right]
\le
\frac{\E U_0^r+\E U_1^r}{2}
<\infty.
\]
Therefore Lemma~\ref{lem:terminal-mass-moment-criterion} implies
\[
\sup_{\ell\ge0}\E[\widetilde Y_\ell^r]<\infty.
\]
Since \(\widetilde Y_\ell\to\widetilde Y\) almost surely, Fatou's lemma yields
\[
\E[\widetilde Y^r]\le\liminf_{\ell\to\infty}\E[\widetilde Y_\ell^r]<\infty.
\]
\end{proof}

\begin{lemma}
\label{lem:prelimit-limiting-r-mass-budgets-paper}
There exists a constant \(C_r<\infty\) such that, for all \(0\le k\le \ell\),
\begin{equation}\label{eq:prelimit-r-mass-budget-paper}
\E\left[
\sum_{J\in\mathcal D_k^\circ}
\widetilde\nu_\ell(J)^r
\right]
\le
C_r2^{-r(s+\delta)k}.
\end{equation}
Moreover, for all \(k\ge0\),
\begin{equation}\label{eq:limiting-r-mass-budget-paper}
\E\left[
\sum_{J\in\mathcal D_k^\circ}
\widetilde\nu(J)^r
\right]
\le
C_r2^{-r(s+\delta)k}.
\end{equation}
\end{lemma}

\begin{proof}
Fix \(J_v\in\mathcal D_k^\circ\) with \(\lvert v \rvert=k\). For \(\ell\ge k\),
\[
\widetilde\nu_\ell(J_v)
=
2^{-k}Q_v\widetilde Y_{\ell-k}^{(v)},
\]
where \(
\widetilde Y_{\ell-k}^{(v)}\) is independent of \(Q_v\) and has the same law as \(\widetilde Y_{\ell-k}\). Hence
\[
\E[\widetilde\nu_\ell(J_v)^r]
=
2^{-kr}\E[Q_v^r]\E[(\widetilde Y_{\ell-k})^r]
\le
C_r2^{-kr}(\E[U^r])^k.
\]
Summing over the \(2^k\) intervals \(J_v\), we obtain
\[
\E\left[
\sum_{J\in\mathcal D_k^\circ}
\widetilde\nu_\ell(J)^r
\right]
\le
C_r(2^{1-r}\E[U^r])^k
\le
C_r2^{-r(s+\delta)k}.
\]
This proves \eqref{eq:prelimit-r-mass-budget-paper}. For the limiting measure,
\[
\widetilde\nu(J_v)=2^{-k}Q_v\widetilde Y^{(v)},
\]
where \(\widetilde Y^{(v)}\) is an independent copy of \(\widetilde Y\). The same computation, now using \(\E[\widetilde Y^r]<\infty\), implies \eqref{eq:limiting-r-mass-budget-paper}.
\end{proof}

\begin{proposition}
\label{prop:local-mass-good-events-paper}
There exist constants \(C<\infty\) and \(c>0\) such that, for all \(n\ge1\),
\begin{equation}\label{eq:local-mass-good-event-probability-paper}
\Pbb(\mathcal G_n^c)\le C2^{-cn}.
\end{equation}
Consequently, \(
\sum_{n=1}^{\infty}\Pbb(\mathcal G_n^c)<\infty\). 
\end{proposition}

\begin{proof}
Fix \(k\ge0\) and \(J\in\mathcal D_k^\circ\). Since \(
(\widetilde\nu_\ell(J))_{\ell\ge k}\) is a nonnegative martingale, Doob's \(L^r\) maximal inequality implies, for every \(L\ge k\),
\[
\Pbb\left(
\max_{k\le \ell\le L}\widetilde\nu_\ell(J)>T_{k,n}
\right)
\le
\left(\frac{r}{r-1}\right)^r
T_{k,n}^{-r}\E[\widetilde\nu_L(J)^r].
\]
After summing over \(J\in\mathcal D_k^\circ\), applying \eqref{eq:prelimit-r-mass-budget-paper}, and letting \(L\to\infty\), we get
\[
\Pbb\left(
\exists J\in\mathcal D_k^\circ,\ \exists \ell\ge k:
\widetilde\nu_\ell(J)>T_{k,n}
\right)
\le
C T_{k,n}^{-r}2^{-r(s+\delta)k}.
\]
Using \(T_{k,n}=2^{\varepsilon n}2^{-ak}\) and \(a=s+\delta_1\), the right-hand side becomes \(C2^{-r\varepsilon n}2^{-r(\delta-\delta_1)k}\).

Summing over \(k\ge0\) yields
\[
\Pbb((\mathcal G^{\mathrm{pre}}_n)^c)\le C2^{-r\varepsilon n}.
\]

For the limiting event, Markov's inequality together with \eqref{eq:limiting-r-mass-budget-paper} implies
\[
\Pbb\left(
\exists J\in\mathcal D_k^\circ:
\widetilde\nu(J)>T_{k,n}
\right)
\le
C2^{-r\varepsilon n}2^{-r(\delta-\delta_1)k}.
\]
Summing over \(k\ge0\), we obtain \(\Pbb((\mathcal G^{\mathrm{lim}}_n)^c)\le C2^{-r\varepsilon n}\). 
Since
\[
\mathcal G_n^c
\subset
(\mathcal G^{\mathrm{pre}}_n)^c
\cup
(\mathcal G^{\mathrm{lim}}_n)^c,
\]
we obtain \(\Pbb(\mathcal G_n^c)\le C2^{-r\varepsilon n}\), which implies \eqref{eq:local-mass-good-event-probability-paper} after decreasing \(c>0\), if necessary. Summability is immediate.
\end{proof}

\begin{lemma}
\label{lem:frostman-consequences-good-event-paper}
Suppose that \(\mathcal G_n\) holds. Then the following estimates hold with a universal constant \(C<\infty\).

\begin{enumerate}[label=\textup{(\roman*)}]
\item For every interval \(B\subset[0,1)\),
\begin{equation}\label{eq:limiting-frostman-good-event-paper}
\widetilde\nu(B)\le C2^{\varepsilon n}\lvert B\rvert^a.
\end{equation}

\item For every \(\ell\ge0\) and every interval \(B\subset[0,1)\),
\begin{equation}\label{eq:prelimit-frostman-good-event-paper}
\widetilde\nu_\ell(B)
\le
C2^{\varepsilon n}(\lvert B\rvert+2^{-\ell})^a.
\end{equation}

\item For every \(\ell\ge0\) and every interval \(B\subset[0,1)\),
\begin{equation}\label{eq:prelimit-square-budget-good-event-paper}
\sum_{\substack{I\in\mathcal D_\ell^\circ\\ I\cap B\neq\varnothing}}
\widetilde\nu_\ell(I)^2
\le
C2^{2\varepsilon n}2^{-a\ell}(\lvert B\rvert+2^{-\ell})^a.
\end{equation}
\end{enumerate}
\end{lemma}

\begin{proof}
For (i), choose \(k\ge0\) such that \(2^{-(k+1)}<\lvert B\rvert\le2^{-k}\). 
Up to the wrap-around convention, the interval \(B\) can be covered by at most two intervals in \(\mathcal D_k^\circ\). On \(\mathcal G^{\mathrm{lim}}_n\), each such interval has \(\widetilde\nu\)-mass at most \(T_{k,n}=2^{\varepsilon n}2^{-ak}\). 
Hence
\[
\widetilde\nu(B)\le C2^{\varepsilon n}2^{-ak}\le C2^{\varepsilon n}\lvert B\rvert^a.
\]

For (ii), set \(\rho=\lvert B\rvert+2^{-\ell}\).
If \(\rho>1\), the estimate is immediate from the case \(k=0\) and the choice of the constant \(C\). Otherwise, choose \(0\le k\le\ell\) such that \(2^{-(k+1)}<\rho\le2^{-k}\).
The union of the level-\(\ell\) intervals meeting \(B\) is contained in an interval of length \(O(\rho)\), and hence in a bounded number of level-\(k\) intervals. On \(\mathcal G^{\mathrm{pre}}_n\),
\[
\widetilde\nu_\ell(B)\le C2^{\varepsilon n}2^{-ak}\le C2^{\varepsilon n}\rho^a.
\]
This proves \eqref{eq:prelimit-frostman-good-event-paper}.

For (iii), on \(\mathcal G^{\mathrm{pre}}_n\), every level-\(\ell\) interval \(I\) satisfies \(
\widetilde\nu_\ell(I)\le T_{\ell,n}=2^{\varepsilon n}2^{-a\ell}\).
Hence
\[
\sum_{\substack{I\in\mathcal D_\ell^\circ\\ I\cap B\neq\varnothing}}
\widetilde\nu_\ell(I)^2
\le
2^{\varepsilon n}2^{-a\ell}
\sum_{\substack{I\in\mathcal D_\ell^\circ\\ I\cap B\neq\varnothing}}
\widetilde\nu_\ell(I).
\]
The last sum is bounded by the prelimit mass of an enlarged interval of length \(O(\lvert B\rvert+2^{-\ell})\). Applying (ii) to this enlarged interval implies \eqref{eq:prelimit-square-budget-good-event-paper}.
\end{proof}

\begin{lemma}
\label{lem:good-event-total-mass-lipschitz-paper}
Suppose that \(\mathcal G^{\mathrm{lim}}_n\) holds. Then
\[
\nu_\circ(\mathbb S^1)=\widetilde\nu([0,1))\le 2^{\varepsilon n}.
\]
Consequently, for all \(\xi,\eta\in\R^2\),
\begin{equation}\label{eq:good-event-fourier-lipschitz-paper}
\lvert\wh{\nu_\circ}(\xi)-\wh{\nu_\circ}(\eta)\rvert
\le
2\pi 2^{\varepsilon n}\lvert\xi-\eta\rvert.
\end{equation}
\end{lemma}

\begin{proof} 
Taking \(k=0\) and \(J=[0,1)\) in \(
\mathcal G^{\mathrm{lim}}_n\), we obtain \(\widetilde\nu([0,1))\le T_{0,n}=2^{\varepsilon n}\). 
Since \(\nu_\circ=f_\#\widetilde\nu\), it follows that \(\nu_\circ(\mathbb S^1)\le2^{\varepsilon n}\).
Also, 
\[ \begin{aligned} \lvert\wh{\nu_\circ}(\xi)-\wh{\nu_\circ}(\eta)\rvert &\le \int_{\mathbb S^1} \left\lvert e^{-2\pi i x\cdot \xi} - e^{-2\pi i x\cdot \eta} \right\rvert\,d\nu_\circ(x)\\ &\le 2\pi\lvert\xi-\eta\rvert\,\nu_\circ(\mathbb S^1) \le 2\pi 2^{\varepsilon n}\lvert\xi-\eta\rvert, \end{aligned} \] 
using \(\lvert e^{-iu}-e^{-iv}\rvert\le \lvert u-v\rvert\) and \(\lvert x \rvert=1\) on \(\mathbb S^1\). 
\end{proof}

Proposition~\ref{prop:local-mass-good-events-paper} and the Borel--Cantelli lemma immediately imply the following.

\begin{corollary}
\label{cor:eventually-good-local-mass-events-paper}
Almost surely, there exists a finite random integer \(N_{\mathcal G}\) such that \(\mathcal G_n\) holds for all \(n\ge N_{\mathcal G}\).
\end{corollary}

\subsection{Stationary tubes and phase bins}
\label{subsec:stationary-tubes-phase-bins-paper}

Fix \(\xi\in\R^2\setminus\{0\}\), and write \(\Lambda=\lvert\xi\rvert\).  Set
\[
\phi_\xi(t)=-2\pi\,\xi\cdot f(t),
\qquad
f(t)=(\cos2\pi t,\sin2\pi t).
\]
Then
\[
\wh{\nu_\circ}(\xi)
=
\int_0^1 e^{i\phi_\xi(t)}\,d\widetilde\nu(t).
\]
If \(\xi=\Lambda(\cos\theta,\sin\theta)\), then
\[
\phi_\xi(t)=-2\pi\Lambda\cos(2\pi t-\theta),
\]
and
\[
\phi_\xi'(t)=4\pi^2\Lambda\sin(2\pi t-\theta),
\qquad
\phi_\xi''(t)=8\pi^3\Lambda\cos(2\pi t-\theta).
\]
The stationary set \(Z_\xi=\{t\in\R/\Z:\phi_\xi'(t)=0\}\) consists of two points.

We shall also use
\begin{equation}\label{eq:theta-bigger-than-s-over-two-paper}
\vartheta
=
1-\frac{s}{2a}-\frac{8\varepsilon}{a}
>
\frac{s}{2}.
\end{equation}
Indeed, at \(\varepsilon=0\), one has \(1-s/(2a)>1/2>s/2\), since \(a>s\) and \(s<1\); then choose \(\varepsilon>0\) sufficiently small.

\begin{lemma}
\label{lem:stationary-derivative-band-partition-paper}
There exists an absolute constant \(C\ge1\) such that, for every \(\xi\neq0\), there exist nonnegative \(C^1\) functions
\[
\chi_{\xi,\mathrm{stat}},
\qquad
\chi_{\xi,d}\quad(d\in\mathfrak D_\xi),
\]
on \(\R/\Z\), where \(\mathfrak D_\xi\) is a finite set of dyadic numbers, satisfying the following properties.

\begin{enumerate}[label=\textup{(\roman*)}]
\item Partition of unity:
\[
\chi_{\xi,\mathrm{stat}}(t)+\sum_{d\in\mathfrak D_\xi}\chi_{\xi,d}(t)=1
\qquad(t\in\R/\Z).
\]

\item Number of bands:
\[
\#\mathfrak D_\xi\le C(1+\log_2(2+\Lambda)).
\]

\item Stationary tube:
\[
\spt\chi_{\xi,\mathrm{stat}}
\subset
\{t:\operatorname{dist}(t,Z_\xi)\le C\Lambda^{-1/2}\}.
\]

\item Localization of derivative bands:
for every \(d\in\mathfrak D_\xi\),
\[
C^{-1}\Lambda^{-1/2}\le d\le C,
\]
and
\[
\spt\chi_{\xi,d}
\subset
\{t:C^{-1}d\le\operatorname{dist}(t,Z_\xi)\le Cd\}.
\]

\item Derivative size:
on \(\spt\chi_{\xi,d}\),
\[
C^{-1}\Lambda d\le \lvert\phi_\xi'(t)\rvert\le C\Lambda d.
\]

\item Cutoff derivative:
\[
\lvert\chi_{\xi,d}'(t)\rvert\le Cd^{-1}.
\]

\item Bounded overlap:
\[
\sum_{d\in\mathfrak D_\xi}\one_{\spt\chi_{\xi,d}}(t)\le C
\qquad(t\in\R/\Z).
\]
\end{enumerate}
\end{lemma}

\begin{proof}
We first treat the low-frequency case.  If \(\Lambda\le
\Lambda_0\), where \(\Lambda_0\) is a sufficiently large absolute constant,
we simply take
\[
\chi_{\xi,\mathrm{stat}}\equiv 1,
\qquad
\mathfrak D_\xi=\varnothing.
\]
Then all assertions are immediate after enlarging the absolute constant
\(C\).

Assume from now on that \(\Lambda>\Lambda_0\).  Let \(h_\xi(t)=\sin^2(2\pi t-\theta)\). Since \(Z_\xi\) is the zero set of \(\sin(2\pi t-\theta)\), we have the
uniform comparison
\[
h_\xi(t)^{1/2}\asymp \operatorname{dist}(t,Z_\xi)
\qquad(t\in\mathbb R/\mathbb Z).
\]
Choose a nonnegative \(C^\infty\) cutoff \(\psi\) on \([0,\infty)\) such that
\[
\psi(u)=1\quad(0\le u\le1),
\qquad
\psi(u)=0\quad(u\ge4),
\]
and set
\[
\chi_{\xi,\mathrm{stat}}(t)
=
\psi\bigl(\Lambda h_\xi(t)\bigr).
\]
Then
\[
\spt\chi_{\xi,\mathrm{stat}}
\subset
\{t:h_\xi(t)\le 4\Lambda^{-1}\}
\subset
\{t:\operatorname{dist}(t,Z_\xi)\le C\Lambda^{-1/2}\}.
\]

Next choose a smooth dyadic partition of unity on \((0,\infty)\). Let \(\rho:[0,\infty)\to[0,1]\) be nonincreasing and \(C^\infty\), with
\[
\rho(u)=1\quad(0\le u\le1),
\qquad
\rho(u)=0\quad(u\ge4),
\]
and set \(
\zeta(u)=\rho(u)-\rho(4u)\). 
Then \(\zeta\ge0\), \(\spt\zeta\subset[1/4,4]\), and, writing \(d=2^j\),
\[
\sum_{d\in2^{\mathbb Z}}\zeta\left(\frac{u}{d^2}\right)
=
\sum_{j\in\mathbb Z}
\left[
\rho\left(\frac{u}{4^j}\right)
-
\rho\left(\frac{u}{4^{j-1}}\right)
\right] 
=1,
\qquad u>0,
\]
by telescoping. This provides the required \(C^\infty\) dyadic partition of
unity.

Let \(\mathfrak D_\xi\) be the finite set of dyadic numbers \(d\) for which
\[
\zeta\left(\frac{h_\xi(t)}{d^2}\right)
\]
can be nonzero for some \(t\) with \(1-\chi_{\xi,\mathrm{stat}}(t)\neq0\).
For \(d\in\mathfrak D_\xi\), define
\[
\chi_{\xi,d}(t)
=
(1-\chi_{\xi,\mathrm{stat}}(t))
\zeta\left(\frac{h_\xi(t)}{d^2}\right).
\]
Then, for every \(t\),
\[
\chi_{\xi,\mathrm{stat}}(t)
+
\sum_{d\in\mathfrak D_\xi}\chi_{\xi,d}(t)
=
1.
\]

The support condition for the derivative bands follows from the support of
\(\zeta\).  If \(t\in\spt\chi_{\xi,d}\), then \(
h_\xi(t)\asymp d^2\), and hence \(
\operatorname{dist}(t,Z_\xi)\asymp d\). 
Thus
\[
\spt\chi_{\xi,d}
\subset
\{t:C^{-1}d\le \operatorname{dist}(t,Z_\xi)\le Cd\}.
\]
Moreover, since \(1-\chi_{\xi,\mathrm{stat}}\) vanishes inside a stationary
tube of radius \(\asymp\Lambda^{-1/2}\), any such \(d\) satisfies
\[
C^{-1}\Lambda^{-1/2}\le d\le C.
\]
It follows that the number of possible dyadic scales is bounded by
\[
\#\mathfrak D_\xi
\le
C(1+\log_2(2+\Lambda)).
\]

We next verify the derivative size.  Since
\[
\lvert\phi_\xi'(t)\rvert
=
4\pi^2\Lambda\lvert\sin(2\pi t-\theta)\rvert
=
4\pi^2\Lambda h_\xi(t)^{1/2},
\]
and since \(h_\xi(t)^{1/2}\asymp d\) on \(\spt\chi_{\xi,d}\), we get
\[
C^{-1}\Lambda d
\le
\lvert\phi_\xi'(t)\rvert
\le
C\Lambda d
\qquad(t\in\spt\chi_{\xi,d}).
\]
It remains to prove the cutoff derivative bounds. We have
\[
\lvert h_\xi'(t)\rvert\le C h_\xi(t)^{1/2}.
\]
On the support of \(\zeta(h_\xi/d^2)\), we obtain
\[
\left\lvert
\frac{d}{dt}\zeta\left(\frac{h_\xi(t)}{d^2}\right)
\right\rvert
\le
C\frac{\lvert h_\xi'(t)\rvert}{d^2}
\le
Cd^{-1}.
\]
Also,
\[
\lvert\chi_{\xi,\mathrm{stat}}'(t)\rvert
\le
C\Lambda \lvert h_\xi'(t)\rvert.
\]
This derivative is supported where \(h_\xi(t)\asymp\Lambda^{-1}\), and hence
where \(h_\xi(t)^{1/2}\asymp\Lambda^{-1/2}\).
If this region intersects the
support of \(\zeta(h_\xi/d^2)\), then \(d\asymp\Lambda^{-1/2}\), so
\[
\lvert\chi_{\xi,\mathrm{stat}}'(t)\rvert
\le
C\Lambda h_\xi(t)^{1/2}
\le
C\Lambda^{1/2}
\le
Cd^{-1}.
\]
Combining the two derivative estimates, we obtain \(
\lvert\chi_{\xi,d}'(t)\rvert\le Cd^{-1}\).

Finally, the bounded overlap follows from the dyadic construction: for each
fixed \(t\), the condition
\[
\zeta\left(\frac{h_\xi(t)}{d^2}\right)\neq0
\]
can hold for only \(O(1)\) dyadic values of \(d\).  Hence
\[
\sum_{d\in\mathfrak D_\xi}\one_{\spt\chi_{\xi,d}}(t)\le C.
\]
This proves all the stated properties.
\end{proof}

\begin{lemma}
\label{lem:stationary-tube-estimate-paper}
Assume \(
\mathcal G_n^{\mathrm{lim}}\) holds.  If \(
2^n\le\lvert\xi\rvert\le2^{n+1}\), then
\begin{equation}\label{eq:stationary-tube-estimate-paper}
\left|
\int e^{i\phi_\xi(t)}
\chi_{\xi,\mathrm{stat}}(t)\,d\widetilde\nu(t)
\right|
\le
C2^{-sn/2}2^{-c_{\mathrm{stat}}n}
\end{equation}
for some \(c_{\mathrm{stat}}>0\).
\end{lemma}

\begin{proof}
The absolute value is at most \(\widetilde\nu(\spt\chi_{\xi,\mathrm{stat}})\).
By Lemma~\ref{lem:stationary-derivative-band-partition-paper}, this support is contained in the union of two intervals, each of length at most \(C\Lambda^{-1/2}\). By \eqref{eq:limiting-frostman-good-event-paper},
\[
\widetilde\nu(\spt\chi_{\xi,\mathrm{stat}})
\le
C2^{\varepsilon n}\Lambda^{-a/2}.
\]
Since \(\Lambda\ge2^n\),
\[
2^{\varepsilon n}\Lambda^{-a/2}
\le
2^{-sn/2}2^{-(\delta_1/2-\varepsilon)n}.
\]
The parameter choice ensures \(\delta_1/2-\varepsilon>0\), so \eqref{eq:stationary-tube-estimate-paper} holds with \(c_{\mathrm{stat}}=\delta_1/2-\varepsilon\).
\end{proof}

We next separate derivative bands into those that are already small by local mass and those where oscillation is needed.

\begin{definition}
\label{def:mass-only-oscillatory-bands-paper}
Let \(2^n\le\lvert\xi\rvert\le2^{n+1}\).  
A derivative band \(d\in\mathfrak D_\xi\) is called mass-only at annular scale \(n\) if 
\[d^a\le2^{-sn/2}2^{-8\varepsilon n}.\]

It is called oscillatory, or non-mass, if \(d^a>2^{-sn/2}2^{-8\varepsilon n}\).

We denote the corresponding families by \(\mathfrak D_{\xi,n}^{\mathrm{mass}}\) and \(\mathfrak D_{\xi,n}^{\mathrm{osc}}\), respectively.
\end{definition}

\begin{lemma}
\label{lem:mass-only-band-estimate-paper}
Assume \(\mathcal G_n^{\mathrm{lim}}\) holds. If \(2^n\le\lvert\xi\rvert\le2^{n+1}\), then
\begin{equation}\label{eq:mass-only-band-sum-estimate-paper}
\sum_{d\in\mathfrak D_{\xi,n}^{\mathrm{mass}}}
\left|
\int e^{i\phi_\xi(t)}\chi_{\xi,d}(t)\,d\widetilde\nu(t)
\right|
\le
C2^{-sn/2}2^{-6\varepsilon n}.
\end{equation}
\end{lemma}

\begin{proof}
For a fixed \(d\), the support of \(\chi_{\xi,d}\) is contained in a bounded number of intervals of length at most \(Cd\). By \eqref{eq:limiting-frostman-good-event-paper},
\[
\widetilde\nu(\spt\chi_{\xi,d})
\le
C2^{\varepsilon n}d^a.
\]
If \(d\) is mass-only, then \(d^a\le2^{-sn/2}2^{-8\varepsilon n}\), hence
\[
\left|
\int e^{i\phi_\xi(t)}\chi_{\xi,d}(t)\,d\widetilde\nu(t)
\right|
\le
C2^{-sn/2}2^{-7\varepsilon n}.
\]
There are \(O(n)\) derivative bands in the annulus, and \(n2^{-7\varepsilon n}\le C2^{-6\varepsilon n}\).
 
This proves \eqref{eq:mass-only-band-sum-estimate-paper}.
\end{proof}

\begin{definition}
\label{def:phase-bin-scale-paper}
Let \(d\in\mathfrak D_\xi\).  The phase-bin scale \(m_{\xi,d}\) is the unique integer satisfying 
\[2^{m_{\xi,d}-1}<\Lambda d\leq 2^{m_{\xi,d}}.\]

Thus \(2^{-m_{\xi,d}}\asymp(\Lambda d)^{-1}\).
\end{definition}

\begin{lemma}
\label{lem:nonmass-phase-bin-lower-bound-paper}
Let \(2^n\le\lvert\xi\rvert\le2^{n+1}\). 
If \(d\in\mathfrak D_{\xi,n}^{\mathrm{osc}}\) and \(m=m_{\xi,d}\), then there is an absolute constant \(C_0<\infty\) such that 
\[m\ge \vartheta n-C_0, \qquad\text{where}\quad\vartheta=1-\frac{s}{2a}-\frac{8\varepsilon}{a}.\]
Moreover, \(m\le n+C_0\).
\end{lemma}

\begin{proof}
Since \(d\) is non-mass,
\[
d^a>2^{-sn/2}2^{-8\varepsilon n},
\qquad
d>2^{-sn/(2a)}2^{-8\varepsilon n/a}.
\]
Using \(\Lambda\ge2^n\),
we get \(\Lambda d\ge2^{\vartheta n}\).
By Definition~\ref{def:phase-bin-scale-paper}, \(m_{\xi,d}\ge\vartheta n-C_0\).
For the upper bound, derivative bands satisfy \(d\le C\), and \(\Lambda\le2^{n+1}\), hence \(\Lambda d\le C2^n\), so \(m_{\xi,d}\le n+C_0\) after increasing \(C_0\).
\end{proof}

Let \(d\in\mathfrak D_\xi\).  For \(I\in\mathcal D_\ell^\circ\) and a child \(J\in\mathcal D_{\ell+1}^\circ\) of \(I\), define 
\[c_J(\xi,d,\ell)
=
2^\ell\int_J e^{i\phi_\xi(t)}\chi_{\xi,d}(t)\,dt.\]

\begin{lemma}
\label{lem:prefix-postbin-coefficient-estimates-paper}
There is an absolute constant \(C<\infty\) such that, for every \(\xi\neq0\), \(d\in\mathfrak D_\xi\), \(\ell\ge0\), and child interval \(J\in\mathcal D_{\ell+1}^\circ\), 
\[
\lvert c_J(\xi,d,\ell)\rvert
\le C2^{\ell-m_{\xi,d}}
\qquad
(0\le\ell<m_{\xi,d}),
\]
and
\[
\lvert c_J(\xi,d,\ell)\rvert\le C
\qquad
(\ell\ge m_{\xi,d}).
\]
\end{lemma}

\begin{proof}
The second estimate is immediate from \(
\lvert c_J\rvert\le2^\ell\lvert J\rvert=1/2\).
 
Assume \(0\le\ell<m_{\xi,d}\). It suffices to prove
\[
\left\lvert
\int_J e^{i\phi_\xi(t)}\chi_{\xi,d}(t)\,dt
\right\rvert
\le
C(\Lambda d)^{-1},
\]
since multiplying by \(2^\ell\) yields \(\lvert c_J\rvert\le C2^\ell(\Lambda d)^{-1}\le C2^{\ell-m_{\xi,d}}\).

Let \(G(t)=\chi_{\xi,d}(t)\).  On \(\spt G\),
\[
\lvert\phi_\xi'(t)\rvert\ge C^{-1}\Lambda d,
\qquad
\lvert G'(t)\rvert\le Cd^{-1},
\qquad
\lvert\phi_\xi''(t)\rvert\le C\Lambda.
\]
The set \(J\cap\spt G\) has a bounded number of connected components. On each component, integration by parts yields
\[
\int e^{i\phi_\xi(t)}G(t)\,dt
=
\left[
\frac{e^{i\phi_\xi(t)}G(t)}{i\phi_\xi'(t)}
\right]_{\mathrm{endpoints}}
-
\int e^{i\phi_\xi(t)}
\left(
\frac{G'(t)}{i\phi_\xi'(t)}
-
\frac{G(t)\phi_\xi''(t)}{i(\phi_\xi'(t))^2}
\right)\,dt.
\]
The boundary term is \(O((\Lambda d)^{-1})\).  Also,
\[
\int_{\spt G}\frac{\lvert G'(t)\rvert}{\lvert\phi_\xi'(t)\rvert}\,dt
\le
C\int_{\spt G}\frac{d^{-1}}{\Lambda d}\,dt
\le
C(\Lambda d)^{-1},
\]
because \(\spt G\) has length \(O(d)\), and
\[
\int_{\spt G}\frac{\lvert G(t)\rvert\,\lvert\phi_\xi''(t)\rvert}{\lvert\phi_\xi'(t)\rvert^2}\,dt
\le
C\int_{\spt G}\frac{\Lambda}{(\Lambda d)^2}\,dt
\le
C(\Lambda d)^{-1}.
\]
Combining these bounds proves the prefix estimate.
\end{proof}

For \(d\in\mathfrak D_\xi\), define
\[
F^{(0)}_{\xi,d}
=
\int_0^1 e^{i\phi_\xi(t)}\chi_{\xi,d}(t)\,dt.
\]

\begin{lemma}
\label{lem:arclength-forcing-nonmass-paper}
If \(2^n\le\lvert\xi\rvert\le2^{n+1}\) and \(d\in\mathfrak D_{\xi,n}^{\mathrm{osc}}\), then
\[
\lvert F^{(0)}_{\xi,d}\rvert
\le
C2^{-sn/2}2^{-c_{\mathrm{arc}}n}
\]
for some \(c_{\mathrm{arc}}>0\).  Consequently,
\[
\sum_{d\in\mathfrak D_{\xi,n}^{\mathrm{osc}}}
\lvert F^{(0)}_{\xi,d}\rvert
\le
C2^{-sn/2}2^{-c_{\mathrm{arc}}n}
\]
after decreasing \(c_{\mathrm{arc}}\) if necessary.
\end{lemma}

\begin{proof}
The integration-by-parts estimate from Lemma~\ref{lem:prefix-postbin-coefficient-estimates-paper} yields
\[
\lvert F^{(0)}_{\xi,d}\rvert\le C(\Lambda d)^{-1}.
\]
Since \(d\) is non-mass, \(d^{-1}<2^{sn/(2a)}2^{8\varepsilon n/a}\).
Using \(\Lambda\ge2^n\),
\[
(\Lambda d)^{-1}
\le
2^{-n}2^{sn/(2a)}2^{8\varepsilon n/a}
=
2^{-\vartheta n}=
2^{-sn/2}2^{-(\vartheta-s/2)n}.
\]
By \eqref{eq:theta-bigger-than-s-over-two-paper}, \(\vartheta>s/2\). 
This proves the first estimate with \(
c_{\mathrm{arc}}=\vartheta-s/2>0\).
The \(O(n)\) derivative bands are absorbed into the exponential decay after decreasing \(c_{\mathrm{arc}}>0\), if necessary.
\end{proof}

We now define the exact martingale arrays. Let \(I\in\mathcal D_\ell^\circ\), let \(J\subset I\) be a dyadic child, and write \(M_I=\widetilde\nu_\ell(I)\).  If \(J\) corresponds to an extension by \(i\in\{0,1\}\), let \(U_J\) be the fresh scalar weight on that edge.  Then
\[
\widetilde\nu_{\ell+1}(J)=\frac12M_IU_J.
\]

\begin{definition}
\label{def:raw-prefix-postbin-arrays-paper}
Let \(d\in\mathfrak D_\xi\), and set \(m=m_{\xi,d}\).  Define 
\[F^{\mathrm{pre}}_{\xi,d}
=
\sum_{\ell=0}^{m-1}
\sum_{I\in\mathcal D_\ell^\circ}
\sum_{\substack{J\in\mathcal D_{\ell+1}^\circ\\ J\subset I}}
c_J(\xi,d,\ell)M_I(U_J-1).\]
For \(L>m\), define 
\[F^{\mathrm{post}}_{\xi,d,L}
=
\sum_{\ell=m}^{L-1}
\sum_{I\in\mathcal D_\ell^\circ}
\sum_{\substack{J\in\mathcal D_{\ell+1}^\circ\\ J\subset I}}
c_J(\xi,d,\ell)M_I(U_J-1).\]
Whenever the limit exists, set \(F^{\mathrm{post}}_{\xi,d}=\lim_{L\to\infty}F^{\mathrm{post}}_{\xi,d,L}\).
\end{definition}

\begin{lemma}
\label{lem:exact-martingale-decomposition-band-paper}
Let \(d\in\mathfrak D_\xi\).  On the weak-convergence event \(\widetilde\nu_L\weak\widetilde\nu\), the post-bin limit \(F^{\mathrm{post}}_{\xi,d}\) exists, and
\begin{equation}\label{eq:exact-band-decomposition-paper}
\int e^{i\phi_\xi(t)}\chi_{\xi,d}(t)\,d\widetilde\nu(t)
=
F^{(0)}_{\xi,d}
+
F^{\mathrm{pre}}_{\xi,d}
+
F^{\mathrm{post}}_{\xi,d}.
\end{equation}
\end{lemma}

\begin{proof}
Let \(G_{\xi,d}(t)=e^{i\phi_\xi(t)}\chi_{\xi,d}(t)\).
For \(L\ge1\),
\[
\int G_{\xi,d}\,d\widetilde\nu_L
=
\int G_{\xi,d}(t)\,dt
+
\sum_{\ell=0}^{L-1}
\int G_{\xi,d}\,d(\widetilde\nu_{\ell+1}-\widetilde\nu_\ell).
\]
The first term is \(F^{(0)}_{\xi,d}\).  For a parent \(I\in\mathcal D_\ell^\circ\) and child \(J\subset I\), the densities of \(\widetilde\nu_\ell\) on \(I\) and \(\widetilde\nu_{\ell+1}\) on \(J\) are \(2^\ell M_I\) and \(2^\ell M_IU_J\), respectively.  Hence
\[ \int_JG_{\xi,d}\,d(\widetilde\nu_{\ell+1}-\widetilde\nu_\ell) = M_I(U_J-1)\,2^\ell\int_JG_{\xi,d}(t)\,dt=
M_I(U_J-1)\,c_J(\xi,d,\ell). \]
Summing over children and generations, we obtain, for \(L>m_{\xi,d}\),
\[
\int G_{\xi,d}\,d\widetilde\nu_L
=
F^{(0)}_{\xi,d}
+
F^{\mathrm{pre}}_{\xi,d}
+
F^{\mathrm{post}}_{\xi,d,L}.
\]
Since \(G_{\xi,d}\) is continuous and \(
\widetilde\nu_L\weak\widetilde\nu\), the left side converges to \(
\int G_{\xi,d}\,d\widetilde\nu\).  
Thus the post-bin partial sums converge and satisfy \eqref{eq:exact-band-decomposition-paper}.
\end{proof}

We now assemble the deterministic reduction.

\begin{proposition}[Stationary-tube and phase-bin reduction]
\label{prop:stationary-tube-phase-bin-reduction-paper}
Assume \(\mathcal G_n\) holds and \(
2^n\le\lvert\xi\rvert\le2^{n+1}\). 
On the weak-convergence event for \(\widetilde\nu\),
\begin{equation}\label{eq:stationary-phase-bin-reduction-paper}
\wh{\nu_\circ}(\xi)
=
E^{\mathrm{safe}}_{\xi,n}
+
\sum_{d\in\mathfrak D_{\xi,n}^{\mathrm{osc}}}
\left(
F^{\mathrm{pre}}_{\xi,d}
+
F^{\mathrm{post}}_{\xi,d}
\right),
\end{equation}
where 
\[\lvert E^{\mathrm{safe}}_{\xi,n}\rvert
\le
C2^{-sn/2}2^{-c_{\mathrm{safe}}n} \qquad \text{for some}\quad c_{\mathrm{safe}}>0.\]
Moreover, for every \(d\in\mathfrak D_{\xi,n}^{\mathrm{osc}}\), the coefficients satisfy
\[
\lvert c_J(\xi,d,\ell)\rvert\le C2^{\ell-m_{\xi,d}}
\quad(0\le\ell<m_{\xi,d}),
\qquad
\lvert c_J(\xi,d,\ell)\rvert\le C
\quad(\ell\ge m_{\xi,d}).
\]
\end{proposition}

\begin{proof}
By Lemma~\ref{lem:stationary-derivative-band-partition-paper},
\[
1=\chi_{\xi,\mathrm{stat}}+\sum_{d\in\mathfrak D_\xi}\chi_{\xi,d}.
\]
Therefore
\[
\wh{\nu_\circ}(\xi)
=
\int e^{i\phi_\xi(t)}\chi_{\xi,\mathrm{stat}}(t)\,d\widetilde\nu(t)
+
\sum_{d\in\mathfrak D_\xi}
\int e^{i\phi_\xi(t)}\chi_{\xi,d}(t)\,d\widetilde\nu(t).
\]
The stationary term is bounded by Lemma~\ref{lem:stationary-tube-estimate-paper}.  The sum over mass-only bands is bounded by Lemma~\ref{lem:mass-only-band-estimate-paper}. We include these terms in \(E^{\mathrm{safe}}_{\xi,n}\). 

For each non-mass band \(
d\in\mathfrak D_{\xi,n}^{\mathrm{osc}}\), Lemma~\ref{lem:exact-martingale-decomposition-band-paper} provides the decomposition
\[
\int e^{i\phi_\xi(t)}\chi_{\xi,d}(t)\,d\widetilde\nu(t)
=
F^{(0)}_{\xi,d}
+
F^{\mathrm{pre}}_{\xi,d}
+
F^{\mathrm{post}}_{\xi,d}.
\]
The arclength term \(F^{(0)}_{\xi,d}\) is bounded by Lemma~\ref{lem:arclength-forcing-nonmass-paper}; summing over all non-mass bands still contributes a term of size
\(C2^{-sn/2}2^{-c n}\)
for some \(c>0\).  These arclength terms are also included in \(E^{\mathrm{safe}}_{\xi,n}\).

Combining the stationary term, the mass-only band terms, and the arclength forcing terms, we obtain \eqref{eq:stationary-phase-bin-reduction-paper} with
\[
\lvert E^{\mathrm{safe}}_{\xi,n}\rvert
\le
C2^{-sn/2}
\left(
2^{-c_{\mathrm{stat}}n}
+
2^{-6\varepsilon n}
+
2^{-c_{\mathrm{arc}}n}
\right).
\]
After decreasing the exponent and increasing \(C\), this is
\[
\lvert E^{\mathrm{safe}}_{\xi,n}\rvert
\le
C2^{-sn/2}2^{-c_{\mathrm{safe}}n}.
\]
The coefficient estimates are exactly Lemma~\ref{lem:prefix-postbin-coefficient-estimates-paper}.
\end{proof}

\subsection{Predictable capping and martingale concentration}
\label{subsec:predictable-capping-martingale-paper}

We now estimate the martingale arrays appearing in Proposition~\ref{prop:stationary-tube-phase-bin-reduction-paper}.  Since \(U\) is not assumed to be bounded, we truncate each child weight at a predictable cap, center the truncated increment, and then apply martingale concentration.  The resulting compensator terms are estimated in Subsection~\ref{subsec:r-tail-compensator-paper}.

Let \(
\mathcal F_\ell=\sigma\{U_v:1\le \lvert v \rvert\le \ell\}\).
For \(I\in\mathcal D_\ell^\circ\), set \(M_I=\widetilde\nu_\ell(I)\).  
For a child \(J\subset I\), let \(U_J\) denote the fresh weight assigned to the edge \(I\to J\).  Then \(\widetilde\nu_{\ell+1}(J)=M_IU_J/2\).

\begin{definition}[Local-mass stopping time]
\label{def:local-mass-stopping-time-paper}
For \(n\ge1\), define
\[
\tau_n
=
\inf\left\{
\ell\ge0:
\exists\,0\le k\le \ell,\ \exists J\in\mathcal D_k^\circ
\text{ such that }
\widetilde\nu_\ell(J)>T_{k,n}
\right\},
\]
with the convention \(\inf\varnothing=+\infty\).
\end{definition}

\begin{lemma}
\label{lem:stopping-time-localization-paper}
The random variable \(\tau_n\) is a stopping time with respect to the filtration \((\mathcal F_\ell)_{\ell\ge0}\). 
Moreover,
\[
\mathcal G_n^{\mathrm{pre}}=\{\tau_n=\infty\}.
\]
\end{lemma}

\begin{proof}
For fixed \(\ell\), the event \(\{\tau_n\le \ell\}\) is determined by the collection of masses \(\widetilde\nu_j(J)\), \(0\le j\le\ell\), \(J\in\mathcal D_k^\circ\), \(0\le k\le j\), all of which are \(\mathcal F_\ell\)-measurable. Hence \(\tau_n\) is a stopping time. The identity \(\mathcal G_n^{\mathrm{pre}}=\{\tau_n=\infty\}\) follows directly from the definitions.
\end{proof}

\begin{definition}[Predictable cap]
\label{def:predictable-cap-paper}
Let \(I\in\mathcal D_\ell^\circ\).  Define
\[
C_{I,n}
=
\begin{cases}
\displaystyle
L_n\,\frac{2T_{\ell+1,n}}{M_I}, & M_I>0,\\[8pt]
+\infty, & M_I=0.
\end{cases}
\]
For a child \(J\subset I\), set \(
U_{J,n}^{\mathrm{cap}}=U_J\wedge C_{I,n}\), and define
\[
\overline U_{I,n}
=
\E[U_{J,n}^{\mathrm{cap}}\mid\mathcal F_\ell].
\]
This quantity is independent of the choice of the child \(J\subset I\), since the fresh child weights are independent copies of \(U\).
\end{definition}

The cap \(C_{I,n}\) is \(\mathcal F_\ell\)-measurable, whereas \(U_J\) is independent of \(\mathcal F_\ell\).

\begin{definition}[Centered capped increment]
\label{def:centered-capped-increment-paper}
Let \(d\in\mathfrak D_{\xi,n}^{\mathrm{osc}},\) 
and write \(c_J=c_J(\xi,d,\ell)\). For \(I\in\mathcal D_\ell^\circ\) and a child \(J\subset I\), define 
\[X_{J,n}^{\mathrm{cap}}(\xi,d,\ell)
=
c_JM_I
\left(
U_{J,n}^{\mathrm{cap}}-\overline U_{I,n}
\right).\]
\end{definition}

\begin{lemma}
\label{lem:centered-capped-basic-bounds-paper}
For every child \(J\subset I\), \(I\in\mathcal D_\ell^\circ\),
\[
\E\left[
X_{J,n}^{\mathrm{cap}}(\xi,d,\ell)
\mid
\mathcal F_\ell
\right]=0.
\]
Moreover,
\begin{equation}\label{eq:capped-jump-bound-paper}
\lvert X_{J,n}^{\mathrm{cap}}(\xi,d,\ell)\rvert
\le
C\lvert c_J\rvert L_nT_{\ell+1,n}
\end{equation}
almost surely, and
\begin{equation}\label{eq:capped-variance-bound-paper}
\E\left[
\lvert X_{J,n}^{\mathrm{cap}}(\xi,d,\ell)\rvert^2
\mid
\mathcal F_\ell
\right]
\le
CL_n\lvert c_J\rvert^2M_IT_{\ell+1,n}.
\end{equation}
\end{lemma}

\begin{proof}
The centeredness follows immediately from the definition of \(\overline U_{I,n}\).  If \(M_I=0\), then \(X_{J,n}^{\mathrm{cap}}=0\).  We may therefore assume that \(M_I>0\).  Since \(0\le U_{J,n}^{\mathrm{cap}}\le C_{I,n}\) and \(C_{I,n}\) is \(\mathcal F_\ell\)-measurable,
\[
0\le\overline U_{I,n}\le C_{I,n},
\qquad
\lvert U_{J,n}^{\mathrm{cap}}-\overline U_{I,n}\rvert\le C_{I,n}.
\]
Therefore
\[
\lvert X_{J,n}^{\mathrm{cap}}\rvert
\le
\lvert c_J\rvert M_IC_{I,n}
=
2\lvert c_J\rvert L_nT_{\ell+1,n},
\]
which proves \eqref{eq:capped-jump-bound-paper}.  For the variance,
\[
\operatorname{Var}(U_{J,n}^{\mathrm{cap}}\mid\mathcal F_\ell)
\le
\E[(U_{J,n}^{\mathrm{cap}})^2\mid\mathcal F_\ell].
\]
Since \((U_{J,n}^{\mathrm{cap}})^2\le C_{I,n}U_{J,n}^{\mathrm{cap}}\le C_{I,n}U_J\), 
\[
\E[(U_{J,n}^{\mathrm{cap}})^2\mid\mathcal F_\ell]\le C_{I,n}.
\]
Thus
\[
\E[\lvert X_{J,n}^{\mathrm{cap}}\rvert^2\mid\mathcal F_\ell]
\le
\lvert c_J\rvert^2M_I^2C_{I,n}
=
2L_n\lvert c_J\rvert^2M_IT_{\ell+1,n},
\]
which establishes \eqref{eq:capped-variance-bound-paper}.
\end{proof}

\begin{lemma}
\label{lem:raw-capped-compensator-identity-paper}
On \(\mathcal G_n^{\mathrm{pre}}\), for every child \(J\subset I\) with \(M_I>0\), we have
\(U_J\le C_{I,n}\).
Consequently, on \(\mathcal G_n^{\mathrm{pre}}\), 
\begin{equation}\label{eq:raw-capped-compensator-identity-paper}
c_JM_I(U_J-1)
=
X_{J,n}^{\mathrm{cap}}+D_{J,n},
\end{equation}
where
\begin{equation}\label{eq:compensator-one-step-paper}
D_{J,n}
=
c_JM_I(\overline U_{I,n}-1)
=
-c_JM_I
\E[(U_J-C_{I,n})_+\mid\mathcal F_\ell].
\end{equation}
\end{lemma}

\begin{proof}
On \(\mathcal G_n^{\mathrm{pre}}\), \(
\widetilde\nu_{\ell+1}(J)\le T_{\ell+1,n}\).
Since \(\widetilde\nu_{\ell+1}(J)=M_IU_J/2\),

if \(M_I>0\), then
\[
U_J\le \frac{2T_{\ell+1,n}}{M_I}\le L_n\frac{2T_{\ell+1,n}}{M_I}=C_{I,n}.
\]
Thus, on \(\mathcal G_n^{\mathrm{pre}}\), we have \(U_{J,n}^{\mathrm{cap}}=U_J\), and
\[
c_JM_I(U_J-1)
=
c_JM_I(U_{J,n}^{\mathrm{cap}}-\overline U_{I,n})
+
c_JM_I(\overline U_{I,n}-1).
\]
This proves \eqref{eq:raw-capped-compensator-identity-paper}.  Finally, since \(\E U_J=1\),
\[
\overline U_{I,n}
=
\E[U_J-(U_J-C_{I,n})_+\mid\mathcal F_\ell]
=
1-\E[(U_J-C_{I,n})_+\mid\mathcal F_\ell].
\]
\end{proof}

We next estimate the centered capped arrays.  For a non-mass band \(d\in\mathfrak D_{\xi,n}^{\mathrm{osc}}\), write
\[
m=m_{\xi,d},
\qquad
S_{\xi,d}=\spt\chi_{\xi,d}.
\]
Recall that \(S_{\xi,d}\) is contained in a bounded number of intervals, each of length at most \(Cd\).

\begin{lemma}
\label{lem:prefix-postbin-variance-budgets-paper}
Assume that the local-mass bounds defining \(\mathcal G_n^{\mathrm{pre}}\) hold up to the levels under consideration. If \(2^n\le\lvert\xi\rvert\le2^{n+1}\) and \(
d\in\mathfrak D_{\xi,n}^{\mathrm{osc}}\), then
\begin{equation}\label{eq:prefix-variance-budget-paper}
\sum_{\ell=0}^{m-1}
\sum_{I\in\mathcal D_\ell^\circ}
\sum_{\substack{J\in\mathcal D_{\ell+1}^\circ\\ J\subset I}}
\lvert c_J(\xi,d,\ell)\rvert^2M_IT_{\ell+1,n}
\le
C2^{2\varepsilon n}\lvert\xi\rvert^{-a},
\end{equation}
and
\begin{equation}\label{eq:postbin-variance-budget-paper}
\sum_{\ell=m}^{\infty}
\sum_{I\in\mathcal D_\ell^\circ}
\sum_{\substack{J\in\mathcal D_{\ell+1}^\circ\\ J\subset I}}
\lvert c_J(\xi,d,\ell)\rvert^2M_IT_{\ell+1,n}
\le
C2^{2\varepsilon n}\lvert\xi\rvert^{-a}.
\end{equation}
\end{lemma}

\begin{proof}
Only intervals meeting a fixed \(C2^{-\ell}\)-neighborhood \(S_{\xi,d}^{(\ell)}\) of \(S_{\xi,d}\) can contribute.  For \(0\le\ell<m\), Lemma~\ref{lem:prefix-postbin-coefficient-estimates-paper} yields \(\lvert c_J\rvert\le C2^{\ell-m}\), and \eqref{eq:prelimit-frostman-good-event-paper} implies
\[
\sum_{\substack{I\in\mathcal D_\ell^\circ\\ I\cap S_{\xi,d}^{(\ell)}\neq\varnothing}}
M_I
\le
C2^{\varepsilon n}(d+2^{-\ell})^a.
\]
Since \(T_{\ell+1,n}\le C2^{\varepsilon n}2^{-a\ell}\), the prefix sum is at most
\[
C2^{2\varepsilon n}
\sum_{\ell=0}^{m-1}
2^{2(\ell-m)}2^{-a\ell}(d+2^{-\ell})^a.
\]
We claim that the last sum is \(O(\lvert\xi\rvert^{-a})\).
We split the sum into the two ranges \(2^{-\ell}\ge d\) and \(2^{-\ell}<d\).  In the first range,
\[
(d+2^{-\ell})^a\le C2^{-a\ell},
\]
so the summand is bounded by \(C2^{-2m}2^{(2-2a)\ell}\).  Since \(a<1\),
\[
\sum_{2^{-\ell}\ge d}
2^{-2m}2^{(2-2a)\ell}
\le
C2^{-2m}d^{-(2-2a)}.
\]
Using \(2^{-m}\asymp(\lvert\xi\rvert d)^{-1}\) and \(d\gtrsim\lvert\xi\rvert^{-1/2}\), this is bounded by \(C\lvert\xi\rvert^{-a}\).  In the second range,
\[
(d+2^{-\ell})^a\le Cd^a,
\]
so
\[
C2^{-2m}d^a\sum_{\ell<m}2^{(2-a)\ell}
\le
Cd^a2^{-am}
\asymp
C\lvert\xi\rvert^{-a}.
\]
This proves \eqref{eq:prefix-variance-budget-paper}.

For \(\ell\ge m\), we have \(\lvert c_J\rvert\le C\).
Since \(2^{-\ell}\le2^{-m}\lesssim(\lvert\xi\rvert d)^{-1}\lesssim d\), we have \(d+2^{-\ell}\le Cd\). Therefore the level-\(\ell\) contribution is bounded by \(C2^{2\varepsilon n}2^{-a\ell}d^a\).
After summing over \(\ell\ge m\), we obtain
\[
C2^{2\varepsilon n}d^a2^{-am}
\asymp
C2^{2\varepsilon n}\lvert\xi\rvert^{-a},
\]
which proves \eqref{eq:postbin-variance-budget-paper}.
\end{proof}

We shall use Freedman's inequality in the following normalized form from \cite{Freedman1975}.

\begin{theorem}[Classical Freedman inequality]
\label{thm:freedman-classical}
Let \((\mathcal F_n)_{n\ge0}\) be an increasing sequence of
\(\sigma\)-fields, and let \(X_1,X_2,\ldots\) be real-valued random variables
such that \(X_n\) is \(\mathcal F_n\)-measurable. Assume that
\[
\lvert X_n\rvert\le 1,
\qquad
\E\left(X_n\mid \mathcal F_{n-1}\right)=0
\]
almost surely for every \(n\ge1\). Define
\(S_0=0\), \(S_n=X_1+\cdots+X_n,\)
and
\[
V_n=\Var\left(X_n\mid\mathcal F_{n-1}\right),
\qquad
T_n=V_1+\cdots+V_n,
\qquad T_0=0.
\]
Then, for all \(a,b>0\),
\[
\Pbb\left(
S_n>a \ \text{and}\ T_n\le b
\ \text{for some } n
\right)
\le
\exp\left(
-\frac{a^2}{2(a+b)}
\right).
\]
\end{theorem}

\begin{lemma}[Rescaled real-valued Freedman inequality]
\label{lem:real-freedman-rescaled}
Let \((\mathcal H_j)_{j=0}^N\) be a filtration, and let \((M_j,\mathcal H_j)_{j=0}^N\) be a real-valued martingale with \(M_0=0\) and martingale differences \(\xi_j=M_j-M_{j-1}\).  Let \(R,V\in(0,\infty)\).  Assume that \(
\lvert \xi_j\rvert\le R\) for every \(j\) and
\[
\sum_{j=1}^N
\E\left(\xi_j^2\mid \mathcal H_{j-1}\right)
\le V
\]
almost surely.  Then, for every \(u>0\),
\[
\Pbb\left(
\max_{0\le j\le N} M_j>u
\right)
\le
\exp\left(
-\frac{u^2}{2(V+Ru)}
\right).
\]
\end{lemma}

\begin{proof}
Set \(X_j=\xi_j/R\).  Then \((X_j)\) is a martingale difference sequence with \(\lvert X_j\rvert\le1\), and for \(0\le n\le N\),
\[
\sum_{j=1}^n X_j=\frac{M_n}{R},
\qquad
\sum_{j=1}^n\Var(X_j\mid\mathcal H_{j-1})
\le
\frac{V}{R^2}.
\]
If \(\max_{0\le j\le N}M_j>u\), then for some \(1\le n\le N\),
\[
\sum_{j=1}^nX_j>\frac{u}{R},
\qquad
\sum_{j=1}^n\Var(X_j\mid\mathcal H_{j-1})\le\frac{V}{R^2}.
\]
Applying Theorem~\ref{thm:freedman-classical} with \(a=u/R\) and \(b=V/R^2\), we get
\[
\Pbb\left(
\max_{0\le j\le N}M_j>u
\right)
\le
\exp\left(
-\frac{(u/R)^2}{2(u/R+V/R^2)}
\right)
=
\exp\left(
-\frac{u^2}{2(V+Ru)}
\right).
\]
\end{proof}

\begin{lemma}[Complex Freedman inequality]
\label{lem:complex-freedman-paper}
Let \(
(S_j,\mathcal H_j)_{j=0}^{N}\) be a complex-valued martingale with \(S_0=0\) and differences
\(\Delta_j=S_j-S_{j-1}\), \(1\le j\le N.\)
Let \(R,V\in(0,\infty)\). Assume that \(\lvert\Delta_j\rvert\le R\)
almost surely for every \(j\), and
\[
\sum_{j=1}^{N}
\E\left(\lvert\Delta_j\rvert^2\mid\mathcal H_{j-1}\right)
\le V
\]
almost surely.  Then, for every \(t>0\), 
\[\Pbb\left(
\max_{0\le j\le N}\lvert S_j\rvert>t
\right)
\le
4\exp\left(
-\frac{t^2}{8(V+Rt)}
\right).\]

\end{lemma}

\begin{proof}
Write
\(S_j=X_j+iY_j\), where 
\(
X_j=\operatorname{Re}S_j
\) 
and 
\(
Y_j=\operatorname{Im}S_j.
\)
Then \(X_j\) and \(Y_j\) are real-valued martingales with differences
\(\operatorname{Re}\Delta_j\) and 
\(
\operatorname{Im}\Delta_j,
\) 
respectively. Moreover,
\(
\lvert\operatorname{Re}\Delta_j\rvert,\ \lvert\operatorname{Im}\Delta_j\rvert
\le \lvert\Delta_j\rvert\le R,
\)
and
\[
\sum_{j=1}^N
\E\bigl((\operatorname{Re}\Delta_j)^2\mid\mathcal H_{j-1}\bigr)
\le V,
\qquad
\sum_{j=1}^N
\E\bigl((\operatorname{Im}\Delta_j)^2\mid\mathcal H_{j-1}\bigr)
\le V.
\]
The same bounds also hold for \(-X_j\) and \(-Y_j\). Therefore, applying
Lemma~\ref{lem:real-freedman-rescaled} to
\(X_j,-X_j,Y_j,-Y_j\) with \(u=t/2\), and then using the union bound,
\[
\Pbb\left(
\max_{0\le j\le N}\lvert S_j\rvert>t
\right)
\le
4\exp\left(
-\frac{(t/2)^2}{2(V+Rt/2)}
\right)
\le
4\exp\left(
-\frac{t^2}{8(V+Rt)}
\right).
\]
This proves the lemma.
\end{proof}

\begin{definition}[Centered capped arrays]
\label{def:centered-capped-arrays-paper}
For \(d\in\mathfrak D_{\xi,n}^{\mathrm{osc}}\), 
set \(m=m_{\xi,d}\).
Define the stopped centered capped prefix array by
\[
\widehat F^{\mathrm{pre,cap}}_{\xi,d,n}
=
\sum_{\ell=0}^{m-1}
\one_{\{\tau_n>\ell\}}
\sum_{I\in\mathcal D_\ell^\circ}
\sum_{\substack{J\in\mathcal D_{\ell+1}^\circ\\J\subset I}}
X_{J,n}^{\mathrm{cap}}(\xi,d,\ell).
\]
For \(L>m\), define the stopped centered capped post-bin partial sum by
\[
\widehat F^{\mathrm{post,cap}}_{\xi,d,n,L}
=
\sum_{\ell=m}^{L-1}
\one_{\{\tau_n>\ell\}}
\sum_{I\in\mathcal D_\ell^\circ}
\sum_{\substack{J\in\mathcal D_{\ell+1}^\circ\\J\subset I}}
X_{J,n}^{\mathrm{cap}}(\xi,d,\ell).
\]
The corresponding unstopped arrays \(F^{\mathrm{pre,cap}}_{\xi,d,n}\) and \(F^{\mathrm{post,cap}}_{\xi,d,n,L}\) are obtained by removing the factors \(\one_{\{\tau_n>\ell\}}\).
\end{definition}

\begin{lemma}
\label{lem:stopped-capped-arrays-martingales-paper}
Fix \(
d\in\mathfrak D_{\xi,n}^{\mathrm{osc}}\), and set \(m=m_{\xi,d}\). Let
\[
\Gamma_{\mathrm{pre}}(d)
=
\left\{
(\ell,I,J):
0\le \ell<m,\ 
I\in\mathcal D_\ell^\circ,\ 
J\in\mathcal D_{\ell+1}^\circ,\ 
J\subset I
\right\},
\]
and choose a deterministic ordering \(\Gamma_{\mathrm{pre}}(d)=\{\gamma_1,\ldots,\gamma_N\}\), by increasing generation, with any deterministic order within each generation, where \(\gamma_r=(\ell_r,I_r,J_r)\). With respect to the filtration \((\mathcal H_r^{\mathrm{pre}})\) generated by revealing the variables \(U_{J_r}\) in this order,
\[
S_r^{\mathrm{pre}}
:=
\sum_{q=1}^r
\one_{\{\tau_n>\ell_q\}}
X_{J_q,n}^{\mathrm{cap}}(\xi,d,\ell_q),
\qquad
0\le r\le N,
\]
is a complex-valued martingale and \(S_N^{\mathrm{pre}}=\widehat F^{\mathrm{pre,cap}}_{\xi,d,n}\).

Moreover, for every \(L>m\), an analogous deterministic ordering of
\[
\Gamma_{\mathrm{post}}(d,L)
=
\left\{
(\ell,I,J):
m\le \ell<L,\ 
I\in\mathcal D_\ell^\circ,\ 
J\in\mathcal D_{\ell+1}^\circ,\ 
J\subset I
\right\}
\]
produces a complex-valued martingale whose terminal value is \(\widehat F^{\mathrm{post,cap}}_{\xi,d,n,L}\).
\end{lemma}

\begin{proof}
We prove the pre-bin assertion; the post-bin assertion follows in the same way. Set
\[
\Delta_r^{\mathrm{pre}}
:=
\one_{\{\tau_n>\ell_r\}}
X_{J_r,n}^{\mathrm{cap}}(\xi,d,\ell_r).
\]
It is enough to prove \(
\E(\Delta_r^{\mathrm{pre}}\mid\mathcal H_{r-1}^{\mathrm{pre}})=0\).
By definition,
\[
\Delta_r^{\mathrm{pre}}
=
\one_{\{\tau_n>\ell_r\}}
c_{J_r}(\xi,d,\ell_r)M_{I_r}
\left(
U_{J_r}\wedge C_{I_r,n}
-
\overline U_{I_r,n}
\right).
\]
At the stage when \(U_{J_r}\) is revealed, the factors
\[
\one_{\{\tau_n>\ell_r\}},
\quad
c_{J_r}(\xi,d,\ell_r),
\quad
M_{I_r},
\quad
C_{I_r,n},
\quad
\overline U_{I_r,n}
\]
are \(\mathcal H_{r-1}^{\mathrm{pre}}\)-measurable.  Since \(U_{J_r}\) is independent of the refined past and \(C_{I_r,n}\) is already measurable with respect to that past,
\[
\E\left(
U_{J_r}\wedge C_{I_r,n}
\mid
\mathcal H_{r-1}^{\mathrm{pre}}
\right)
=
\overline U_{I_r,n}.
\]
Thus \(
\E(\Delta_r^{\mathrm{pre}}\mid\mathcal H_{r-1}^{\mathrm{pre}})=0\), and \(S_r^{\mathrm{pre}}\) is a complex-valued martingale. Its terminal value is
\[
S_N^{\mathrm{pre}}
=
\sum_{\ell=0}^{m-1}
\one_{\{\tau_n>\ell\}}
\sum_{I\in\mathcal D_\ell^\circ}
\sum_{\substack{J\in\mathcal D_{\ell+1}^\circ\\ J\subset I}}
X_{J,n}^{\mathrm{cap}}(\xi,d,\ell)
=
\widehat F^{\mathrm{pre,cap}}_{\xi,d,n}.
\]
The same conditional-expectation computation applied to the post-bin ordering proves the second assertion.
\end{proof}

To apply Lemma~\ref{lem:complex-freedman-paper}, we use the refined filtration constructed in Lemma~\ref{lem:stopped-capped-arrays-martingales-paper}.  Each martingale difference has the form
\[
\Delta
=
\one_{\{\tau_n>\ell\}}
X_{J,n}^{\mathrm{cap}}(\xi,d,\ell).
\]
For fixed \(n,\xi,d\), we introduce a jump budget \(R_{\xi,d,n}\) for \(\lvert\Delta\rvert\) and a quadratic-variation budget \(V_{\xi,d,n}\) for the corresponding predictable quadratic variation.  By Lemma~\ref{lem:centered-capped-basic-bounds-paper}, the predictable quadratic variation is controlled by
\[
CL_n
\sum \lvert c_J(\xi,d,\ell)\rvert^2M_IT_{\ell+1,n}.
\]

\begin{lemma}
\label{lem:jump-quadratic-variation-budgets-paper}
There exist constants \(\eta_R,\eta_V>0\) such that, for every \(n\), every \(2^n\le\lvert\xi\rvert\le2^{n+1}\), and every \(
d\in\mathfrak D_{\xi,n}^{\mathrm{osc}}\), the budgets can be chosen so that
\begin{equation}\label{eq:jump-budget-paper}
R_{\xi,d,n}
\le
C2^{-sn/2}2^{-\eta_R n},
\end{equation}
and
\begin{equation}\label{eq:quadratic-variation-budget-paper}
V_{\xi,d,n}
\le
C2^{-sn}2^{-\eta_V n}.
\end{equation}
\end{lemma}

\begin{proof}
By \eqref{eq:capped-jump-bound-paper},
\[
\lvert X_{J,n}^{\mathrm{cap}}\rvert\le C\lvert c_J\rvert L_nT_{\ell+1,n}.
\]
If \(0\le\ell<m\), then \(
\lvert c_J\rvert\le C2^{\ell-m}\),  
and since \(a<1\),
\[
\lvert c_J\rvert T_{\ell+1,n}
\le
C2^{\ell-m}2^{\varepsilon n}2^{-a\ell}
\le
C2^{\varepsilon n}2^{-am}.
\]
If \(\ell\ge m\), then the bound \(\lvert c_J\rvert\le C\) yields the same estimate.  Therefore
\[
\lvert X_{J,n}^{\mathrm{cap}}(\xi,d,\ell)\rvert
\le
CL_n2^{\varepsilon n}2^{-am}.
\]
Using \(L_n=2^{\kappa n}\), \(
2^{-m}\asymp(\lvert\xi\rvert d)^{-1}\), \(
d^{-a}<2^{sn/2}2^{8\varepsilon n}\), \(\lvert\xi\rvert\ge2^n\), and \(a=s+\delta_1\), we get
\[
\lvert X_{J,n}^{\mathrm{cap}}(\xi,d,\ell)\rvert
\le
C2^{-sn/2}2^{-(\delta_1-\kappa-9\varepsilon)n}.
\]
By the parameter choice, \(\delta_1-\kappa-9\varepsilon>0\), which proves \eqref{eq:jump-budget-paper}.

For the quadratic variation, Lemmas~\ref{lem:centered-capped-basic-bounds-paper} and~\ref{lem:prefix-postbin-variance-budgets-paper} give, uniformly for the prefix and all finite post-bin cutoffs,
\[
V_{\xi,d,n}
\le
CL_n2^{2\varepsilon n}\lvert\xi\rvert^{-a}
\le
C2^{-sn}2^{-(\delta_1-\kappa-2\varepsilon)n}.
\]
Since \(\delta_1-\kappa-2\varepsilon>0\) by the parameter choice, this proves \eqref{eq:quadratic-variation-budget-paper}.
\end{proof}

\begin{proposition} 
\label{prop:capped-martingale-estimate-paper}
There exist constants \(C<\infty\), \(c>0\), and \(\eta>0\) such that, for every \(n\ge1\), every \(2^n\le\lvert\xi\rvert\le2^{n+1}\), and every \(d\in\mathfrak D_{\xi,n}^{\mathrm{osc}}\),
\begin{equation}\label{eq:capped-prefix-estimate-paper}
\Pbb\left(
\left|
\widehat F^{\mathrm{pre,cap}}_{\xi,d,n}
\right|
>
2^{-sn/2}2^{-3\varepsilon n}
\right)
\le
C\exp(-c2^{\eta n}),
\end{equation}
and
\begin{equation}\label{eq:capped-postbin-estimate-paper}
\Pbb\left(
\sup_{L>m_{\xi,d}}
\left|
\widehat F^{\mathrm{post,cap}}_{\xi,d,n,L}
\right|
>
2^{-sn/2}2^{-3\varepsilon n}
\right)
\le
C\exp(-c2^{\eta n}).
\end{equation}
Consequently, on \(\mathcal G_n^{\mathrm{pre}}\), the corresponding unstopped centered capped arrays satisfy the same estimates.
\end{proposition}

\begin{proof}
We prove the prefix estimate; the post-bin estimate is obtained in the same way.  Set
\[
t_n=2^{-sn/2}2^{-3\varepsilon n}.
\]
By Lemmas~\ref{lem:stopped-capped-arrays-martingales-paper} and~\ref{lem:jump-quadratic-variation-budgets-paper}, Lemma~\ref{lem:complex-freedman-paper} applies with
\[
R\le C2^{-sn/2}2^{-(\delta_1-\kappa-9\varepsilon)n},
\qquad
V\le C2^{-sn}2^{-(\delta_1-\kappa-2\varepsilon)n}.
\]
Since \(t_n^2=2^{-sn}2^{-6\varepsilon n}\),
\[
V
\le
Ct_n^2\,2^{-(\delta_1-\kappa-8\varepsilon)n},
\qquad
Rt_n
\le
Ct_n^2\,2^{-(\delta_1-\kappa-12\varepsilon)n}.
\]
By the parameter choice, the exponents are positive. Hence, for some \(\eta>0\),
\[
V+Rt_n
\le
Ct_n^2\,2^{-\eta n},
\qquad
\frac{t_n^2}{V+Rt_n}\ge c2^{\eta n}.
\]
Lemma~\ref{lem:complex-freedman-paper} then yields \eqref{eq:capped-prefix-estimate-paper}.  Applying the same argument to the stopped post-bin martingales implies \eqref{eq:capped-postbin-estimate-paper}.  Finally, on \(\mathcal G_n^{\mathrm{pre}}\), Lemma~\ref{lem:stopping-time-localization-paper} implies \(\tau_n=\infty\), and hence the stopped and unstopped centered capped arrays agree.
\end{proof}

\begin{corollary}
\label{cor:one-band-centered-capped-estimate-paper}
There exist constants \(C<\infty\), \(c>0\), and \(\eta>0\) such that, for every \(n\ge1\), every \(2^n\le\lvert\xi\rvert\le2^{n+1}\), and every \(d\in\mathfrak D_{\xi,n}^{\mathrm{osc}}\), 
\[\Pbb\left(
\mathcal G_n^{\mathrm{pre}}
\cap
\left\{
\lvert F^{\mathrm{pre,cap}}_{\xi,d,n}\rvert
+
\sup_{L>m_{\xi,d}}
\lvert F^{\mathrm{post,cap}}_{\xi,d,n,L}\rvert
>
2^{-sn/2}2^{-2\varepsilon n}
\right\}
\right)
\le
C\exp(-c2^{\eta n}).\]
\end{corollary}

\begin{proof}
On \(\mathcal G_n^{\mathrm{pre}}\), the stopped and unstopped arrays agree.  Proposition~\ref{prop:capped-martingale-estimate-paper} bounds both the prefix and post-bin bad events by \(C\exp(-c2^{\eta n})\).  The union bound, together with \(2\cdot2^{-3\varepsilon n}\le2^{-2\varepsilon n}\) for all sufficiently large \(n\), implies the claim after adjusting \(C\).
\end{proof}

The preceding corollary controls only the centered capped martingale contribution. On \(\mathcal G_n^{\mathrm{pre}}\),
Lemma~\ref{lem:raw-capped-compensator-identity-paper} implies the identities
\[
F^{\mathrm{pre}}_{\xi,d}
=
F^{\mathrm{pre,cap}}_{\xi,d,n}
+
D^{\mathrm{pre}}_{\xi,d,n},
\qquad
F^{\mathrm{post}}_{\xi,d,L}
=
F^{\mathrm{post,cap}}_{\xi,d,n,L}
+
D^{\mathrm{post}}_{\xi,d,n,L}.
\]
The compensators are sums of
\[
D_{J,n}
=
-c_JM_I
\E\left[
(U_J-C_{I,n})_+
\mid \mathcal F_\ell
\right],
\]
which are estimated in the next subsection.

\subsection{The \texorpdfstring{\(r\)}{r}-tail compensator}
\label{subsec:r-tail-compensator-paper}

We estimate the predictable drift arising from capping. Recall from
\eqref{eq:compensator-one-step-paper} that the one-step compensator is
\[
D_{J,n}
=
-c_JM_I
\E[(U_J-C_{I,n})_+\mid\mathcal F_\ell],
\]
where \(J\subset I\), \(I\in\mathcal D_\ell^\circ\), and
\[
C_{I,n}=L_n\frac{2T_{\ell+1,n}}{M_I}
\]
when \(M_I>0\).

\begin{lemma}
\label{lem:pointwise-r-tail-compensator-paper}
For every child \(J\subset I\), \(I\in\mathcal D_\ell^\circ\),
\begin{equation}\label{eq:pointwise-r-tail-compensator-paper}
\lvert D_{J,n}\rvert
\le
C\lvert c_J\rvert L_n^{1-r}T_{\ell+1,n}^{1-r}M_I^r\,\E[U^r],
\end{equation}
where \(C<\infty\) depends only on \(r\).
\end{lemma}

\begin{proof}
If \(M_I=0\), then \(D_{J,n}=0\).  It remains to consider the case \(M_I>0\).  For every \(K>0\),
\[
(U-K)_+\le U\one_{\{U>K\}}\le K^{1-r}U^r.
\]
Since \(C_{I,n}\) is \(\mathcal F_\ell\)-measurable and \(U_J\) is independent of \(\mathcal F_\ell\),
\[
\E[(U_J-C_{I,n})_+\mid\mathcal F_\ell]
\le
C_{I,n}^{1-r}\E[U^r].
\]
Hence
\[
\lvert D_{J,n}\rvert
\le
\lvert c_J\rvert M_IC_{I,n}^{1-r}\E[U^r].
\]
Substituting \(C_{I,n}=2L_nT_{\ell+1,n}/M_I\), we obtain
\[
M_IC_{I,n}^{1-r}
=
2^{1-r}L_n^{1-r}T_{\ell+1,n}^{1-r}M_I^r,
\]
which proves \eqref{eq:pointwise-r-tail-compensator-paper}.
\end{proof}

For \(\ell\ge0\), set 
\[R_{\ell,n}
=
L_n^{1-r}T_{\ell+1,n}^{1-r}\E[U^r]
\sum_{I\in\mathcal D_\ell^\circ}M_I^r.\]

\begin{lemma}
\label{lem:level-compensator-envelope-expectation-paper}
There exists \(C<\infty\) such that, for every \(\ell\ge0\) and \(n\ge1\),
\begin{equation}\label{eq:level-compensator-envelope-expectation-paper}
\E[R_{\ell,n}]
\le
C2^{-(\varepsilon+\kappa)(r-1)n}
2^{-\beta_{\mathrm{comp}}\ell}.
\end{equation}
\end{lemma}

\begin{proof}
We use
\[
L_n^{1-r}=2^{-\kappa(r-1)n},
\qquad
T_{\ell+1,n}^{1-r}
\le
C2^{-\varepsilon(r-1)n}2^{a(r-1)\ell},
\]
together with \eqref{eq:prelimit-r-mass-budget-paper}:
\[
\E\left[\sum_{I\in\mathcal D_\ell^\circ}M_I^r\right]
\le
C2^{-r(s+\delta)\ell}.
\]
It follows that
\[
\E[R_{\ell,n}]
\le
C2^{-(\varepsilon+\kappa)(r-1)n}
2^{(r-1)a\ell}
2^{-r(s+\delta)\ell}.
\]
Since \(a=s+\delta_1\) and \(r(s+\delta)-(r-1)(s+\delta_1)=\beta_{\mathrm{comp}}\), this is \eqref{eq:level-compensator-envelope-expectation-paper}.
\end{proof}

Let
\[
m_n^-=\max\{0, \lfloor\vartheta n-C_0\rfloor\},
\qquad
m_n^+=\lceil n+C_0\rceil,
\]
where \(C_0\) is the constant in Lemma~\ref{lem:nonmass-phase-bin-lower-bound-paper}.  Then every non-mass band in \(2^n\le\lvert\xi\rvert\le2^{n+1}\) satisfies \(m_{\xi,d}\in[m_n^-,m_n^+]\).

\begin{definition}
\label{def:absolute-compensators-paper}
Fix \(d\in\mathfrak D_{\xi,n}^{\mathrm{osc}}\), and write \(m=m_{\xi,d}\). Define
\[
\mathfrak C^{\mathrm{pre}}_{\xi,d,n}
=
\sum_{\ell=0}^{m-1}
\sum_{I\in\mathcal D_\ell^\circ}
\sum_{\substack{J\in\mathcal D_{\ell+1}^\circ\\J\subset I}}
\lvert D_{J,n}\rvert,
\]
and
\[
\mathfrak C^{\mathrm{post}}_{\xi,d,n,L}
=
\sum_{\ell=m}^{L-1}
\sum_{I\in\mathcal D_\ell^\circ}
\sum_{\substack{J\in\mathcal D_{\ell+1}^\circ\\J\subset I}}
\lvert D_{J,n}\rvert.
\]
Set
\[
\mathfrak C^{\mathrm{pre}}_{\xi,n}
=
\sum_{d\in\mathfrak D_{\xi,n}^{\mathrm{osc}}}
\mathfrak C^{\mathrm{pre}}_{\xi,d,n},
\qquad
\mathfrak C^{\mathrm{post}}_{\xi,n,L}
=
\sum_{d\in\mathfrak D_{\xi,n}^{\mathrm{osc}}}
\mathfrak C^{\mathrm{post}}_{\xi,d,n,L}.
\]
When the limit exists, set \(\mathfrak C^{\mathrm{post}}_{\xi,n}=\lim_{L\to\infty}\mathfrak C^{\mathrm{post}}_{\xi,n,L}\).
\end{definition}

\begin{lemma}
\label{lem:prefix-compensator-envelope-paper}
There exists a nonnegative random variable \(E^{\mathrm{pre}}_n\), not depending on \(\xi\), such that
\[
\sup_{2^n\le\lvert\xi\rvert\le2^{n+1}}\mathfrak C^{\mathrm{pre}}_{\xi,n}
\le
E^{\mathrm{pre}}_n,
\]
and, for some \(c>0\), 
\[\E[E^{\mathrm{pre}}_n]
\le
C2^{-sn/2}2^{-cn}.\]

\end{lemma}

\begin{proof}
For a non-mass band \(d\), write \(m=m_{\xi,d}\).  By Lemmas~\ref{lem:prefix-postbin-coefficient-estimates-paper} and~\ref{lem:pointwise-r-tail-compensator-paper},
\[
\mathfrak C^{\mathrm{pre}}_{\xi,d,n}
\le
C\sum_{\ell=0}^{m-1}2^{\ell-m}R_{\ell,n}.
\]
For fixed \(\xi\), only boundedly many dyadic bands \(d\) have the same value of \(m_{\xi,d}\).  Hence
\[
\mathfrak C^{\mathrm{pre}}_{\xi,n}
\le
C\sum_{m=m_n^-}^{m_n^+}
\sum_{\ell=0}^{m-1}2^{\ell-m}R_{\ell,n}.
\]
Let the right-hand side be \(E^{\mathrm{pre}}_n\).  Taking expectations and using Lemma~\ref{lem:level-compensator-envelope-expectation-paper},
\[
\E[E^{\mathrm{pre}}_n]
\le
C2^{-(\varepsilon+\kappa)(r-1)n}
\sum_{m=m_n^-}^{m_n^+}
\sum_{\ell=0}^{m-1}2^{\ell-m}2^{-\beta_{\mathrm{comp}}\ell}.
\]
For fixed \(m\),
\[
\sum_{\ell=0}^{m-1}2^{\ell-m}2^{-\beta_{\mathrm{comp}}\ell}
\le
C(1+m)2^{-\gamma_{\mathrm{comp}}m},
\qquad
\gamma_{\mathrm{comp}}=\min\{\beta_{\mathrm{comp}},1\}.
\]
Thus
\[
\E[E^{\mathrm{pre}}_n]
\le
C2^{-(\varepsilon+\kappa)(r-1)n}
\sum_{m=m_n^-}^{m_n^+}(1+m)2^{-\gamma_{\mathrm{comp}}m}.
\]
Since there are \(O(n)\) terms and \(m_n^-\ge\vartheta n-C_0-1\),
\[
\E[E^{\mathrm{pre}}_n]
\le
Cn^2
2^{-(\varepsilon+\kappa)(r-1)n}
2^{-\gamma_{\mathrm{comp}}\vartheta n}.
\]
By Lemma~\ref{lem:compatible-parameter-choice-paper}, \(\gamma_{\mathrm{comp}}\vartheta>s/2\), so the last expression is bounded by \(C2^{-sn/2}2^{-cn}\) after absorbing the polynomial factor.
\end{proof}

\begin{lemma}
\label{lem:postbin-compensator-envelope-paper}
There exists a nonnegative random variable \(E^{\mathrm{post}}_n\), not depending on \(\xi\), such that, for every finite truncation \(L\),
\[
\sup_{2^n\le\lvert\xi\rvert\le2^{n+1}}
\mathfrak C^{\mathrm{post}}_{\xi,n,L}
\le
E^{\mathrm{post}}_n,
\]
and, for some \(c>0\),
\begin{equation}\label{eq:postbin-compensator-envelope-expectation-paper}
\E[E^{\mathrm{post}}_n]
\le
C2^{-sn/2}2^{-cn}.
\end{equation}
\end{lemma}

\begin{proof}
For \(\ell\ge m_{\xi,d}\), we have \(\lvert c_J\rvert\le C\).  Hence
Lemma~\ref{lem:pointwise-r-tail-compensator-paper} implies
\[
\mathfrak C^{\mathrm{post}}_{\xi,n,L}
\le
C\sum_{m=m_n^-}^{m_n^+}\sum_{\ell=m}^{\infty}R_{\ell,n}.
\]
Let the right-hand side be \(E^{\mathrm{post}}_n\). Taking expectations and using Lemma~\ref{lem:level-compensator-envelope-expectation-paper},
\[
\E[E^{\mathrm{post}}_n]
\le
C2^{-(\varepsilon+\kappa)(r-1)n}
\sum_{m=m_n^-}^{m_n^+}\sum_{\ell=m}^{\infty}2^{-\beta_{\mathrm{comp}}\ell}.
\]
Since \(\beta_{\mathrm{comp}}>0\),
\[
\sum_{\ell=m}^{\infty}2^{-\beta_{\mathrm{comp}}\ell}
\le
C2^{-\beta_{\mathrm{comp}}m}.
\]
Thus
\[
\E[E^{\mathrm{post}}_n]
\le
Cn2^{-(\varepsilon+\kappa)(r-1)n}
2^{-\beta_{\mathrm{comp}}m_n^-}.
\]
By Lemma~\ref{lem:compatible-parameter-choice-paper}, \(\beta_{\mathrm{comp}}\vartheta>s/2\), 
and since \(m_n^-\ge\vartheta n-C_0-1\), \eqref{eq:postbin-compensator-envelope-expectation-paper}
follows after reducing \(c>0\).
\end{proof}

\begin{proposition}[\(r\)-tail compensator estimate]
\label{prop:r-tail-compensator-paper}
There exist constants \(C<\infty\) and \(c>0\) such that, for every \(n\ge1\),
\begin{equation}\label{eq:r-tail-compensator-probability-paper}
\Pbb\left(
\sup_{2^n\le\lvert\xi\rvert\le2^{n+1}}
\left(
\mathfrak C^{\mathrm{pre}}_{\xi,n}
+
\sup_{L\ge1}\mathfrak C^{\mathrm{post}}_{\xi,n,L}
\right)
>
2n^{-2}2^{-sn/2}
\right)
\le
C2^{-cn}.
\end{equation}
Consequently, the probabilities in \eqref{eq:r-tail-compensator-probability-paper} are summable in 
\(n\).
\end{proposition}

\begin{proof}
By Lemmas~\ref{lem:prefix-compensator-envelope-paper} and~\ref{lem:postbin-compensator-envelope-paper},
\[
\sup_{2^n\le\lvert\xi\rvert\le2^{n+1}}
\left(
\mathfrak C^{\mathrm{pre}}_{\xi,n}
+
\sup_{L\ge1}\mathfrak C^{\mathrm{post}}_{\xi,n,L}
\right)
\le
E^{\mathrm{pre}}_n+E^{\mathrm{post}}_n,
\]
and
\[
\E[E^{\mathrm{pre}}_n+E^{\mathrm{post}}_n]
\le
C2^{-sn/2}2^{-cn}.
\]
Markov's inequality therefore yields 
\[
\Pbb\left(
E^{\mathrm{pre}}_n+E^{\mathrm{post}}_n
>
2n^{-2}2^{-sn/2}
\right)
\le
Cn^2 2^{-cn}
\le
C2^{-cn}
\]
after reducing \(c\).  Summability follows.
\end{proof}

\begin{lemma}
\label{lem:raw-arrays-controlled-compensators-paper}
Assume \(\mathcal G_n^{\mathrm{pre}}\) holds. For 
\(d\in\mathfrak D_{\xi,n}^{\mathrm{osc}},\)
define
\[
D^{\mathrm{pre}}_{\xi,d,n}
=
\sum_{\ell=0}^{m_{\xi,d}-1}
\sum_{I\in\mathcal D_\ell^\circ}
\sum_{\substack{J\in\mathcal D_{\ell+1}^\circ\\J\subset I}}
D_{J,n},
\]
and
\[
D^{\mathrm{post}}_{\xi,d,n,L}
=
\sum_{\ell=m_{\xi,d}}^{L-1}
\sum_{I\in\mathcal D_\ell^\circ}
\sum_{\substack{J\in\mathcal D_{\ell+1}^\circ\\J\subset I}}
D_{J,n}.
\]
Then
\[
F^{\mathrm{pre}}_{\xi,d}
=
F^{\mathrm{pre,cap}}_{\xi,d,n}
+
D^{\mathrm{pre}}_{\xi,d,n},
\]
and, for every \(L>m_{\xi,d}\),
\[
F^{\mathrm{post}}_{\xi,d,L}
=
F^{\mathrm{post,cap}}_{\xi,d,n,L}
+
D^{\mathrm{post}}_{\xi,d,n,L}.
\]
Moreover,
\[
\lvert D^{\mathrm{pre}}_{\xi,d,n}\rvert
\le
\mathfrak C^{\mathrm{pre}}_{\xi,d,n},
\qquad
\lvert D^{\mathrm{post}}_{\xi,d,n,L}\rvert
\le
\mathfrak C^{\mathrm{post}}_{\xi,d,n,L}.
\]
If the absolute post-bin compensator is finite, then the signed post-bin compensator converges as 
\(L\to\infty,\) 
and the same identity holds for the limiting post-bin arrays.
\end{lemma}

\begin{proof}
On \(\mathcal G_n^{\mathrm{pre}}\), Lemma~\ref{lem:raw-capped-compensator-identity-paper} yields the identity
\[
c_JM_I(U_J-1)=X_{J,n}^{\mathrm{cap}}+D_{J,n}.
\]
Summing over \(0\le\ell<m_{\xi,d}\) proves the prefix identity, while summing over
\(m_{\xi,d}\le\ell<L\) proves the finite post-bin identity. The absolute-value bounds are immediate from the definitions of \(\mathfrak C^{\mathrm{pre}}\) and \(
\mathfrak C^{\mathrm{post}}\). 
If the absolute post-bin compensator is finite, the signed compensator converges absolutely, so passing to the limit as \(L\to\infty\) proves the limiting identity.
\end{proof}

\subsection{Annular assembly}
\label{subsec:annular-assembly-paper}

We now combine the local-mass good event, the deterministic reduction, the capped martingale estimate, and the compensator estimate.

\begin{lemma}
\label{lem:annular-frequency-grid-paper}
For every \(n\ge1\), there exists a finite set \(\mathcal N_n\subset\{\xi\in\R^2:2^n\le\lvert\xi\rvert\le2^{n+1}\}\) such that for every \(\xi\) with \(2^n\le\lvert\xi\rvert\le2^{n+1}\), there exists \(\eta\in\mathcal N_n\) satisfying
\begin{equation}\label{eq:annular-frequency-grid-mesh-paper}
\lvert\xi-\eta\rvert\le \frac{1}{8\pi}2^{-(s/2+\varepsilon)n}.
\end{equation}
Moreover,
\begin{equation}\label{eq:annular-frequency-grid-cardinality-paper}
\#\mathcal N_n\le C2^{(2+s+2\varepsilon)n}.
\end{equation}
\end{lemma}

\begin{proof}
Let \(\rho_n=(8\pi)^{-1}2^{-(s/2+\varepsilon)n}\). Take a square lattice with mesh comparable to \(\rho_n\), and retain the points needed to cover the annulus. Since the annulus is contained in a square of side length \(O(2^n)\),
\[
\#\mathcal N_n
\le
C\left(\frac{2^n}{\rho_n}\right)^2
\le
C2^{(2+s+2\varepsilon)n}.
\]
\end{proof}

\begin{lemma}
\label{lem:annular-grid-passage-paper}
Assume that \(\mathcal G_n^{\mathrm{lim}}\) holds. If
\[
\max_{\eta\in\mathcal N_n}\lvert\wh{\nu_\circ}(\eta)\rvert
\le
A2^{-sn/2},
\]
then
\[
\sup_{2^n\le\lvert\xi\rvert\le2^{n+1}}\lvert\wh{\nu_\circ}(\xi)\rvert
\le
(A+1)2^{-sn/2}.
\]
\end{lemma}

\begin{proof}
For \(2^n\le\lvert\xi\rvert\le2^{n+1}\), choose 
\(\eta\in\mathcal N_n\) 
satisfying \eqref{eq:annular-frequency-grid-mesh-paper}. By \eqref{eq:good-event-fourier-lipschitz-paper},
\[
\lvert\wh{\nu_\circ}(\xi)-\wh{\nu_\circ}(\eta)\rvert
\le
2\pi2^{\varepsilon n}\lvert\xi-\eta\rvert
\le
\frac14\,2^{-sn/2}.
\]
Thus \(\lvert\wh{\nu_\circ}(\xi)\rvert\le (A+1)2^{-sn/2}\).
\end{proof}

\begin{definition}
\label{def:capped-compensator-exceptional-events-paper}
Define the capped-martingale exceptional event 
\(\mathcal E_n^{\mathrm{cap}}\) 
to be the event that there exist
\[
\xi\in\mathcal N_n
\qquad\text{and}\qquad
d\in\mathfrak D_{\xi,n}^{\mathrm{osc}}
\]
such that 
\(\mathcal G_n^{\mathrm{pre}}\) 
holds and
\[
\lvert F^{\mathrm{pre,cap}}_{\xi,d,n}\rvert
+
\sup_{L>m_{\xi,d}}
\lvert F^{\mathrm{post,cap}}_{\xi,d,n,L}\rvert
>
2^{-sn/2}2^{-2\varepsilon n}.
\]
Define the compensator exceptional event 
\(\mathcal E_n^{\mathrm{comp}}\) 
by
\[
\mathcal E_n^{\mathrm{comp}}
=
\left\{
\sup_{2^n\le\lvert\xi\rvert\le2^{n+1}}
\left(
\mathfrak C^{\mathrm{pre}}_{\xi,n}
+
\sup_{L\ge1}\mathfrak C^{\mathrm{post}}_{\xi,n,L}
\right)
>
2n^{-2}2^{-sn/2}
\right\}.
\]
\end{definition}

\begin{lemma}
\label{lem:capped-exceptional-probability-paper}
There exist constants \(C,c,\eta>0\) such that
\[
\Pbb(\mathcal E_n^{\mathrm{cap}})
\le
C\exp(-c2^{\eta n}).
\]
\end{lemma}

\begin{proof}
For fixed \(\xi\) and \(d\), Corollary~\ref{cor:one-band-centered-capped-estimate-paper} yields the bound \(C\exp(-c2^{\eta n})\). For each \(\xi\), the number of derivative bands is \(O(n)\), and by \eqref{eq:annular-frequency-grid-cardinality-paper},
\[
\#\mathcal N_n\le C2^{(2+s+2\varepsilon)n}.
\]
The union bound therefore yields
\[
\Pbb(\mathcal E_n^{\mathrm{cap}})
\le
Cn2^{(2+s+2\varepsilon)n}\exp(-c2^{\eta n}).
\]
The stretched-exponential factor dominates the ordinary exponential prefactor, and the claim follows after decreasing \(c\) and \(\eta\), if necessary.
\end{proof}

\begin{lemma}
\label{lem:grid-annular-estimate-paper}
There exist constants \(C,c,\eta>0\) such that, for every \(n\ge1\),
\begin{equation}\label{eq:grid-annular-estimate-paper}
\Pbb\left(
\max_{\xi\in\mathcal N_n}
\lvert\wh{\nu_\circ}(\xi)\rvert
>
C2^{-sn/2}
\right)
\le
C\exp(-c2^{\eta n})+C2^{-cn}.
\end{equation}
\end{lemma}

\begin{proof}
By Proposition~\ref{prop:local-mass-good-events-paper}, Lemma~\ref{lem:capped-exceptional-probability-paper}, and Proposition~\ref{prop:r-tail-compensator-paper},
\[
\Pbb(\mathcal G_n^c)\le C2^{-cn},
\qquad
\Pbb(\mathcal E_n^{\mathrm{cap}})\le C\exp(-c2^{\eta n}),
\qquad
\Pbb(\mathcal E_n^{\mathrm{comp}})\le C2^{-cn}.
\]
It remains to prove that, on
\[
\mathcal G_n\cap(\mathcal E_n^{\mathrm{cap}})^c\cap(\mathcal E_n^{\mathrm{comp}})^c,
\]
one has 
\(
\max_{\xi\in\mathcal N_n}\lvert\wh{\nu_\circ}(\xi)\rvert\le C2^{-sn/2}.
\)

Fix \(\xi\in\mathcal N_n\). By Proposition~\ref{prop:stationary-tube-phase-bin-reduction-paper},
\[
\wh{\nu_\circ}(\xi)
=
E^{\mathrm{safe}}_{\xi,n}
+
\sum_{d\in\mathfrak D_{\xi,n}^{\mathrm{osc}}}
\left(
F^{\mathrm{pre}}_{\xi,d}
+
F^{\mathrm{post}}_{\xi,d}
\right),
\]
with 
\(
\lvert E^{\mathrm{safe}}_{\xi,n}\rvert\le C2^{-sn/2}2^{-c_{\mathrm{safe}}n}.
\) 
On \(\mathcal G_n^{\mathrm{pre}}\), Lemma~\ref{lem:raw-arrays-controlled-compensators-paper} decomposes each raw array into a centered capped part and a compensator. Since 
\(
(\mathcal E_n^{\mathrm{cap}})^c
\) 
holds, for every non-mass band \(d\),
\[
\lvert F^{\mathrm{pre,cap}}_{\xi,d,n}\rvert
+
\sup_{L>m_{\xi,d}}\lvert F^{\mathrm{post,cap}}_{\xi,d,n,L}\rvert
\le
2^{-sn/2}2^{-2\varepsilon n}.
\]
The raw post-bin limit exists by Lemma~\ref{lem:exact-martingale-decomposition-band-paper}; on 
\(
(\mathcal E_n^{\mathrm{comp}})^c,
\) 
the post-bin compensator converges absolutely, and hence the limiting centered capped post-bin term exists. Therefore,
\[
\lvert F^{\mathrm{pre,cap}}_{\xi,d,n}\rvert
+
\lvert F^{\mathrm{post,cap}}_{\xi,d,n}\rvert
\le
2^{-sn/2}2^{-2\varepsilon n}.
\]
Since there are \(O(n)\) derivative bands,
\[
\sum_{d\in\mathfrak D_{\xi,n}^{\mathrm{osc}}}
\left(
\lvert F^{\mathrm{pre,cap}}_{\xi,d,n}\rvert
+
\lvert F^{\mathrm{post,cap}}_{\xi,d,n}\rvert
\right)
\le
Cn2^{-2\varepsilon n}2^{-sn/2}
\le
C2^{-sn/2}.
\]
On 
\(
(\mathcal E_n^{\mathrm{comp}})^c,
\)
\[
\sum_{d\in\mathfrak D_{\xi,n}^{\mathrm{osc}}}
\left(
\lvert D^{\mathrm{pre}}_{\xi,d,n}\rvert
+
\lvert D^{\mathrm{post}}_{\xi,d,n}\rvert
\right)
\le
2n^{-2}2^{-sn/2}.
\]
Combining the safe term, the centered capped terms, and the compensators, we obtain
\[
\lvert\wh{\nu_\circ}(\xi)\rvert
\le
C2^{-sn/2}.
\]
Taking the union of the exceptional probabilities proves \eqref{eq:grid-annular-estimate-paper}.
\end{proof}

\begin{proof}[Proof of Theorem~\ref{thm:finite-r-annular}]
By Lemma~\ref{lem:grid-annular-estimate-paper},
\[
\Pbb\left(
\max_{\xi\in\mathcal N_n}
\lvert\wh{\nu_\circ}(\xi)\rvert
>
C2^{-sn/2}
\right)
\le
C\exp(-c2^{\eta n})+C2^{-cn}.
\]
By Proposition~\ref{prop:local-mass-good-events-paper},
\[
\Pbb\bigl((\mathcal G_n^{\mathrm{lim}})^c\bigr)
\le
C2^{-cn}.
\]
On \(\mathcal G_n^{\mathrm{lim}}\), Lemma~\ref{lem:annular-grid-passage-paper} upgrades the grid estimate to the full annulus. Therefore, after increasing \(C\),
\[
\Pbb\left(
\sup_{2^n\le\lvert\xi\rvert\le2^{n+1}}
\lvert\wh{\nu_\circ}(\xi)\rvert
>
C2^{-sn/2}
\right)
\le
C\exp(-c2^{\eta n})+C2^{-cn}.
\]
This is \eqref{eq:finite-r-annular-theorem-paper}. The right-hand side is summable in \(n\). Hence, by the Borel--Cantelli lemma, almost surely, for all sufficiently large \(n\),
\[
\sup_{2^n\le\lvert\xi\rvert\le2^{n+1}}
\lvert\wh{\nu_\circ}(\xi)\rvert
\le
C2^{-sn/2}.
\]
This completes the proof.
\end{proof}

\section{Endpoint lower bound and circle theorem}
\label{sec:endpoint-lower-bound-circle-theorem}

This section completes the proof of the circle endpoint theorem. Corollary~\ref{cor:endpoint-upper-bound-paper} yields
\[
\dimF(\mu_\circ)\le \Aloc(W)
\qquad
\text{almost surely on }\mathcal S_\circ.
\]
The reverse inequality follows from Theorem~\ref{thm:finite-r-annular}: every strict subendpoint exponent \(s<\Aloc(W)\) is witnessed by a finite moment of some order \(r>1\), and hence yields the corresponding Fourier decay.

Throughout this section, \(W\) is assumed to be in the minimal Kahane--Peyri\`ere regime:
\[
W\ge0,
\qquad
\E W=1,
\qquad
\E[W\log_2^+W]<\infty,
\qquad
\E[W\log_2W]<1.
\]
Let \(\mu_\circ\) be the circle cascade generated by \(W\), and set
\[
\mathcal S_\circ=\{\mu_{\circ}(\mathbb S^1)>0\}.
\]

\begin{theorem}[Endpoint Fourier dimension formula on the circle]
\label{thm:endpoint-circle-equality}
Assume the minimal Kahane--Peyri\`ere regime. Then, almost surely on \(\mathcal S_\circ\),
\begin{equation}\label{eq:endpoint-circle-equality-paper}
\dimF(\mu_\circ)=\Aloc(W).
\end{equation}
\end{theorem}

\begin{proof}
We prove the lower bound. If \(\Aloc(W)=0\), there is nothing to prove. Assume \(\Aloc(W)>0\), and fix \(0<\sigma<\Aloc(W)\). By the definition of \(\Aloc(W)\), there exists \(r>1\) such that
\[
\E[W^r]<\infty,
\qquad
\frac{r-1-\log_2\E[W^r]}{r}>\sigma.
\]
Choose \(\delta>0\) such that
\[
\sigma+\delta
<
\frac{r-1-\log_2\E[W^r]}{r}.
\]
Equivalently,
\[
1-r+\log_2\E[W^r]<-r(\sigma+\delta),
\]
and therefore 
\[
2^{1-r}\E[W^r]\le 2^{-r(\sigma+\delta)}.
\]

Apply Theorem~\ref{thm:finite-r-annular} with \(U=W\) and \(s=\sigma\). The interval cascade generated by \(U\), pushed forward by \(f(t)=(\cos 2\pi t,\sin 2\pi t)\), is precisely \(\mu_\circ\). Hence Theorem~\ref{thm:finite-r-annular} provides a finite random constant \(C_\sigma(\omega)<\infty\) such that \(
\lvert\wh{\mu_\circ}(\xi)\rvert
\le
C_\sigma(\omega)\lvert\xi\rvert^{-\sigma/2}\) for all sufficiently large \(\lvert\xi\rvert\), almost surely. This is exactly the Fourier decay required for the exponent \(\sigma\); no endpoint estimate at \(\sigma=\Aloc(W)\) is needed. Consequently, \(\dimF(\mu_\circ)\ge \sigma\) almost surely on \(\mathcal S_\circ\).

Applying the preceding argument to every rational \(\sigma\in\mathbb Q\cap(0,\Aloc(W))\) and intersecting the corresponding probability-one events, we obtain \(\dimF(\mu_\circ)\ge \Aloc(W)\) almost surely on \(\mathcal S_\circ\). The countability of the chosen exponents is the only measurability point in the passage to the endpoint. The opposite inequality is Corollary~\ref{cor:endpoint-upper-bound-paper}. Hence Theorem~\ref{thm:finite-r-annular} completes the proof.
\end{proof}

\begin{corollary} 
\label{cor:circle-positive-zero-branches-paper}
Assume the minimal Kahane--Peyri\`ere regime. Then the following holds.

\begin{enumerate}[label=\textup{(\roman*)}]
\item If there exists \(q>1\) such that \(\E[W^q]<2^{q-1}\), then \(
\dimF(\mu_\circ)>0\) almost surely on \(\mathcal S_\circ\).

\item If \(\E[W^q]=\infty\) for every \(q>1\), then \(\dimF(\mu_\circ)=0\) almost surely on \(\mathcal S_\circ\).
\end{enumerate}
\end{corollary}

\begin{proof}
For (i), the assumption ensures that \(q-1-\log_2\E[W^q]>0\) for some \(q>1\), hence \(\Aloc(W)>0\). For (ii), every \(q>1\) term in the definition of \(\Aloc(W)\) is interpreted as \(0\), so \(\Aloc(W)=0\). Both conclusions follow from Theorem~\ref{thm:endpoint-circle-equality}.
\end{proof}

\begin{example}[A Bernoulli zero-weight cascade]
\label{ex:bernoulli-zero-weight-cascade}
Let
\[
W=
\begin{cases}
0, & \text{with probability }1-p,\\
p^{-1}, & \text{with probability }p,
\end{cases}
\qquad
\frac12<p\le1.
\]
Then
\[
\E W=1,
\qquad
\E[W\log_2 W]=\log_2(p^{-1})<1,
\]
so \(W\) is in the minimal Kahane--Peyri\`ere regime. For every \(q>1\),
\[
\E[W^q]=p^{1-q},
\qquad
q-1-\log_2\E[W^q]
=
(q-1)(1+\log_2 p).
\]
Since \(p>1/2\),
\[
\Aloc(W)
=
\sup_{q>1}
\frac{(q-1)(1+\log_2 p)}{q}
=
1+\log_2 p.
\]
By Theorem~\ref{thm:main-circle-endpoint-formula},
\[
\dimF(\mu_\circ)=1+\log_2 p
\qquad
\text{almost surely on }\{\mu_\circ(\mathbb S^1)>0\}.
\]

For comparison, for the scalar interval cascade generated by the same law,
\[
D^+(W)
=
\sup_{1<q<2}
2\,\frac{(q-1)(1+\log_2 p)}{q}
=
1+\log_2 p,
\]
where the last equality follows by letting \(q\uparrow2\). Thus, by the scalar interval theorem,
\[
\dimF(\mu)=\dimE(\mu)=\dimtwo(\mu)=1+\log_2 p
\qquad
\text{almost surely on }\{M>0\}.
\]
In this special example,
\[
D^+(W)=\Aloc(W)=1+\log_2 p.
\]
The coincidence is caused by the linear moment profile and should not be expected in general.
\end{example}

\section*{Notation Index}

\small
\renewcommand{\arraystretch}{1.35}
\setlength{\tabcolsep}{2pt}

\begin{longtable}{@{}
>{\centering\arraybackslash}p{0.17\textwidth}
>{\centering\arraybackslash}p{0.53\textwidth}
>{\centering\arraybackslash}p{0.24\textwidth}
@{}}
\hline
\textbf{Notation} & \textbf{Meaning} & \textbf{Place defined/remarks}\\
\hline\hline
\endfirsthead

\hline
\textbf{Notation} & \textbf{Meaning} & \textbf{Place defined/remarks}\\
\hline\hline
\endhead

\hline
\endfoot

\(\prec, \preceq\) & Prefix relation. & Section~\ref{subsec:binary-tree-dyadic-intervals-circle-arcs}.\\

\(\alphamin(\nu)\) & Minimum lower local dimension. & Section~\ref{subsec:endpoint-local-exponent-paper}.\\

\(\beta_X(q)\) & \(-(1/q)\log_2\rho(q)\). & Equation~\eqref{eq:beta-X-q-paper}.\\

\(\kappa(q)\) & \(2^{1-q}\E[W^q]\). & Equation~\eqref{E:light-tail}.\\

\(\rho(q)\)
& Vector \(q\)-moment profile.
& Definition~\ref{def:vector-energy-parameter}.\\

\(\tau(q)\) & \(q-1-\log_2\E[W^q]\). & Section~\ref{subsec:endpoint-local-exponent-paper}.\\

\(\mu_n,\mu,\mu^{(u)}\)
& Level-\(n\) interval cascade measure, weak limit, and descendant limiting measure below \(u\).
& Equation~\eqref{eq:vector-level-measure-definition};

Theorem~\ref{thm:vector-limiting-measure-construction-paper};

Section~\ref{subsec:probability-conventions-descendants}.\\

\(\mu_{\circ,n},\mu_\circ,\mu_\circ^{(v)}\)
& Level-\(n\) circle cascade measure, weak limit, and descendant circle cascade below \(v\).
& Equation~\eqref{eq:circle-level-measure-paper};

Proposition~\ref{prop:circle-cascade-construction-paper};

Section~\ref{subsec:circle-cascade-construction-paper}.\\

\(\widetilde\nu_\ell,\widetilde\nu,\nu_\circ\)
& Level-\(\ell\) parameter cascade measure on \([0,1)\), its weak limit, and the circle pushforward.
& Section~\ref{sec:finite-r-annular-theorem}.\\

\(\Sigma_n\) & \(\sum_{\lvert u \rvert=n}C_u^2\). & Definition~\ref{def:tree-cylinder-square-sums-paper}.\\

\(\Sigma_n(\nu)\) &
\(\sum_{\lvert u \rvert=n}\nu(I_u)^2\). &
Section~\ref{subsec:fourier-energy-dimension-definitions}.\\

\(\Gamma_n(\alpha)\) & Dense grid in \([1,2]\) with mesh at most \(2^{-\alpha n}\). & Definition~\ref{def:dense-grids-paper}.\\

\(\Aloc(W)\) & Endpoint local exponent. & Definition~\ref{def:circle-endpoint-exponent}.\\

\(C_u\) & Limiting tree-cylinder mass associated with \(u\). & Remark~\ref{rem:tree-cylinder-masses-versus-euclidean-intervals}.\\

\(C_{I,n}\), \(U_{J,n}^{\mathrm{cap}}\), \(\overline U_{I,n}\)
& Predictable cap, capped child weight, and capped mean.
& Definition~\ref{def:predictable-cap-paper}.\\

\(\mathfrak C^{\mathrm{pre}}_{\xi,d,n}\),
\(\mathfrak C^{\mathrm{post}}_{\xi,d,n,L}\),
\(\mathfrak C^{\mathrm{pre}}_{\xi,n}\),
\(\mathfrak C^{\mathrm{post}}_{\xi,n,L}\),
\(\mathfrak C^{\mathrm{post}}_{\xi,n}\)
& One-band and summed prefix/post-bin compensators.
& Definition~\ref{def:absolute-compensators-paper}.\\

\(\mathcal D_n,\mathcal D_{\circ,n}\)
& Level-\(n\) dyadic interval/arc partitions.
& Section~\ref{subsec:binary-tree-dyadic-intervals-circle-arcs}.\\

\(D_E(X)\) & Energy parameter of the interval vector law. & Definition~\ref{def:vector-energy-parameter}.\\

\(\mathfrak D_\xi\), \(\chi_{\xi,d}\) & Derivative-band scales and corresponding cutoffs. & Lemma~\ref{lem:stationary-derivative-band-partition-paper}.\\

\(\mathcal E_n^{\mathrm{cap}},\calE_n^{\mathrm{comp}}\) & Capped-martingale and compensator exceptional events. & Definition~\ref{def:capped-compensator-exceptional-events-paper}.\\

\(\mathcal F_n^X\) & Natural filtration of the interval vector cascade. & Section~\ref{subsec:probability-conventions-descendants}.\\

\(\mathcal F_n^W\), \(\mathcal S_\circ\) & Natural filtration of the circle cascade and non-extinction event \(\{Y>0\}\). & Section~\ref{subsec:circle-cascade-construction-paper}.\\

\(F^{\mathrm{pre}}_{\xi,d}\), \(F^{\mathrm{post}}_{\xi,d,L}\), \(F^{\mathrm{post}}_{\xi,d}\)
& Prefix array, truncated post-bin array, and limiting post-bin array.
& Definition~\ref{def:raw-prefix-postbin-arrays-paper}.\\

\(\mathcal G_n^{\mathrm{pre}},\mathcal G_n^{\mathrm{lim}},\mathcal G_n\) & Prelimit, limiting, and full local-mass good events. & Definition~\ref{def:local-mass-good-events-paper}.\\

\(I_X\) & \(\{q\in(1,2):\rho(q)<1\}\). & Lemma~\ref{lem:rho-convexity-subcritical-interval-paper}.\\

\(I_u,J_u,\mathcal I_u\)
& Dyadic interval/arc associated with \(u\).
& Section~\ref{subsec:binary-tree-dyadic-intervals-circle-arcs}.\\

\(L_u\) & Path weight along \(u\). & Equation~\eqref{eq:vector-path-weight-definition}.\\

\(L_n\) &\(2^{\kappa n}\), cap-growth factor. & Equation~\eqref{eq:cap-growth-factor-paper}.\\

\(m_s(q)\) & \(2^{qs/2}\rho(q)\). & Equation~\eqref{eq:ms-profile}.\\

\(m_{\xi,d}\) & Phase-bin scale determined by \(2^{m_{\xi,d}-1}<\lvert\xi\rvert d\le2^{m_{\xi,d}}\). & Definition~\ref{def:phase-bin-scale-paper}.\\

\(M_n,M,M^{(u)}\)
& Total mass martingale, terminal total mass, and terminal descendant mass below \(u\).
& Equation~\eqref{eq:vector-total-mass-definition};

Lemma~\ref{lem:vector-total-mass-martingale-paper};

Section~\ref{subsec:probability-conventions-descendants}.\\

\(T_{k,n}\) & \(2^{\varepsilon n}2^{-ak}\), local-mass threshold. & Equation~\eqref{eq:local-mass-threshold-paper}.\\

\(\mathcal{T}\) & \(\bigcup_{n\ge0}\{0,1\}^n\), rooted binary tree. & Section~\ref{subsec:binary-tree-dyadic-intervals-circle-arcs}.\\

\(\partial\mathcal T\), \([u]_\partial\), \(\pi\) & Boundary of the binary tree, boundary cylinder, and coding map to \([0,1]\). & Section~\ref{subsec:deterministic-square-sums-energy}.\\

\(Y_n,Y,Y^{(v)}\)
& Circle total mass martingale, terminal total mass, and terminal descendant mass below \(v\).
& Section~\ref{subsec:circle-cascade-construction-paper}.\\

\(Z_n^{(s)}\) & \(2^{sn}\Sigma_n\). & Equation~\eqref{eq:normalized-square-sum-recursion-paper}.\\

\end{longtable}

\section*{Acknowledgments}

G. C. was partially supported by the National Natural Science Foundation of China
(NSFC), grant no. 12371126. X. F. was partially supported by the National Science
and Technology Council, Taiwan, grant no. 114-2115-M-A49-003-MY3.

The authors used an artificial-intelligence tool for language editing and
\LaTeX\ formatting; all mathematical content was written, checked, and approved
by the authors.

\end{document}